\documentclass[french]{smfart}
\usepackage[utf8]{inputenc}
\usepackage[english,francais]{babel}
\usepackage{smfthm}
\usepackage{mathrsfs}
\usepackage{graphics}
\usepackage{hyperref} 
\usepackage{amssymb, amsmath, mathabx}
\usepackage{lmodern}
\usepackage[all]{xy}
\usepackage{multicol}
\usepackage[shortlabels]{enumitem}
\usepackage{scalerel}
\makeindex
\DeclareUnicodeCharacter{00A0}{\relax}

\author{Antoine Ducros}
\address{Sorbonne Université, Université Paris-Diderot, CNRS, Institut de Mathématiques de Jussieu-Paris
Rive Gauche, IMJ-PRG, DMA, PSL,
École normale supérieure,
F-75005, Paris, France
}
\email{antoine.ducros\at imj-prg.fr}
\urladdr{http://www.imj-prg.fr/$\sim$antoine.ducros}
\title{Réduction en famille d'espaces affinoïdes}
\alttitle{Reduction of affinoid spaces in family}
\thanks{}
\NumberTheoremsAs{subsection}
\SwapTheoremNumbers

\newcommand{\eg}{e.\@g.\@}

\newcommand{\ie}{i.\@e.\@}

\newcommand{\abs}[1]{\mathopen|#1\mathclose|}

\newcommand{\an}{^{\mathrm{an}}}

\newcommand{\gpm}{^{\times}}

\newcommand{\hotimes}{\hat \otimes}
\newcommand{\hr}[1]{\mathscr H(#1)}
\newcommand{\hrt}[1]{\widetilde{\mathscr H(#1)}}

\newcommand{\inv}{^{-1}}
\newcommand{\alg}[1]{\widehat {{#1}^{\mathrm a}}}
\newcommand{\norm}[1]{\mathopen\|#1\mathclose \|}

\def\spec{\mathop{\mathrm{Spec}}}
\newcommand{\spv}{\mathrm{Spv}}

\newcommand{\A}{\mathbf A}

\newcommand{\N}{\mathbf Z_{\geq 0}}
\renewcommand{\P}{\mathbf P}
\newcommand{\Q}{\mathbf Q}
\newcommand{\R}{\mathbf R}
\newcommand{\Z}{\mathbf Z}

\renewcommand{\phi}{\varphi}
\renewcommand{\epsilon}{\varepsilon}
\renewcommand{\leq}{\leqslant}
\renewcommand{\geq}{\geqslant}

\SetEnumerateShortLabel{1}{\textnormal{(\arabic{enumi})}}
\SetEnumerateShortLabel{s}{\textnormal{(\arabic{enumi}$^\ast$)}}
\SetEnumerateShortLabel{i}{\textnormal{(\roman{enumi})}}
\SetEnumerateShortLabel{a}{\textnormal{(\alph{enumi})}}
\SetEnumerateShortLabel{A}{\textnormal{(\Alph{enumi})}}
\SetEnumerateShortLabel{2}{\textnormal{(\arabic{enumii})}}
\SetEnumerateShortLabel{j}{\textnormal{(\roman{enumii})}}
\SetEnumerateShortLabel{b}{\textnormal{(\arabic{enumi}\alph{enumii})}}
\SetEnumerateShortLabel{c}{\textnormal{(\alph{enumii})}}
\SetEnumerateShortLabel{B}{\textnormal{(\Alph{enumii})}}

\begin{document}

\thanks{Lors
de la rédaction de cet article, l'auteur a bénéficié du soutien de l'ANR à travers les projets {\em Valuations, combinatoire et théorie des modèles} (ANR-13-BS01-0006), 
et  {\em Définissabilité en géométrie non archimédienne}
(ANR-15-CE40-0008), ainsi que de celui de l'IUF dont il était membre junior d'octobre 2012 à octobre 2017. Il a aussi profité en mars 2019 de l'hospitalité de l'université hébraïque de Jérusalem, avec le soutien du projet ERC Consolidator  770922 (BirNonArchGeom) de Michael Temkin}

\maketitle
\begin{abstract}
Soit $k$ un corps ultramétrique complet. Nous démontrons un substitut au théorème de la fibre réduite (de Bosch, Lütkebohmert et Raynaud) valable pour tout morphisme $Y\to X$ plat et à fibres géométriquement réduites entre espaces $k$-affinoïdes au sens de Berkovich, sans supposer que $X$ et $Y$ sont stricts ni que la dimension relative de $Y$ sur $X$ est constante. Nous ne faisons pas appel au théorème de la fibre réduite original, ni aux techniques ou au langage de la géométrie formelle. Notre énoncé est formulé en termes de réduction graduée à la Temkin ;  notre preuve repose sur un théorème de finitude de Grauert et Remmert et sur la théorie de la réduction (graduée) des germes d'espaces analytiques, due à Temkin. 
\end{abstract}

\begin{altabstract}
Let $k$ be a non-archimedean complete field. We prove a substitute for the reduced fiber theorem (of  Bosch, Lütkebohmert and Raynaud) that holds for every morphism $Y\to X$ flat and with geometrically reduced fibers between $k$-affinoid spaces in the sense of Berkovich, without assuming that $X$ and $Y$ are strict, nor that the relative dimension of $Y$ over $X$ is constant. We do not use the original reduced fiber theorem, nor the language or the techniques of formal geometry. Our statement is formulated in terms of Temkin's graded reduction; our proof rests on a finiteness result of Grauert and Remmert and on Temkin's theory of (graded) reduction of germs of analytic spaces. 
\end{altabstract}
\setcounter{section}{-1}
\section{Introduction}

Soit $k$ un corps ultramétrique complet ; nous travaillerons avec la théorie des espaces $k$-analytiques au sens de Berkovich (\cite{berkovich1990}, \cite{berkovich1993}). Fixons un sous-groupe $\Gamma$ de $\R_{>0}$, non trivial
si $\abs{k\gpm}=\{1\}$ ; pour simplifier
la présentation, supposons que
$\Gamma$ est divisible. 
Soit $A$ une algèbre affinoïde \emph{$\Gamma$-stricte},
c'est-à-dire quotient d'une algèbre de la forme
$k\{T_1/r_1,\ldots, T_n/r_n\}$
où les $r_i$ appartiennent à $\Gamma$ ;
soit $\norm \cdot$ la semi-norme spectrale de $A$ (qui est
une norme si et seulement si $A$
est réduite). Lorsque $\Gamma=\{1\}$ on associe classiquement à
l'algèbre $A$
sa \emph{réduction}
\[\widetilde A:=\{a\in A, \norm a\leq 1\}
/\{a \in A, \norm a <1\}.\]

Lorsque $\Gamma$ n'est pas supposé trivial
cette définition garde un sens, mais l'anneau obtenu
est en général trop petit pour être exploitable
(par exemple, si les $r_i$ sont $\Q$-linéairement
indépendants modulo $\abs {k\gpm} ^\Q$, la
réduction de $k\{T_1/r_1,\ldots, T_n/r_n\}$ au
sens précédent est 
simplement 
le corps résiduel de $k$).
Le bon objet à considérer 
dans ce cas est la réduction \emph{graduée}
de $A$ introduite par Temkin dans \cite{temkin2004}. 

Avant de donner sa construction, précisons que pour éviter une répétition
fastidieuse des adjectifs «homogène» et «gradué», nous avons
opté pour un formalisme qui
diffère en apparence
de celui de Temkin, tout en
lui étant rigoureusement équivalent : nous définissons un \emph{annéloïde} comme 
une \emph{union disjointe}
$R=\coprod_{r\in \Gamma}R^r$ où
chaque $R^r$ est un groupe abélien
noté additivement, munie d'une collection d'applications
bi-additives
$R^r\times R^s\to R^{rs}$ 
définissant sur $R$ une loi commutative, associative,
et possédant un élément neutre $1\in R^1$. Un \emph{corpoïde} est un annéloïde non nul dans lequel tout élément non nul est inversible. La plupart des définitions et résultats de base de l'algèbre commutative et de la géométrie algébrique (idéaux, idéaux premiers,  produit tensoriel, valuations, spectres et même schémas, etc.)
admettent des avatars dans ce nouveau cadre, 
que nous utiliserons librement dans ce qui suit ; nous renvoyons à la section \ref{s-alg-grad}
pour davantage de précisions. 

Soit $A$ une algèbre $k$-affinoïde $\Gamma$-stricte ; posons $X=\mathscr M(A)$.  On définit
d'après Temkin la \emph{réduction graduée} (ou encore l'\emph{annéloïde résiduel}) $\widetilde A$ de $A$ par
la formule \[\widetilde A:=\coprod_{r \in \Gamma}
\{a\in A, \norm a\leq r\}
/\{a \in A, \norm a <r\}.\]
C'est une algèbre de type fini sur le \emph{corpoïde résiduel}
$\widetilde k$ de $k$ ; par exemple,
si $A$ est égale à $k\{T_1/r_1,\ldots, T_n/r_n\}$
alors $\widetilde A=\widetilde k[r_1\backslash T_1,\ldots, r_n\backslash T_n]$, 
qui désigne l'algèbre commutative libre engendrée sur $\widetilde k$ par des indéterminées $T_1,\ldots, T_n$ de degrés
respectifs $r_1,\ldots, r_n$
(\ref{ss-notation-tgamma}). Si $a$ est un élément de $A$ tel que $\norm a\in \Gamma$, nous noterons
$\widetilde a$ son image dans $\widetilde A^{\norm a}$. Nous désignerons par $\widetilde X$ le spectre de $\widetilde A$. La dimension de $\widetilde X$ est égale à celle de 
$X$, 
et l'on dispose d'une application anti-continue surjective
naturelle $X\to \widetilde X$.

Cette opération de réduction (graduée) des espaces affinoïdes est un outil majeur en géométrie analytique,
qui permet de convertir un certain
nombre de questions en des problèmes
de géométrie algébrique (éventuellement graduée,
mais on se ramène ensuite le plus souvent par des méthodes
standard
décrites à la section
\ref{s-alg-grad} à la géométrie algébrique usuelle). Elle est toutefois un peu
délicate
à utiliser, car si elle est fonctorielle, elle ne commute pas aux opérations usuelles.
En effet, soient $X$ et $Y$
deux espaces $k$-affinoïdes $\Gamma$-stricts, et soit $Z$ un espace affinoïde $\Gamma$-strict sur une
extension
de $k$. Supposons donnés
deux morphismes $Y\to X$ et $Z\to X$. On dispose
alors d'un morphisme naturel
$\widetilde{Y\times_X Z}\to
\widetilde Y\times_{\widetilde X}\times \widetilde Z$ qui est fini,
mais n'est pas un isomorphisme en général.
En particulier (prendre $Z=\mathscr M(\hr x)$ où $x\in X$) 
\emph{la formation de la réduction
ne commute pas à la formation des fibres}. 
Cet article remédie dans certaines situations à ce dernier point.

Avant de décrire plus avant nos résultats,
commençons par une définition. Soit
$A\to B$
un morphisme
entre algèbres $k$-affinoïdes $\Gamma$-strictes. Posons 
$X=\mathscr M(A)$
et $Y=\mathscr M(B)$. Une
présentation $B\simeq A\{T_1/r_1,\ldots, T_n/r_n\}/(a_1,\ldots, a_m)$
de $B$ sur $A$ (ou de $Y$ sur $X$) est
\emph{$\Gamma$-sympathique} 
si les 
$r_i$ appartiennent à $\Gamma$ et si
les propriétés suivantes sont satisfaites : 

\begin{enumerate}[1]
\item pour tout $i$ compris entre $1$ et $m$, la
semi-norme spectrale $\rho_i$
de $a_i$ appartient à 
$\Gamma$ ; 
\item le
morphisme
\[p\colon \spec \widetilde A[r\backslash T]/(\widetilde{a_1},\ldots, \widetilde{a_m})
\to \widetilde X\] est plat (nous avons écrit $\widetilde A[r\backslash T]$ pour
$\widetilde A[r_1\backslash T_1,\ldots,
r_n\backslash T_n]$) ; 
\item les fibres de $p$ sont géométriquement réduites ; 
\item pour tout $x\in X$
la norme 
spectrale de $a_i|_{Y_x}$  est égale à $\rho_i$ quel que soit $i$, et
tout élément $a\in \hr x\{T/r\}$ appartenant à 
l'idéal engendré par les $a_i|_{Y_x}$ possède une écriture de la forme
$\sum b_i (a_i|_{Y_x})$ telle que $\norm {b_i}\cdot \rho_i\leq \norm a$ pour
tout $i$ ; 
\item les composantes irréductibles des fibres de $p$ sont géométriquement irréductibles ; 
\item il existe un recouvrement ouvert fini $(\Omega_j)$ de 
$ \spec \widetilde A[r\backslash T]/(\widetilde{a_1},\ldots, \widetilde{a_m})$ tel que pour tout indice
$j$ et tout $\xi\in \widetilde X$, l'intersection $p\inv(\xi)\cap \Omega_j$ soit ou bien vide, 
ou bien une composante connexe de $p\inv(\xi)$. 
\end{enumerate}
L'intérêt essentiel de cette notion de notre point de vue est le suivant : la conjonction des propriétés (3) et (4) implique que pour tout $x\in X$
dont on note $\widetilde x$ l'image sur $\widetilde X$, 
le morphisme 
naturel $\widetilde{Y_x}\to 
\spec \hrt x[r\backslash T]/(\widetilde{a_1|_{Y_x}}, \ldots, \widetilde{a_m|_{Y_x}})=p\inv(\widetilde x)_{\hrt x}$
est un isomorphisme (proposition \ref{prop-univdistingue-1}).
En un sens un peu vague, une présentation $\Gamma$-sympathique d'un morphisme $Y\to X$ induit donc par réduction des relations une présentation plate sur $\widetilde X$ de
la famille $(\widetilde{Y_x})_{x\in X}$.

La notion de présentation $\Gamma$-sympathique est stable par changement de base quelconque : si $B\simeq A\{T/r\}/(a_1,\ldots, a_m)$ est une présentation $\Gamma$-sympathique de $B$ sur $A$ alors pour toute extension complète $L$ de $k$, toute algèbre $L$-affinoïde $\Gamma$-stricte $A'$ et tout morphisme $A\to A'$,
la présentation
$B\widehat \otimes_A A'\simeq A'\{T/r\}/(a_1,\ldots, a_m)$ de $B\widehat \otimes_A A'$ sur $A'$
est $\Gamma$-sympathique.

Nous pouvons maintenant énoncer
le théorème principal de cet article
(théorème \ref{theo-main}).

\begin{enonce*}{Théorème}Soit $Y\to X$ un morphisme entre espaces $k$-affinoïdes $\Gamma$-stricts, 
supposé plat et à fibres géométriquement réduites. Il existe une famille finie $(X_i\to X)$
de morphismes dont les images recouvrent $X$ et telle que pour tout $i$, 
les conditions suivantes soient satisfaites, en posant $Y_i=Y\times_X X_i$ : 
\begin{itemize}[label=$\diamond$]
\item l'espace $X_i$ est affinoïde et $\Gamma$-strict ; 
\item si $\abs{k\gpm}\neq \{1\}$ le morphisme $X_i\to X$ est quasi-étale  ; 
\item si $\abs{k\gpm}=\{1\}$ le morphisme $X_i\to X$ possède une factorisation $X_i\to X'_i\to X$
où  $X'_i$ est affinoïde et $\Gamma$-strict, où $X_i\to X'_i$ est fini, radiciel et plat et où $X'_i\to X$ est quasi-étale ;

\item l'espace $Y_i$ possède une présentation $\Gamma$-sympathique sur $X_i$.
\end{itemize}
\end{enonce*}

\subsection*{Liens avec le théorème de la fibre réduite}
Le théorème ci-dessus était déjà connu lorsque $\Gamma=\{1\}$ (ce qui entraîne que $\abs {k\gpm}\neq \{1\}$) et lorsque la dimension relative de $Y\to X$ est constante : c'est alors
en effet une conséquence du \emph{théorème de la fibre réduite}
de Bosch, Lütkebohmert et Raynaud (\cite{frg4}, théorème 2.1) qui assure sous ces hypothèses l'existence, après un changement de base quasi-étale, d'un modèle \emph{formel} affine
de $Y\to X$, plat et à fibres géométriquement réduites ; et un tel modèle fournit une présentation $\{1\}$-sympathique de $Y\to X$ (du moins, une présentation satisfaisant les conditions (1) à (4) ci-dessus ; mais il n'est pas très difficile d'imposer ensuite (5) et (6) quitte à procéder à un changement de base quasi-étale supplémentaire, provenant d'un morphisme étale au niveau de la fibre spéciale). 

Notre énoncé est donc plus faible que le théorème de la fibre réduite lorsque celui-ci s'applique, mais il vaut sans
hypothèse d'équidimensionalité et pour un
groupe $\Gamma$ arbitraire.

Précisons que nous n'utilisons pas dans notre preuve le théorème de la fibre réduite, et que nous le \emph{retrouvons}
sous certaines hypothèses additionnelles. En effet si la valeur absolue de $k$ n'est pas triviale, si $Y\to X$ est un morphisme 
plat et à fibres géométriquement réduites
entre espaces strictement affinoïdes, et \emph{si $X$ admet un modèle formel sur $k^\circ$ à fibre spéciale réduite} (c'est automatiquement le cas lorsque $X$ est réduit et $k$ algébriquement clos), il n'est pas difficile de construire à partir d'une présentation $\{1\}$-sympathique de $Y$ sur $X$ un modèle formel plat et à fibres géométriquement réduites du morphisme $Y\to X$ (proposition \ref{prop-retrouve-fibred}). 

Bien entendu, il est naturel de se demander si l'on peut
espérer une version $\Gamma$-stricte du théorème de la fibre réduite
pour $\Gamma$ quelconque. C'est vraisemblable, mais cela requerrait 
en
premier lieu de définir une variante $\Gamma$-graduée ou $\Gamma$-filtrée de la théorie des schémas formels, ne serait-ce que pour pouvoir formuler un énoncé. 

\subsection*{Motivations}
Pour tout espace affinoïde $\Gamma$-strict $X$,
l'application anti-continue $X\to \widetilde X$ induit une bijection $\pi_0(X)\simeq \pi_0(\widetilde X)$.
Notre théorème peut donc être utilisé pour
contrôler la variation des composantes connexes géométriques des fibres
d'un morphisme plat et à fibres géométriquement réduites, ce dont nous
avons besoin de manière cruciale
dans notre travail \cite{ducrosYYY}
sur l'aplatissement par éclatements en géométrie de Berkovich. 
C'était
notre principale motivation pour ce présent article
et c'est la raison pour laquelle
nous avons inclus les conditions (5) et (6) dans la définition d'une présentation
sympathique. 

Mentionnons que
cette stratégie détournée pour
comprendre la variation des composantes connexes en famille en géométrie analytique n'est pas nouvelle : elle a par exemple été suivie par Abbes et Saito dans
\cite{abbes-s2002} ou par Poineau dans
\cite{poineau2008} (ces auteurs travaillent dans le cadre strictement analytique
en se fondant sur le théorème de la fibre réduite).

\subsection*{À propos de notre démonstration} 
L'ingrédient majeur de ce travail est
le théorème de finitude de Grauert et Remmert ; il est 
établi dans \cite{grauert-r1966}, mais on en trouve des variantes un peu renforcées dans la littérature postérieure, comme par exemple
le théorème 1 de  \cite{bosch-g-r1984} 6.2.4, ou le théorème 1.3 de \cite{frg4} -- on peut voir
ce dernier comme
la version \emph{absolue}, c'est-à-dire au-dessus de $\mathrm{Spf}\;k^\circ$, du théorème de la fibre réduite de Bosch, Lütkebohmert et Raynaud. Ce théorème
de Grauert et Remmert ne porte que sur les algèbres strictement $k$-affinoïdes, et nous devons l'étendre au cas $\Gamma$-strict pour pouvoir nous en servir. C'est l'objet de la section \ref{s-grg}. On y démontre 
un analogue 
du théorème  1 de \cite{bosch-g-r1984} 6.2.4
(théorème \ref{theo-reduit-distingue}), dont on déduit un
analogue du théorème 1.3 de \cite{frg4}
(théorème \ref{theo-affinoide-ponctuel})
qui peut s'énoncer comme suit :
\emph{si $A$ est une algèbre $k$-affinoïde $\Gamma$-stricte
et géométriquement réduite,
il existe une extension finie $L$ de $k$, qu'on peut
prendre séparable si $\abs{k\gpm}\neq \{1\}$, telle que $A_L$
possède une présentation $\Gamma$-sympathique sur $L$}. 
Le théorème \ref{theo-reduit-distingue}
se démontre
par réduction au cas strict
\emph{via} une extension des scalaires 
de $k$ à $\hr {\eta_r}$ pour une famille
finie convenable $r=(r_1,\ldots,r_n)$ d'éléments de $\Gamma$, où $\eta_r$ est le point de $\A^{n,\mathrm{an}}_k$
donné par la semi-norme $\sum a_IT^I \mapsto
\max \abs{a_i}\cdot r^I$.
(La descente de $\hr{\eta_r}$ à $k$
fonctionne parce que le $k$-espace de Banach $\hr{\eta_r}$
est «somme directe orthogonale» de droites ;
c'est évident si les $r_i$ sont multiplicativement
indépendants modulo $\abs{k\gpm}$, mais c'est
vrai en toute généralité, \emph{cf.} le lemme \ref{lem-kr-schauder}.)

Nous avons également besoin du «théorème de la fibre réduite sur un annéloïde de valuation» (corollaire \ref{coro-fibre-reduite}),
mais c'est en fait encore une conséquence du théorème de Grauert et Remmert. En effet, nous le prouvons en nous réduisant au cas non gradué, dans lequel il a été démontré par Temkin (\cite{temkin2010}, étapes 2, 3 et 4 de la preuve du théorème 3.5.5)  selon le schéma de preuve  suivant : 
réduction au cas de hauteur finie par approximation, puis de hauteur 1 par dévissage,  
et  procédé de complétion et décomplétion
pour se ramener au théorème de Grauert-Remmert. 

Disons maintenant quelques mots de la façon dont nous utilisons ces résultats pour démontrer notre théorème principal, en supposant (pour alléger un peu l'exposition) que $\Gamma=\R_{>0}$ et que la valeur absolue de $k$ n'est pas triviale ; nous écrirons «sympathique» au lieu de «$\R_{>0}$-sympathique». Notre preuve fait appel de manière cruciale à la théorie de la réduction (graduée) des germes d'espaces analytiques, développée dans \cite{temkin2004} ; c'est un substitut à la géométrie formelle, plus souple d'utilisation, de nature plus locale, et qu'on peut employer sur des espaces non stricts.

Soit donc un morphisme $Y\to X$ entre espaces $k$-affinoïdes, supposé plat et à fibres géométriquement réduites. Soient $A$ et $B$ les algèbres affinoïdes respectivement
associées à $X$ et $Y$. On fixe un point $x$ de $X$, et on note $C$ l'algèbre $B\hat \otimes_A\hr x$. On fixe enfin une
valuation $\abs \cdot$
appartenant à la réduction $\widetilde{(X,x)}$ du germe $(X,x)$
(c'est un ouvert quasi-compact et non vide de l'espace
de Zariski-Riemann gradué $\P_{\hrt x/\widetilde k}$ de $\hrt x$ sur $\widetilde k$, égal à $\P_{\hrt x/\widetilde k}$
tout entier si et seulement si $x\notin \partial X$). On
note $\hrt x^\circ$ l'annéloïde de $\abs \cdot$. 

Par compacité de $X$ et par quasi-compacité de $\widetilde{(X,x)}$, il suffit d'exhiber un espace affinoïde $V$ et un morphisme $V\to X$ quasi-étale possédant les trois propriétés suivantes : 
\begin{enumerate}[a]
\item l'image $W$ de $V\to X$ (qui est un domaine analytique compact de $X$) contient $x$ ; 
\item $\widetilde{(W,x)}$ contient $\abs \cdot$ ; 
\item $Y\times_X V\to V$ admet une présentation
sympathique.
\end{enumerate}
Dans ce qui suit, l'expression «restreindre $X$» signifiera «remplacer $X$ par $V$ pour un certain espace affinoïde $V$ muni d'un morphisme quasi-étale $V\to X$ vérifiant les conditions (a) et (b) ci-dessus».  

La fibre $Y_x$ étant géométriquement réduite, elle
est justiciable de notre
variante $\Gamma$-stricet du théorème de Grauert et Remmert
(théorème \ref{theo-affinoide-ponctuel}) ; elle affirme
l'existence
d'une extension finie séparable $E$ de $\hr x$ telle que $Y_{x,E}$
admette une présentation sympathique
sur $E$ ;  on choisit un prolongement arbitraire de $\abs \cdot $ à $\widetilde E$,
encore noté $\abs \cdot$.
Il existe un morphisme fini étale de germes $(X',x')\to (X,x)$ tel que $\hr {x'}=E$ ; on peut donc
restreindre $X$ de sorte que $Y_x$
admette une présentation sympathique
sur
$\hr x$. 
Soit $\mathsf B$ l'image de $\widetilde B$
dans $\widetilde C$. L'annéloïde $\widetilde C$ est
entier sur $\hrt x^\circ\cdot \mathsf B$ ; 
il existe donc une sous-$\hrt x^\circ$-algèbre
$R$ de $\widetilde C$, plate (car sans torsion)
et de présentation finie, contenant $\mathsf B$ et entière sur $\hrt x^\circ\cdot\mathsf B$, telle que $R_{\hrt x}=\widetilde C.$
Remarquons que si $C$ et $\widetilde C$ ne dépendent que de la fibre $Y_x$, ce n'est pas le cas de l'annéloïde $R$, qui contient des informations 
sur le comportement de
l'espace total $Y$ au voisinage de $Y_x$ dans la «direction» codée par la valuation $\abs \cdot \in \widetilde{(X,x)}$. 

L'étape suivante consiste à appliquer le théorème de la fibre réduite sur un annéloïde de valuation au morphisme 
$\spec R\to \spec \hrt x^\circ$, ce qui est licite car sa fibre générique $\spec \widetilde C$ est géométriquement réduite. Il affirme l'existence d'une extension valuée finie $\Lambda$ de $\hrt x$ et d'un sous-annéloïde $R'$ de $\widetilde \Lambda$, fini sur $R_\Lambda$, tel que les fibres du morphisme naturel $q\colon \spec R'\to \spec \Lambda^\circ$ soient géométriquement réduites ;  quitte à agrandir $\Lambda$, on peut supposer que 
les composantes irréductibles des fibres de $q$ sont géométriquement irréductibles. Par ailleurs le caractère géométriquement réduit
des fibres de $q$ entraîne l'existence d'un recouvrement
ouvert fini $(\Upsilon_j)$ de $\spec R'$ tel que pour tout
$j$ et tout $\xi \in \spec \Lambda^\circ$, l'intersection $\Upsilon_j\cap q\inv(\xi)$ soit ou bien vide ou bien
une composante connexe de $q\inv(\xi)$ (proposition \ref{prop-schema-decoupage}). 
En réalisant $\Lambda$ comme le corpoïde résiduel d'une extension finie séparable de $\hr x$ et en raisonnant comme plus haut, on voit qu'on peut restreindre $X$ de sorte que $\Lambda=\hrt x$, puis supposer que $R'=R$ (en remplaçant $R$ par $R'$). 

En partant d'une présentation sympathique arbitraire de $Y_x$ sur $\hr x$ et d'une présentation de $R$ sur $\hrt x^\circ$ on construit alors
par différentes méthodes (relèvements, renormalisation, concaténation, petites perturbations\ldots) une présentation sympathique de $Y_x$ sur $\hr x$ qui a la vertu de s'étendre, quitte à restreindre $X$, en une présentation sympathique de $Y$ sur $X$ ; la platitude de $Y$ sur $X$ joue un rôle crucial dans cette étape.

\subsection*{Remerciements}
Je sais gré au rapporteur anonyme de sa lecture très attentive du manuscrit et de ses nombreuses remarques et suggestions, qui m'ont notamment permis de corriger une erreur importante dans la version initiale du critère \ref{ss-critere-utile} (d'autres modifications ont ensuite été nécessaires pour pouvoir utiliser la version correcte dudit critère  ; c'est ce qui a conduit  à la rédaction de l'appendice). 

\section{Algèbre graduée}\label{s-alg-grad}

\subsection{}Nous ferons un usage systématique du formalisme gradué introduit par Temkin dans \cite{temkin2004}. 
Nous avons toutefois choisi d'adopter une présentation de sa théorie qui diffère légèrement de l'originale 
et nous paraît un peu plus maniable (ce point de vue a déjà été suivi par l'auteur dans \cite{ducrosXXX}, et par Marc Chapuis dans
\cite{chapuis2017}).

Plus précisément, donnons-nous un groupe abélien divisible $D$ noté multiplicativement (en pratique, $D$ sera le plus souvent contenu
dans
$\R_{>0}$). Pour Temkin, un anneau $D$-gradué est un anneau $A$ muni d'une décomposition en \emph{somme directe} $A=\bigoplus_{d\in D} A^d$ avec 
$A^{d}\cdot A^{d'}\subset A^{dd'}$. Or l'expérience montre que lorsqu'on utilise ces anneaux gradués en géométrie analytique il n'arrive \emph{jamais}
qu'on ait besoin d'additionner deux éléments homogènes de degrés différents. Il nous a donc semblé plus commode de considérer des décompositions
en \emph{unions disjointes} ; l'avantage de ce choix est qu'il n'y a plus besoin de spécifier en permanence que les éléments manipulés sont homogènes, puisqu'ils
le sont automatiquement ; ceci permet d'alléger substantiellement la rédaction sans changer quoi que ce soit sur le fond. 

Nous appellerons donc \emph{$D$-annéloïde}
un ensemble $A$ muni des données suivantes : 
\begin{itemize}[label=$\diamond$]
\item une décomposition $A=\coprod_{d\in D} A^d$ ; 
\item sur chaque $A^d$, une loi de composition interne notée additivement qui fait de $A^d$ un groupe abélien, de neutre $0^d$ ; 
\item une loi de composition interne sur $A$ associative et commutative, notée multiplicativement, qui possède un élément neutre $1\in A^1$
et induit pour tout $(d,d')$ une application bi-additive $A^d\times A^{d'}\to A^{dd'}$.
\end{itemize}

Un élément $a$ d'un $D$-annéloïde $A$ appartient à un unique sous-ensemble $A^d$ ; on dit que $d$ est le \emph{degré}
de $a$ et on le note $\mathfrak d(a)$ ; si $A'$ est une partie de $A$, nous noterons $\mathfrak d(A')$ l'ensemble des 
$\mathfrak d(a)$ pour $a$ parcourant $A'$. 

On emploiera le plus souvent la notation $0$ au lieu de $0^d$ si le degré en jeu est clairement indiqué par le contexte; on pourra ainsi par exemple écrire
«$a=0$» au lieu de «$a=0^{\mathfrak d(a)}$», et on notera $A_{\neq 0}$ l'ensemble des éléments non nuls de $A$, c'est-à-dire
la réunion disjointe des $A^d\setminus \{0^d\}$. L'ensemble $A_{\neq 0}$ est vide si et seulement si $1=0$ (c'est-à-dire $1=0^1$ !) ; on dit alors que $A$ est nul. 

Les $D$-annéloïdes forment une catégorie ; les morphismes sont les applications qui préservent les degrés, commutent aux deux opérations et envoient 1 sur 1. 
La catégorie des $\{1\}$-annéloïdes est celle des anneaux commutatifs unitaires. 

Si $E$ est un sous-ensemble de $D$ et si $A$ est un annéloïde, on notera $A^E$ la réunion des $A^d$ pour $d\in E$. Supposons de plus que
$E$ soit un sous-groupe de $D$. Pour tout annéloïde $A$, le 
sous-ensemble $A^E$ hérite alors par restriction d'une structure de $E$-annéloïde. On dispose ainsi d'un foncteur de la catégorie des $D$-annéloïdes vers
celle des $E$-annéloïdes. Ce foncteur possède un adjoint à gauche, le «prolongement par $0$» qui envoie un $E$-annéloïde $B$ sur le $D$-annéloïde $A$ tel que $A^d=B^d$ si $d\in E$
et $A^d=\{0^d\}$ sinon. Le prolongement par zéro induit une équivalence entre la catégorie des $E$-annéloïdes et celle des $D$-annéloïdes dont le degré de tout élément
non nul appartient à $E$ ; comme la plupart des propriétés que nous considérerons seront invariantes par cette équivalence de catégories, il arrivera souvent
qu'on identifie subrepticement un $E$-annéloïde au $D$-annéloïde obtenu en le prolongeant par zéro.

Un \emph{$D$-corpoïde}
est un $D$-annéloïde non nul dans lequel tout élément non nul est inversible. 

Si $A$ est un $D$-annéloïde, un $A$-module est un ensemble $M$ muni d'une décomposition $M=\coprod_{d\in D} M^d$, d'une structure de groupe abélien
sur chacun des $M^d$, et d'une multiplication externe $A\times M\to M$ qui induit pour chaque couple $(d,d')$ une application bi-additive $A^d\times M^{d'}\to M^{dd'}$
et vérifie les égalités 
\[1\cdot m=m\;\;\text{et}\;\;a(bm)=(ab)m\]pour tout
$(a,b,m)\in A^2\times M$. (Lorsque $A$ est un $D$-corpoïde, on parlera plutôt de $A$-espace vectoriel.)
Une $A$-algèbre est un $D$-annéloïde $B$ muni d'un morphisme  $A\to B$ ; c'est en particulier un $A$-module. Si $F$ est un $D$-corpoïde, une
extension de $F$ est une $F$-algèbre qui est elle-même un $D$-corpoïde. 

À partir de maintenant et dans toute la suite de cette section nous dirons simplement «annéloïde» et «corpoïde» -- la graduation est donc indexée par $D$
sauf mention expresse du contraire.

\subsection{}
La plupart des notions usuelles d'algèbre commutative admettent un avatar gradué ayant en général
la même définition \emph{mutatis mutandis} -- on vient d'en voir quelques exemples.
On peut ainsi définir ce que sont un idéal d'un annéloïde (puis un idéal premier ou maximal), un annéloïde intègre, un annéloïde réduit, un annéloïde local, la localisation d'un annéloïde 
par une partie multiplicative, le corpoïde des fractions d'un annéloïde intègre, le spectre d'un annéloïde, un élément entier sur un annéloïde (ou algébrique sur un corpoïde),
le produit tensoriel de deux modules sur un annéloïde (puis un module plat), etc. Et une bonne partie des énoncés classiques se retranscrivent
à peu près tels quels dans ce contexte, avec les «mêmes» démonstrations : on dispose ainsi d'une théorie de la dimension des espaces vectoriels et du degré de transcendance des extensions
de corpoïdes, du lemme de Nakayama, du lemme de \emph{going-down}, d'une théorie des valuations (nous donnons quelques précisions sur cette dernière au
paragraphe \ref{ss-valuation} plus bas)
\ldots Nous utiliserons tout ceci librement, et renvoyons le lecteur à différents textes pour les détails et (certaines) preuves : la section 1 de \cite{temkin2004}, 
la section 1 de \cite{ducros2013b} et la section 2.2 de \cite{ducrosXXX} (dans ces deux derniers textes, nous présentons notamment la théorie de Galois graduée), 
ou encore l'appendice A de \cite{ducros2018}. Le lecteur pourra aussi se reporter à la section 2 de \cite{poineau2013} pour une version graduée du \emph{Nullstellensatz} qui, une fois n'est pas coutume, n'est pas le décalque de sa version classique (ledit décalque est grossièrement faux) ; voir aussi à ce sujet \cite{ducrosXXX}, 2.2.24 et 2.2.29.2. 

Si $A$ est un annéloïde , si $M$ est un $A$-module et si $B$ est une $A$-algèbre,
le $B$-module  $B\otimes_AM$ sera souvent notée $M_B$.

Le spectre d'un annéloïde $A$ possède un faisceau d'annéloïdes naturel à fibres locales ; en recollant de tels objets, on obtient ce qu'on appelle un \emph{schémoïde}. Si $X$ est un schémoïde, le corpoïde résiduel d'un point $x$ de $X$ sera noté $\kappa(x)$, et l'application $\mathscr O_{X,x}\to \kappa(x)$ sera notée $a\mapsto a(x)$. Si $Y\to X$ est un morphisme de schémoïdes, sa fibre en un point $x$ de $X$ sera notée $Y_x$ ; c'est un $\kappa(x)$-schémoïde.

\subsection{}\label{ss-notation-tgamma}
Soit $A$ un annéloïde. Si $r=(r_i)_{i\in I}$ est une famille d'éléments de $D$ et si $T=(T_i)_{i\in I}$ est une famille d'indéterminées, 
nous noterons $A[r\backslash T]$ l'annéloïde dont le sommande de degré $d$ est constitué des sommes formelles finies
du type $\sum a_J T^J$ où $J \in A^{dr^{J}}$ pour tout $J$. Chaque $T_i$ est donc de degré $r_i$
(et $(A[r\backslash T], (T_i)_i)$ représente le foncteur qui envoie une $A$-algèbre $B$ sur $\prod_i B^{r_i}$).

Si $F$ est un corpoïde et si $r$ est comme ci-dessus, $F[r\backslash T]$ est intègre et son corpoïde des fractions
sera noté $F(r\backslash T)$ ; lorsque $r$ est $F$-libre, c'est-à-dire lorsque les $r_i$
forment une famille libre du $\Q$-espace vectoriel $D/(\mathfrak d(F\gpm)^\Q)$, le corpoïde
$F(r\backslash T)$ est simplement $F[r\backslash T, r\inv\backslash S]/(S_iT_i-1)_i$. 

Mentionnons que ce sont ces annéloïdes de polynômes avec degrés des indéterminées prescrits qui sont
utilisés pour définir les notions d'éléments entiers ou algébriques : par exemple si $B$ est une $A$-algèbre, 
un élément $b$ de $B$ est dit entier sur $A$ s'il annule un polynôme unitaire appartenant à $A[\mathfrak d(b)\backslash T]$.

\subsection{}\label{ss-methode-schemas}
Soit $A$ un annéloïde et soit $B$ une $A$-algèbre
telle que
\[\mathfrak d(B_{\neq 0})=\mathfrak d(A_{\neq 0})= \mathfrak d(A\gpm).\]
Dans ce cas pour tout élément $r$ appartenant à $\mathfrak d(B_{\neq 0})$
la flèche naturelle de
$A^r\otimes_{A^1}B^1$ vers $B^r$ est un isomorphisme (et $A^r$ est libre de rang $1$
comme $A^1$-module) si bien que $B\simeq A\otimes_{A^1}B^1$.

Cette remarque peut permettre de démontrer une propriété
en se ramenant au cas non gradué, grâce à un premier changement de base de la forme $A\to A[r\backslash T, r\inv \backslash T\inv]$ ou $A\to A(r\backslash T)$ si $A$ est un corpoïde (pour une famille convenable $r$ d'éléments de $D$). Bien entendu, il faut s'assurer au préalable que 
la propriété concernée satisfait la descente relative aux  changements de base de ce type, ce qui est souvent facile à établir par un calcul explicite. Donnons deux exemples. 

\subsubsection{}
Si $M$ est un $A$-module, il est fini si et seulement si $M\otimes_A
A[r\backslash T, r\inv \backslash T\inv]$ est fini sur 
$A[r\backslash T, r\inv \backslash T\inv]$ (considérer les termes constants d'une famille finie de polynômes à coefficients dans $M$ engendrant 
 $M\otimes_AA[r\backslash T, r\inv \backslash T\inv]$). 
 
 \subsubsection{}
 Si $X$ est un $A$-schémoïde et si $\pi$ désigne la projection $X_{A[r\backslash T, r\inv\backslash T\inv]}\to X$
 une partie $U$ de $X$ est ouverte si et seulement si $\pi\inv(U)$ est ouverte (remarquer que la section
 $\sigma$ de 
$\pi$ qui envoie un point $\xi$ sur le point générique de $\spec \kappa(\xi)[r\backslash T,r\inv \backslash T\inv]$
est continue).

\subsection{Quelques propriétés géométriques dans le cadre gradué}
Nous allons utiliser la méthode générale évoquée au \ref{ss-methode-schemas}
pour étudier le comportement de certaines propriétés des algèbres sur un corpoïde par extension des scalaires, en nous ramenant
à ce qui est connu en algèbre commutative classique. 

\subsubsection{Algèbres réduites sur un corpoïde parfait}
\label{sss-red-parfait}
Soit $F$ un corpoïde parfait et soit $A$ une $F$-algèbre réduite. 
Soit $L$ une extension de $F$ ; la $L$-algèbre $A_L$ est encore réduite. Pour le voir, on
choisit une famille $r$ d'éléments de $D$ telle que
\[\mathfrak d(L(r\backslash T)\gpm)
=\mathfrak d(F(r\backslash T)\gpm)\supset \mathfrak d(A_{\neq 0}).\]
Soit $p$ l'exposant caractéristique de $F$. 
Posons
\[K=\lim_{\rightarrow}F(r^{1/p^n}\backslash T^{1/p^n})\;\;\text{et}\;\;
\Lambda=\lim_{\rightarrow}L(r^{1/p^n}\backslash T^{1/p^n}).\]

On a alors 
$\mathfrak d(A_{K\neq 0})\subset \mathfrak d(K\gpm)
=\mathfrak d(\Lambda\gpm)
\supset\mathfrak d(A_{\Lambda\neq 0})$
(les deux relations d'inclusion sont des égalités dès que $A$
est non nulle), 
si bien que
\[A_K=(A_K)^1\otimes_{K^1}K,\;\;
(A_\Lambda)^1=(A_K)^1\otimes_{K^1}\Lambda^1\;\text{et}\;
A_\Lambda=(A_\Lambda)^1\otimes_{\Lambda^1}\Lambda.\]
Pour montrer que $A_L$ est réduite, il
suffit de montrer que $A_\Lambda$ est réduite, 
et même que $(A_{\Lambda})^1$ est réduite (en effet, si $a$ est un élément non nul de
$A_\Lambda$, son degré $\mathfrak d(a)$ appartient à $\mathfrak d(\Lambda^\times)$, 
si bien qu'il existe $\lambda \in \Lambda^\times$ tel que $\lambda a$ soit de degré $1$, 
et $a$ est alors nilpotent si et seulement si $\lambda a$ est nilpotent). 
Mais comme $A$ est réduite, $A_K$ est réduite par la forme même de $K$ ; 
en particulier, $(A_K)^1$ est réduite. Par ailleurs, puisque $F$ est parfait, $K$ est parfait, et $K^1$ est \emph{a fortiori}
parfait. Il s'ensuit que $(A_{\Lambda})^1=(A_K)^1\otimes_{K^1}\Lambda^1$ est réduite. 

\subsubsection{Algèbres géométriquement réduites}
Soit $F$ un corpoïde et soit $A$ une  $F$-algèbre. 
Soit $L$ une extension parfaite de $F$. Il résulte de \ref{sss-red-parfait}
que $A_L$ est réduite si et seulement si $A_\Lambda$ est réduite pour toute extension $\Lambda$ de $F$ ; si ces conditions équivalentes sont satisfaites nous dirons que $A$ est géométriquement réduite.

Cette notion donne lieu par recollement à la notion de schémoïde géométriquement réduit sur $F$, qui se teste sur une extension parfaite de $F$.

\subsubsection{Algèbres intègres sur un corpoïde algébriquement clos}
Soit $F$ un corpoïde algébriquement clos et soit $A$ une $F$-algèbre
intègre. Soit $L$ une extension de $F$ ; la $L$-algèbre $A_L$ est encore intègre. 
Pour le voir, on peut supposer par un argument de limite inductive que $A$ et $L$
sont de type fini sur $F$ (respectivement comme algèbre et comme extension). Dans ce cas
le sous-groupe de $D/\mathfrak d(F\gpm)$ engendré par $\mathfrak d(A_{\neq 0})$ et $\mathfrak d(L\gpm)$ est 
de type fini, et partant libre de rang fini puisque $\mathfrak d(F\gpm)$ est divisible. 
Il existe alors une famille finie et $F$-libre $r$ d'éléments de $D$ telle que
 \[\mathfrak d(L(r\backslash T)\gpm)
=\mathfrak d(F(r\backslash T)\gpm)\supset \mathfrak d(A_{\neq 0}).\]
Posons $K=F(r\backslash T)$ et $\Lambda=L(r\backslash T)$. On  a
 \[A_K=(A_K)^1\otimes_{K^1}K,\;\;
(A_\Lambda)^1=(A_K)^1\otimes_{K^1}\Lambda^1\;\text{et}\;
A_\Lambda=(A_\Lambda)^1\otimes_{\Lambda^1}\Lambda.\]
Pour montrer que $A_L$ est intègre il
suffit de montrer que $A_\Lambda$ est intègre, 
et même que $(A_{\Lambda})^1$ est intègre (par un argument analogue à celui utilisé plus haut pour le caractère réduit). 

La fermeture séparable de $K$ dans 
$A_K$ est triviale. En effet, soit $M$ une extension finie séparable de $K$.
Le groupe $\mathfrak d(M\gpm)$ est de rang rationnel fini sur $\mathfrak d(F\gpm)$ ; comme ce dernier est divisible, $\mathfrak  d(M\gpm)/\mathfrak d(F\gpm)$ est libre de rang fini ; soit $s=(s_1,\ldots, s_n)$ une base de ce quotient. En choisissant pour tout $i$ un élément de $M\gpm$ de degré $s_i$ on obtient un plongement 
de $F(s\backslash T)$ dans $M$ ; on a alors
$\mathfrak d(F(s\backslash T)\gpm)=\mathfrak d(M\gpm)$ et $M^1=F^1=F(s\backslash T)^1$ (la première égalité provenant du fait que $F^1$
est algébriquement clos). Par conséquent
$M=F(s\backslash T)$ et un calcul direct montre que $A_M$ est intègre ; par conséquent, si $M$ n'est pas triviale elle n'admet aucun $K$-plongement dans $A_K$. 

Dès lors la fermeture séparable de $K^1$ dans $(A_K)^1$ est triviale et $(A_{\Lambda})^1=A^1\otimes_{F^1}\Lambda^1$ est intègre. 

\subsubsection{Algèbres géométriquement intègres}\label{sss-alg-geomintegre}
Soit $F$ un corpoïde et soit $A$ une  $F$-algèbre. 
Soit $L$ une extension algébriquement close de
$F$. Il résulte de \ref{sss-red-parfait}
que $A_L$ est intègre si et seulement si $A_\Lambda$ est intègre pour toute extension $\Lambda$ de $F$ ; si ces conditions équivalentes sont satisfaites nous dirons que $A$ est géométriquement intègre.

\subsubsection{}
Un schémoïde $X$ sur un corpoïde $F$ est dit géométriquement irréductible si $X_L$ est irréductible pour toute extension $L$ de $F$. 

Supposons qu'il existe une extension $\Lambda$ séparablement close de $F$ tel que $X_\Lambda$ soit irréductible.
Soit $\Lambda'$ une clôture algébrique de  $\Lambda$. Le schémoïde $Y:=(X_{\Lambda'})_{\mathrm{red}}$
est homéomorphe à $X_\Lambda$, et est par conséquent irréductible; et il est réduit par définition.
Il résulte alors de \ref{sss-alg-geomintegre}
que $Y_M$ est irréductible pour toute extension $M$ de $\Lambda'$ ; ceci entraîne que $X_M$ est irréductible pour toute telle $M$
(il est en effet homéomorphe à $Y_M$), puis que $X$ est géométriquement irréductible. 

\subsection{Dimension}
Soit $F$ un corpoïde et soit $A$ une $F$-algèbre de type fini. Soit $L$ une extension de $F$ telle
que $\mathfrak d(A_{L \neq 0})=\mathfrak d(L\gpm)$. La dimension de Krull de $A_L^1$ ne dépend alors pas de $L$, 
et est appelée \emph{$F$-dimension}
de $A$, ou dimension de $A$ s'il n'y a pas d'ambiguïté (mais on prendra garde que cette notion dépend du corpoïde de base $F$ : si $\rho$ est un élément de $D$
qui n'est pas de torsion modulo $\mathfrak d(F\gpm)$ alors $F(\rho\backslash T)$ est de dimension $0$ sur
lui-même mais de dimension $1$
sur $F$). 

Si les $\mathfrak p_i$ sont les idéaux premiers minimaux de $A$ alors
\[\dim_F A=\sup_i \mathrm{deg. tr.}(\mathrm{Frac}(A/\mathfrak p_i)/F)\]
comme on le voit en se ramenant au cas non gradué par extension des scalaires
à un corpoïde $L$ comme ci-dessus de la forme $F(r\backslash T)$ avec $r$ convenable (ce type
d'extension a l'avantage de préserver l'intégrité). 

On étend la notion de $F$-dimension aux schémoïdes de type fini sur $F$ 
en posant pour un tel schémoïde $X$
\[\dim_F X=\sup_{x\in X} \mathrm{deg. tr.}(\mathrm{Frac}(\kappa(x)/F).\]
Si $x$ est un point de $X$ la dimension de $X$ en $x$ est le maximum des dimensions
des composantes irréductibles de $X$ contenant $x$. Si $f\colon Y\to X$ est un morphisme
entre $F$-schémoïdes de type fini, la dimension relative de $Y$ sur $X$ en un point $y$ de $Y$
est la dimension de $Y_{f(y)}$ en $y$. 

\subsection{}\label{ss-valuation}
Soit $A$ un annéloïde.
Une \emph{valuation}
sur $A$ est la donnée d'un groupe abélien ordonné $\Gamma$
noté multiplicativement et d'une application $\abs\cdot  \colon A\to \Gamma\cup\{0\}$
(où $0$ est un élément absorbant pour la multiplication et plus petit que tout élément de $\Gamma$) telle que:
\begin{itemize}[label=$\diamond$]
\item $\abs 1=1$ ; 
\item $ \abs {ab}=\abs a \cdot \abs b$ pour tout couple $(a,b)$ d'éléments de $A$ ; 
\item $\abs{a+b}\leq \max (\abs a, \abs b)$ pour tout couple
$a$ et $b$ d'éléments de même degré de $A$. 
\end{itemize}

Deux valuations $\abs
\cdot  \colon A\to \Gamma\cup\{0\}$ et $\abs\cdot' \colon A\to \Gamma'\cup\{0\}$ sont dites équivalentes s'il existe un groupe abélien ordonné $\Gamma''$, deux morphismes injectifs croissants $i\colon \Gamma''\to \Gamma$ et $j\colon \Gamma''\to \Gamma'$, et une valuation $\abs \cdot''
\colon
F\to \Gamma''\cup\{0\}$ telle que $\abs \cdot =i\circ \abs \cdot ''$ et $\abs \cdot' =j\circ \abs \cdot''$.

\subsection{}
Soit $F$ un corpoïde valué (c'est-à-dire muni d'une valuation). Nous emploierons par défaut
les notations suivantes : 

\begin{itemize}[label=$\diamond$]
\item $\abs \cdot$ pour la valuation de $F$ 
\item $F^\circ$ pour l'annéloïde $\{z\in F, \abs z \leq 1\}$ ; ce dernier est local et est appelé l'annéloïde
de la valuation $\abs \cdot$;
\item  $F^{\circ \circ}$ pour l'unique idéal maximal de $F^\circ$, qui n'est autre que $\{z\in F, \abs z < 1\}$.
\end{itemize}

\subsection{}
Soit $A$ un annéloïde. Nous noterons $\spv (A)$ le \emph{spectre valuatif}
de $A$, c'est-à-dire l'ensemble des classes d'équivalence de valuations sur $A$.
On peut définir de manière équivalente $\spv (A)$ comme l'ensemble des couples $(x,\abs \cdot)$
où $x$ est un point de $\spec A$ et $\abs \cdot$ une classe d'équivalence de valuations
sur $\kappa(x)$.
Si $\xi$ est un point de $\spv (A)$ correspondant à un couple $(x,\abs \cdot)$ le corpoïde valué
$(\kappa(x),\abs \cdot)$ sera noté $\kappa(\xi)$ et sa valuation structurale restera notée 
$\abs \cdot$ ; lorsqu'on voudra préciser que le morphisme d'évaluation 
$a\mapsto a(x)$ est vu comme à valeurs dans le corpoïde \emph{valué} 
$\kappa(\xi)$ on l'écrira $a\mapsto a(\xi)$. Grâce à cette convention on a $\xi(a)=\abs{a(\xi)}$ pour tout $a\in A$
et c'est la seconde écriture dont nous nous servirons la plupart du temps.

Si $E$ est un sous-ensemble de $A$ nous noterons $\spv(A,E)$ le sous-ensemble 
de $\spv(A)$ formé des points $\xi$ tels que $\abs {a(\xi)}\leq 1$ pour tout $a\in E$.

\subsection{}
Soit $F$ un corpoïde. Deux valuations sur
$F$ sont équivalentes si et seulement si elles ont même annéloïde ; un sous-annéloïde $A$ de $F$
est l'annéloïde d'une valuation
de $F$
si et seulement si il est local et maximal parmi les sous-annéloïdes locaux de $F$ pour la relation
 de domination. Ceci équivaut à demander que pour tout $x\in F$ on ait $x\in A$ ou $x\inv \in A$. 
 
On dispose d'une relation d'ordre
sur l'ensemble des valuations de $F$, qui se définit
\emph{mutatis mutandis}
comme dans la théorie classique ; rappelons
brièvement en quoi elle consiste. 

Donnons-nous donc deux valuations $\abs \cdot_1$ et $\abs \cdot_2$ sur $F$
dont on note $A$ et $B$
les annéloïdes respectifs, et $\kappa$ et $\lambda$ les corpoïdes résiduels respectifs. Les
assertions suivantes sont alors équivalentes : 
\begin{enumerate}[i]
\item il existe un sous-groupe convexe $H$ de $\abs{F\gpm}_1$ tel que $\abs \cdot _2$ soit (équivalente à) la composée de $\abs \cdot_1$ et de l'application quotient $\abs F_1\to (\abs {F\gpm}_1/H)\cup\{0\}$ ; 
\item l'annéloïde $B$ est égal à $A_{\mathfrak p}$ pour un certain idéal premier $\mathfrak p$ de $A$ ; 
\item il existe un annéloïde de valuation $\alpha$ de corpoïde des fractions $\lambda$ tel que $A$ soit l'image réciproque
de $\alpha$ par la flèche de réduction $B\to \lambda$.
\end{enumerate}
On dit alors que $\abs \cdot_1$ est plus fine que $\abs \cdot_2$, ou qu'elle la raffine ; on dit aussi que $\abs \cdot_2$ est plus grossière que $\abs \cdot _1$. 

Indiquons quelques relations entre les différents objets évoqués dans les équivalences ci-dessus : s'il existe $H$
comme dans (i), alors (ii) est vraie en prenant pour $\mathfrak p$ l'ensemble des éléments $a$ de $A$ tels que $\abs a_1\notin H$ ; s'il existe $\mathfrak p$ comme dans (ii), 
alors (iii) est vraie en prenant pour $\alpha$ l'annéloïde $A/\mathfrak p$ et (i) est vraie en prenant pour $H$ le groupe $\abs{B\gpm}_1$ ; et s'il existe 
$\alpha$ comme dans (iii) alors (ii) est vraie en prenant pour $\mathfrak p$ l'image réciproque de l'idéal maximal de $\alpha$. 

Soit $\abs \cdot$ une valuation sur $F$ ; utilisons les notations standard $F^\circ$ et $F^{\circ \circ}$
pour l'annéloïde de $\abs \cdot$ et pour l'idéal maximal de celui-ci.
La caractérisation (iii) de la relation de raffinement fournit une bijection entre l'ensemble des valuations de $F$ plus fines que $\abs \cdot$
(à équivalence près)
et l'ensemble des valuations sur le corpoïde $F^\circ/F^{\circ \circ}$. 
Soit maintenant $\mathscr V$ l'ensemble des valuations sur $F$ plus grossières que $\abs \cdot$,
muni de la relation de raffinement. La caractérisation (i)
(resp. (ii)) de cette dernière fournit une bijection croissante (resp. décroissante)
entre $\mathscr V$ et l'ensemble des sous-groupes convexes de $\abs{F\gpm}$ (resp. des 
idéaux premiers de $F^\circ$) ; le premier étant totalement ordonné par inclusion, il en va de même du second. 
Le cardinal 
de $\mathscr V\setminus\{\abs \cdot\}$ est appelé la \emph{hauteur}
de $\abs \cdot$.

\subsection{Valuations de Gauß}\label{ss-valgauss}
Soit $F$ un corpoïde valué, soit $r=(r_i)_{i\in I}$ une famille d'éléments de $D$, et soit $\gamma=
(\gamma_i)_{i\in I}$
une famille d'éléments d'un groupe ordonné contenant $\abs {F\gpm}$.
On note $F(r\backslash T/\gamma)$ le corpoïde 
$F(r\backslash T)$ muni de la \emph{valuation de Gauß
de paramètre $\gamma$} définie par la formule $\sum a_J T^J\mapsto \max\; \abs{a_J}\cdot\gamma^J$. 
Dans le cas où $D=\{1\}$ on écrira $F(T/\gamma)$ au lieu de $F(1\backslash T/\gamma)$. 

\subsection{}\label{ss-gauss-partiedeg1}
Soient $F, r$ et $\gamma$ comme ci-dessus. 
Le corps valué $F(r\backslash T/\gamma)^1$ est alors de la forme $F^1(\Theta/\delta)$. 
En effet, par le lemme de Zorn ou  une récurrence transfinie, 
on se ramène au cas où l'ensemble $I$ possède un unique élément $i$ et on écrit $r$ et $\gamma$
au lieu de $r_i$ et $\gamma_i$. On distingue deux cas. 

\begin{itemize}[label=$\diamond$]
\item Le cas où $r$ est d'ordre fini $n$ modulo $\mathfrak d(F\gpm)$. On choisit alors un élément $\lambda$ de $F\gpm$
tel que $r^n=\mathfrak d(\lambda)$; si l'on pose $\delta=\gamma\cdot \abs \lambda \inv$ et $\Theta=T/\lambda$ alors
$F(r\backslash T/\gamma)^1$ s'identifie à $F^1(\Theta/\delta)$. 
\item Le cas où $r$ n'est pas de torsion modulo $n$. On a alors  $F(r\backslash T/\gamma)^1=F^1$. 
\end{itemize}

\begin{rema}\label{rem-transcendant-f1}
Soit $F$ un corpoïde et soit $r$ une famille d'éléments de $D$. Le corps $F(r\backslash T)^1$ est alors transcendant pur sur $F^1$ : il suffit en effet d'appliquer \ref{ss-gauss-partiedeg1} en munissant $F$ de la valuation triviale et en prenant $\gamma=1$ ou, de manière plus raisonnable, d'en reprendre la preuve sans se préoccuper des valuations. 
\end{rema}

\subsection{Corpoïdes résiduels}
Soit $F$ un corpoïde muni d'une valuation $\abs \cdot$. Le \emph{corpoïde résiduel} de $F$ est le $(D\times\abs{F\gpm})$-corpoïde
\[\widetilde F:=\coprod_{(d,\gamma)\in D\times \abs{F\gpm}}
\{x\in F^d, \abs x \leq \gamma\}/\{x\in F^d, \abs x <\gamma\}.\]
Notons que $\widetilde F^{D\times \{1\}}$ est le corpoïde résiduel $F^\circ/F^{\circ \circ}$
de l'annéloïde local $F^\circ$. Si $x$ est un élément non nul de $F$ on notera
$\widetilde x$ son image dans $\widetilde F^{(\mathfrak d(x),\abs x)}$. 

Soit $L$ une extension valuée de $F$ et soit $(\ell_i)$ une famille d'éléments non nuls de $L$ ; pour tout $i$ posons $r_i=\mathfrak d(\ell_i)$ et $\gamma_i=\abs {\ell_i}$. Les faits suivants résultent immédiatement des définitions : 

\begin{enumerate}[1]
\item La famille $(\widetilde {\ell_i})$ est libre sur $\widetilde F$ si et seulement si 
$\abs{\sum a_i \ell_i}=\max\;\abs{a_i}\cdot \gamma_i$ pour toute famille $(a_i)$ d'éléments presque tous nuls de $F$ telle que les $a_i\ell_i$ soient tous de même degré. Si c'est le cas, la famille $(\ell_i)$ est alors libre sur $F$. 

\item La famille $(\widetilde {\ell_i})$ est
algébriquement indépendante sur $\widetilde F$ si et seulement si 
il existe un $F$-isomorphisme de corpoïdes valués $F(r_i\backslash T_i/\gamma_i)_i\simeq L$ envoyant $T_i$ sur $\ell_i$ pour tout $i$. 
\end{enumerate}

Il résulte de (1) que $\dim_{\widetilde F}\widetilde L\leq \dim_F L$. Mentionnons par ailleurs que 
$\dim_{\widetilde F}\widetilde L=ef$ où $e$ est l'indice de ramification de $L$ sur $F$, c'est-à-dire
l'indice de $\abs {F\gpm}$
dans $\abs{L\gpm}$, et où $f$ est son indice d'inertie,
c'est-à-dire la dimension du $\widetilde F^{D\times\{1\}}$-espace
vectoriel $\widetilde
L^{D\times\{1\}}$ (voir par exemple \cite{ducrosXXX}, 2.2.46.3 ; précisons qu'ici les indices et dimensions
sont à considérer comme des cardinaux).

\subsection{}\label{prolonge-gauss}
Soit $F$ un corpoïde valué et soit $L$ une extension algébrique de $F$. Soit $r=(r_i)_{i\in I}$ une famille d'éléments de $D$, soit $\gamma=(\gamma_i)$
une famille d'éléments d'un groupe abélien ordonné contenant $\abs{F\gpm}^\Q$, et soit $T=(T_i)$ une famille d'indéterminées. Munissons $L$ d'un prolongement
de $\abs \cdot$, encore noté $\abs \cdot$. L'unique valuation sur $L(r\backslash T)$ qui prolonge les valuations données sur $L$ et $F(r\backslash T/\gamma)$ est
alors la valuation de Gauß de paramètre $\gamma$.

En effet, fixons une telle valuation sur $L(r\backslash T)$, et soit $\Lambda$
son corpoïde résiduel. Pour tout $i$ notons
$\widetilde {T_i}$ l'image de $T_i$ dans $\Lambda^{(r_i,\gamma_i)}$. Comme les $\widetilde T_i$ sont 
algébriquement indépendants sur $\widetilde F$ et comme $\widetilde L$
est algébrique sur $\widetilde F$, les $\widetilde T_i$ sont 
algébriquement indépendants sur $\widetilde L$, ce qui permet de conclure. 

\subsection{Prolongements d'une valuation à une extension algébrique}
Soit $F$ un corpoïde valué et soit $L$ une extension algébrique de $F$. 

\subsubsection{}\label{sss-efn-1}
Supposons que $L^1=F^1$. Pour tout $x\in L\gpm$, il existe $N>0$ tel que $\mathfrak d(x)^N\in \mathfrak d(F\gpm)$ ; il en résulte
qu'il existe $\lambda\in F\gpm$ tel que $x^N/\lambda \in L^1=F^1$ ; par conséquent $x^N\in F\gpm$. Il s'ensuit que $\abs \cdot$
possède un unique prolongement $\abs \cdot $ à $L$, dont on munit ce
dernier. Soit $(\ell_i)$ une famille d'éléments de $L\gpm$ tels que les $\mathfrak d(\ell_i)$ constituent un système de représentants de $\mathfrak d(L\gpm)/\mathfrak d(F\gpm)$. Tout élément de $L\gpm$ possède une unique écriture
de la forme $a\ell_i$ pour un certain $i$ et un certain $a\in F\gpm$ ; 
les $\ell_i$ forment donc
une base de $L$ sur $F$,
et les $\widetilde{\ell_i}$ forment une base de $\widetilde L$ sur $\widetilde F$. 

\subsubsection{}\label{sss-efn-2}
Supposons maintenant que
$\mathfrak d(L\gpm)=\mathfrak d(F\gpm)$.
On a alors 
$L=L^1\otimes_{F^1}F$ et tout élément de $L\gpm$ a une écriture de la forme
$\ell x$ avec $\ell \in L^1$ et $x\in F$ ; il en résulte
que
tout prolongement à $L^1$ de la restriction
de $\abs \cdot$ à $F^1$ s'étend d'une unique manière en un prolongement 
de $\abs \cdot$ à $L$ ; si on le  note encore $\abs \cdot$ il est donné par la formule $\abs {\ell x}=\abs \ell \cdot\abs x$ et l'on a $\widetilde L=\widetilde {L^1}\otimes_{\widetilde {F^1}}\widetilde F$. 

\subsubsection{}\label{sss-efn-recap}
Plaçons-nous maintenant
dans le cas général et posons $\Lambda=L^{\mathfrak d(F\gpm)}$.
En considérant le dévissage $F\hookrightarrow \Lambda
\hookrightarrow L$
on déduit de \ref{sss-efn-1}
et \ref{sss-efn-2} les faits suivants : 
\begin{itemize}[label=$\diamond$]
\item la restriction à $L^1$ induit une bijection entre l'ensemble des extensions de $\abs \cdot$ à $L$ et celui
des extensions à $L^1$ de la restriction de $\abs \cdot$ à $F^1$ ;
\item si l'on munit $L$ d'un prolongement donné de $\abs \cdot$ on a 
\[ [\widetilde L:\widetilde \Lambda]=[L:\Lambda]\;\;\text{et}\;\;\widetilde \Lambda=\widetilde {\Lambda^1}\otimes_{\widetilde{F^1}}\widetilde F.\]
\end{itemize}

\subsubsection{}
Supposons maintenant que $L$ est finie sur $F$. Dans ce cas $L^1$ est finie sur $F^1$, et la restriction de $\abs \cdot$ à $F^1$ admet donc un ensemble fini $\{\abs \cdot_1, \ldots, \abs \cdot_r\}$ d'extensions
à $L^1$ ; pour tout $i$, il
résulte de \ref{sss-efn-recap}
que $\abs \cdot _i$ s'étend d'une unique manière en un prolongement de $\abs \cdot$ à $L$, qu'on note encore
$\abs \cdot_i$ ; soit $L_i$ (resp. $\Lambda_i$)
le corpoïde valué $(L,\abs \cdot_i)$ (resp. $(\Lambda, \abs \cdot_i)$). Notons  $e_i$ et $f_i$ les indices d'inertie et de ramification de $L^1_i$
sur $F^1$, et $e^\sharp_i$ et $f^\sharp_i$ les indices d'inertie et de ramification de $L_i$
sur $F$. On sait que $\sum e_i f_i\leq [L^1:F^1]$. En vertu de \ref{sss-efn-recap} on a 
\begin{eqnarray*}
\sum e^\sharp_i f^\sharp_i&=&\sum_i [\widetilde {L_i}:\widetilde F]\\
&=&\sum_i [\widetilde L_i:\widetilde {\Lambda_i}]\cdot[\widetilde {\Lambda_i}:\widetilde F]\\
&=&[L:\Lambda]\sum_i [\widetilde {\Lambda_i}:\widetilde F]\\
&=&[L:\Lambda]\sum_i [\widetilde {\Lambda^1_i}:\widetilde {F^1}]\\
&=&[L:\Lambda]\sum_i e_i f_i\\
&\leq&[L:\Lambda]\cdot[L^1:F^1]\\
&=&[L:\Lambda]\cdot[\Lambda:F]\\
&=&[L:F].\end{eqnarray*}
Ainsi, $\sum e_i^\sharp f^\sharp_i\leq [L:F]$, et l'on a plus précisément
\[\frac{\sum e_i^\sharp f^\sharp_i}{[L:F]}=\frac{\sum e_i f_i}{[L^1:F^1]}.\]
Le quotient $\frac{\sum e_i^\sharp f^\sharp_i}{[L:F]}$ est appelé le \emph{défaut}
de l'extension $F\hookrightarrow L$. On dit que cette extension est \emph{sans défaut}
si son défaut vaut $1$.

\paragraph{}\label{par-defaut-deg1}
Par ce qui précède, $F\hookrightarrow L$ est sans défaut si et seulement si
$F^1\hookrightarrow L^1$ est sans défaut. 

\paragraph{}\label{par-defaut-gauss}
Soit $r$
une famille d'éléments de $D$, soit $\Gamma$ un groupe abélien ordonné contenant
$\abs{L\gpm}$ et soit $\gamma$ une famille d'éléments de $L$, 
paramétrée par le même ensemble que $r$.  Il résulte de \ref{prolonge-gauss}  que le défaut de $F(r\backslash T/\gamma)\hookrightarrow L(r\backslash T)$ est égal à celui de $L$ sur $F$.

\subsection{}\label{ss-stabilite}
Soit  $F$ un corpoïde valué. On dit que $F$ est \emph{stable}, ou que sa valuation $\abs \cdot$ est stable, si toute extension finie de $F$ est sans défaut. 
Il découle de \ref{par-defaut-deg1}
que $F$ est stable si et seulement si $F^1$ est stable.

\subsubsection{}
Si $\abs \cdot$ est la composée d'une valuation $\abs\cdot^\star$ sur $F$ et d'une valuation $\abs \cdot_\star$ alors $\abs \cdot$ est stable si et seulement si $\abs \cdot^\star$ et $\abs \cdot_\star$ le sont :
on le démontre sans difficulté comme dans le cas classique, auquel on peut également se ramener puisque la stabilité se teste en degré 1.

\subsubsection{}
On suppose que $D=\{1\}$
et que la valuation de $F$ est de hauteur $1$. Si $F$ est stable, son complété $\widehat F$ est stable
d'après  \cite{bosch-g-r1984}, 3.6.2, Prop. 3 ; précisons
que la définition de stabilité dans cet ouvrage (\emph{op. cit.}, 3.6.1, Def. 1) n'est pas \emph{a priori}
la même que la nôtre, mais les deux coïncident en fait d'après \emph{op. cit.}, 3.6.2, Prop. 5. 

\subsubsection{}\label{sss-gauss-stable1}
Supposons toujours que $D=\{1\}$ (mais on ne fait plus
d'hypothèse sur la hauteur de la valuation). Soit 
$\Gamma$ un groupe abélien ordonné contenan t $\abs {F\gpm}$ et soit $\gamma=(\gamma_i)_{i\in I}$
une famille \emph{finie}
d'éléments de
$\Gamma$. 
Si $F$ est stable, alors $F(T/\gamma)$ est stable : cela a été démontré par Kuhlmann,
voir \cite{kuhlmann2010} ; pour des résultats partiels
antérieurs, voir \cite{gruson1968}, \cite{ohm1989}, \cite{temkin2010},
\cite{teissier2014} ; pour une preuve plus récente fondée sur les travaux de Hrushovski et Loeser, voir
\cite{ducros2014}.

\subsubsection{}\label{sss-gauss-stable2}
On ne suppose plus que $D=\{1\}$ ; nous allons brièvement expliquer
comment étendre le résultat de \ref{sss-gauss-stable1}
dans le contexte gradué, et pour une famille $\gamma$ éventuellement infinie.

Soit donc $r=(r_i)_{i\in I}$ une famille d'éléments de $D$, soit $\Gamma$ un groupe abélien ordonné
contenant $\abs{F\gpm}^\Q$ et soit $\gamma=(\gamma_i)$ une famille d'éléments de
$\Gamma$ ; soit $T=(T_i)$ une famille d'indéterminées. Supposons que $F$ est stable, 
et montrons que $F(r\backslash T/\gamma)$
est stable. Comme la stabilité se teste en degré 1, 
il résulte de \ref{ss-notation-tgamma}
qu'on peut se ramener au cas non gradué. 
Soit $L$ une extension finie de $F(T/\gamma)$. Il existe un sous-ensemble fini $J$ de $I$ 
et une extension finie $L_0$ de $F(T_i/\gamma_i)_{i\in J}$ telle que $L=L_0(T_i)_{i\in I\setminus J}$. Par le cas fini énoncé au \ref{sss-gauss-stable1}, $L_0$ est une extension sans
défaut de  $F(T_i/\gamma_i)_{i\in J}$. Il s'ensuit en vertu de \ref{par-defaut-gauss}
que $L=L_0(T_i)_{i\in I\setminus J}$ est une extension sans défaut du corps valué
$F(T/\gamma)=F(T_i/\gamma_i)_{i\in J}(T_i/\gamma_i)_{i\in I\setminus J}$.

\subsection{Platitude sur un annéloïde de valuation}
\label{ss-plat-annelval}
Soit $F$
un corpoïde valué. Un $F^\circ$-module  est plat si et seulement si il est sans torsion
(\cite{ducros2018}, preuve du lemme 5.2.1). 
Par ailleurs, toute $F^\circ$-algèbre
plate et de type fini est de présentation finie. En effet, soit $A$ une telle algèbre. Choisissons une famille finie
$r=(r_1,\ldots,r_n)$ d'éléments de $D$ et une surjection 
$F^\circ[r\backslash T]\to A$. Soit $I$ le noyau de cette surjection ; nous allons montrer que $I$ est de type fini.

Expliquons tout d'abord pourquoi il suffit 
de prouver que c'est le cas après extension des scalaires à $L^\circ$ où $L$ est une extension valuée de $F$ bien choisie. 
Écrivons $I$ comme la réunion filtrante d'une famille $(I_\lambda)$ d'idéaux de type fini. Par platitude, 
$I_{L^\circ}$ et les $I_{\lambda, L^\circ}$ sont des idéaux de $L^\circ[r\backslash T]$. Supposons $I_{L^\circ}$ de type fini ; il coïncide alors avec $I_{\lambda,L^\circ}$ pour un certain $\lambda$. Par conséquent, $(I/I_\lambda)_{L^\circ}$ est nul, et il suffit de montrer que cela entraîne la nullité de $(I/I_\lambda)$ ; nous allons plus précisément vérifier que tout sous$-F^\circ$-module de type fini de  $(I/I_\lambda)$  est nul, ce qui permettra de conclure. Soit $J$ un tel sous-module. Par platitude, $J_{L^\circ}$ s'injecte dans $(I/I_\lambda)_{L^\circ},$, et est donc nul. Il s'ensuit que $J_{L^\circ/L^{\circ \circ}}$ est nul, puis que $J_{F^\circ/F^{\circ \circ}}$ est nul (car $F^\circ/F^{\circ \circ}
\hookrightarrow L^\circ/L^{\circ \circ}$ est une extension de corpoïdes), et l'on conclut à l'aide du lemme de Nakayama
(gradué). 

Nous allons appliquer cette remarque en prenant pour $L$ le corpoïde valué
$F(r_1\backslash T_1/1,\ldots, r_n\backslash T_n/1, (1\backslash \Theta_\gamma/\gamma)_\gamma)$ où $\gamma$ parcourt le groupe $\abs{F\gpm}$.
Par construction, les $r_i$ appartiennent à $\mathfrak d(L)$
et $\abs{L\gpm}$ est égal à $\abs{(L^1)\gpm}$. En étendant les scalaires à $L^\circ$ on se ramène donc au cas où
les $r_i$ appartiennent à $\mathfrak d(F)$ (puis par renormalisation au cas où ils sont tous égaux à $1$), et 
où
$\abs{F\gpm}$ est égal à $\abs{(F^1)\gpm}$. Mais dans ce cas tout élément $f$ de $I$ est de la forme 
$ag$ où $a$ appartient à $(F^\circ)\gpm$ et où $g$ est de degré $1$. L'idéal $I$ est donc engendré par le noyau de 
la surjection $(F^\circ)^1[T]\to A^1$, qui est de type fini par le cas classique (non gradué) 
de l'assertion à démontrer, \emph{cf.} \cite{nagata1966} -- mentionnons que ce résultat a été étendu des anneaux de valuation
aux anneaux intègres \emph{quelconques}
par Raynaud et Gruson (\cite{raynaud-g1971}, cor. 3.4.7). 

\begin{lemm}\label{lem-plat-dimension}
Soit $Y\to X$ un morphisme plat, dominant et de présentation finie entre schémoïdes intègres, 
et soit $d$ la dimension de sa fibre générique. Les fibres de $Y\to X$ sont alors
purement de dimension $d$.
\end{lemm}

\begin{proof}
Par une extension des scalaires à un corpoïde de
la forme $F(r\backslash T)$ avec $r$ convenable, on peut supposer $Y\to X$ est induit
par extension des scalaires de $F^1$ à $F$ par un morphisme de schémas usuels ; il suffit
maintenant
d'établir le résultat pour ce dernier morphisme, c'est-à-dire dans le cadre non gradué. 
Mais il est alors bien connu ; combiner par exemple la proposition 13.2.3 de \cite{ega43}
et le corollaire 6.1.2 de \cite{ega42}.
\end{proof}

\begin{coro}\label{coro-controle-dimfibres}
Soit $F$ un corpoïde valué et soit $X$ un $F^\circ$-schémoïde de type fini.
tel que le morphisme $X\to \spec F^circ$ soit générisant. 
Soit
$\xi$ un point de $\spec F^\circ$, soit $Y$ une composante irréductible de $X_\xi$, soit $d$
sa dimension et soit 
$\tau$ son point générique. Il existe un ouvert $U$ de $X$ contenant $\tau$ tel que les fibres de 
$U\to X$ soient toutes purement de dimension $d$. 
\end{coro}

\begin{proof} 
L'énoncé à montrer est insensible aux phénomènes de nilpotence ; on peut
donc supposer $X$ réduit. 
Soit $\eta$ le point générique de $\spec F^\circ$. Comme $X\to \spec F^\circ$ est générisant, il existe
une générisation de $\tau$ sur la fibre générique $X_\eta$ ; autrement dit, l'ensemble $\mathscr Z$
des composantes irréductibles $Z$
de $X_\eta$ telles que $\tau \in \overline Z$ est non vide. 
Soit $E$ la réunion des composantes irréductibles $Z$ de $X_\eta$ qui n'appartiennent pas à $\mathscr Z$. Par construction, 
$\tau$ n'appartient pas au fermé $\overline E$ de $X$ ;
posons $U=X\setminus \overline E$, et montrons que les fibres de $U\to X$ sont purement de dimension $d$. 

Soit $Z\in \mathscr Z$. Le $F^\circ$-schémoïde
de type fini
$\overline Z$ muni de sa structure réduite est sans $F^\circ$-torsion, et partant plat
(et dès lors de présentation finie, \emph{cf.} \ref{ss-plat-annelval}). 
Le lemme \ref{lem-plat-dimension} assure que les fibres de $\overline Z\to X$ sont purement de dimension $\dim Z$ ; or
$Y$ est une composante irréductible de $\overline Z_\xi$ puisque $\tau\in \overline Z$ ; en conséquence, $\dim Z=d$ et
les fibres de $\overline Z\to X$ sont toutes purement de dimension $d$. 

Comme $X\to \spec F^\circ$
est générisant, tout point de 
$U$ appartient à $\overline Z$ pour au moins une composante irréductible $Z$ de $X_\eta$ ;
puisque $U$ est le complémentaire de $\overline E$, la composante $Z$ appartient nécessairement à $\mathscr Z$. En conséquence, 
$U$ est un ouvert du schémoïde réduit $\bigcup_{Z\in \mathscr Z}\overline Z$. Les fibres de $\overline Z\to X$
étant toutes purement de dimension $d$ quel que soit $Z\in \mathscr Z$, il en va de même de celles de $(\bigcup_{Z\in \mathscr Z}\overline Z)
\to X$, puis de celles de $U\to X$.
\end{proof}

\begin{defi}
Soit $R$ un annéloïde intègre de corpoïde
des fractions $K$
et soit $A$ une $R$-algèbre 
sans $R$-torsion et réduite. Le \emph{normalisé faible}
de $A$ est la fermeture intégrale de $A$ dans $A_K$.
Nous dirons
que la $R$-algèbre $A$ est \emph{faiblement normale}
si elle est égale à son normalisé faible. 

Ces notions permettent de définir par recollement la notion de normalisé faible
d'un $R$-schémoïde réduit sans $R$-torsion, puis celle
de $R$-schémoïde faiblement normal.
\end{defi}

\begin{lemm}\label{lem-entier-valuation}
Soit $F$ un corpoïde valué, soit $A$ une $F^\circ$-algèbre plate
et soit 
$a$ un élément de $A_F$. Les assertions suivantes sont équivalentes : 
\begin{enumerate}[i]
\item L'élément $a$ est entier sur $A$ ; 
\item on a $\abs {a(\xi)}\leq 1$ pour tout $\xi \in \spv(A_F, A)$.
\end{enumerate}
\end{lemm}

\begin{proof}Il est clair que (i)$\Rightarrow$(ii). Supposons maintenant que (ii) soit vraie. 
Posons $r=\mathfrak d(a)$.  Soit $\A^{1,r}_A$ le schémoïde $\spec A[r\backslash T]$
et soit $\P^{1,r}_A$ le recollement de $\A^{1,r}_A$ et $\A^{1,r\inv}_A=\spec A[r\inv \backslash S]$ 
le long de l'isomorphisme $\mathrm D(T)\simeq \mathrm D(S), S\mapsto T\inv$.
Soit $Z$ le sous-schémoïde fermé de
$\A^{1,r}_{A_F}:=\spec A_F[r\backslash T]$ défini par l'idéal $(T-a)$
et soit $\overline Z$ son adhérence schématique
dans $\P^{1,r}_A$, c'est-à-dire le sous-schémoïde fermé de $\P^{1,r}_A$ dont la trace sur $\A^{1,r}_A$ est définie par le noyau de l'évaluation en $a$, et dont la trace sur $\A^{1,r\inv}_A$ est définie par le noyau de
$A[r\inv\backslash S]\to A_F[r\inv \backslash S]/(aS-1)$. 

Soit $z\in \overline Z$. Par construction $\overline Z$ est sans $F^\circ$-torsion, et est dès lors plat sur $F^\circ$. Un morphisme plat étant générisant (la preuve est la même \emph{mutatis mutandis} que dans le cas non gradué), $z$ possède une générisation $\zeta$ située sur $Z$. Le choix d'un annéloïde de valuation dominant $\mathscr O_{\overline{\{\zeta\}}, z}$ (où $\overline{\{\zeta\}}$
est muni de sa structure réduite) fournit un point $\xi$ de $\spv(A_F, A)$ ; comme
$\abs {a(\xi)}\leq 1$
d'après l'hypothèse (ii), le point $z\in \P^{1,r}_A$ n'est pas le «point à l'infini de sa fibre» (défini par l'équation $S=0$). Par conséquent, la section à l'infini de  $\P^{1,r}_A\to \spec A$ ne rencontre pas $\overline Z$. Celui-ci est donc à la fois «projectif» et affine sur $\spec A$, et partant fini (on peut pour le voir se ramener au cas non gradué par le procédé général décrit au \ref{ss-methode-schemas}). Il s'ensuit que $a$ est entier sur $A$. 
\end{proof}

\begin{lemm}\label{lem-normal-kr}
Soit $F$ un corpoïde valué et soit $r$ une famille
d'éléments de $D$. Soit $X$ un $F^\circ$-schémoïde plat
et soit $Y$ le normalisé faible de $X$. 
Le normalisé faible du $F(r\backslash T/1)^\circ$-schémoïde
$X_{F(r\backslash T/1)^\circ}$
s'identifie à $Y_{F(r\backslash T/1)^\circ}$. 
\end{lemm}

\begin{proof}
On peut supposer
que les schémoïdes $X$ et $Y$
sont affines ; posons $A=\mathscr O_X(X)$ et $B=\mathscr O_Y(Y)$.
Traitons tout d'abord le cas où $B=A$,
c'est-à-dire celui où la $F^\circ$-algèbre $A$ est faiblement normale ; il s'agit
alors de prouver que la $F(r\backslash T/1)^\circ$-algèbre $A_{F(r\backslash T/1)^\circ}$ est  
faiblement normale. 
Soit $a$ un élément 
de $A_{F(r\backslash T/1)}$ entier sur $A_{F(r\backslash T/1)^\circ}$ ; ceci revient
en vertu du lemme \ref{lem-entier-valuation} à demander que 
$\abs{a(\xi)}\leq 1$ pour tout  point
$\xi$ appartenant à $\spv(A_{F(r\backslash T/1)}, A_{F(r\backslash T/1)^\circ})$. Notre but
est de montrer que $a\in A_{F(r\backslash T/1)^\circ}$.
Écrivons \[a=\frac{\sum a_I T^I}{\sum \lambda_I T^I}\] où les $a_I$ appartiennent à $A_F$ et les $\lambda_I$
appartiennent à $F$ (et sont non tous nuls). Quitte à multiplier numérateur et dénominateur par un scalaire convenable, on peut supposer que $\max \abs{\lambda_I}=1$, c'est-à-dire que $\lambda:=\sum \lambda_I T^I$ appartient à 
$F(r\backslash T/1)^{\circ \times}$. Mais dans ce cas $\abs{\lambda(\xi)}=1$
pour tout $\xi\in \spv(F(r\backslash T/1),F(r\backslash T/1)^\circ)$,
ce qui entraîne que 
 $\abs{\lambda(\xi)}=1$
pour tout $\xi\in \spv(A_{F(r\backslash T/1)}, A_{F(r\backslash T/1)^\circ})$ ; on peut donc remplacer $a$ par $\lambda a$ et se ramener ainsi
au cas où $a$ est de la forme $\sum a_I T ^I$, les $a_I$ étant des éléments de $A_F$. 
Soit $\xi\in \spv (A_F, A)$ et soit $\eta$ l'antécédent
de $\xi$
sur $\spv(A_{F(r\backslash T/1)}, A_{F(r\backslash T/1)^\circ})$ induit par
le morphisme naturel $A_{F(r\backslash T/1)}\to \kappa(\xi)(r\backslash T/1)$. On a $\abs{a(\eta)}\leq 1$ par hypothèse
sur $a$. Par définition de $\eta$, cela signifie que $\abs{a_I(\xi)}\leq 1$ pour tout $I$. Ceci valant quelque soit $\xi\in \spv (A_F,A)$,
chacun des $a_I$ est entier sur $A$ ; puisque $A$ est $F^\circ$-normale, 
ceci entraîne que chacun des $a_I$ appartient à $A$, puis que $a\in A_{F(r\backslash T/1)^\circ}$. Par conséquent, $A_{F(r\backslash T)^\circ}$ est 
$F(r\backslash T/1)^\circ$-normale. 

Traitons maintenant le cas général. Par construction, $B_{F(r\backslash T/1)^\circ}$ est un sous-annéloïde de
$B_{F(r\backslash T/1)}=A_{F(r\backslash T/1)}$ qui contient $A_{F(r\backslash T/1)^\circ}$ et est entier sur celui-ci ; par le cas
particulier déjà traité, la $F(r\backslash T/1)^\circ$-algèbre $B_{F(r\backslash T/1)^\circ}$ est faiblement
normale, ce qui achève la démonstration. 
\end{proof}

\begin{lemm}\label{lem-normal-trivgrad}
Soit $F$ un corpoïde valué et soit $A$ une $F^\circ$-algèbre
fidèlement plate. 
Supposons que $\mathfrak d(A_{\neq 0})=\mathfrak d(F\gpm)=\mathfrak d(F^{\circ \gpm})$ et que ce groupe est divisible. 
Soit $B$ le $F^{\circ 1}$-normalisé de $A^1$. Le produit tensoriel
$B\otimes_{F^{\circ 1}}F^\circ$ s'identifie au $F^\circ$-normalisé
de $A$. 
\end{lemm}

\begin{proof}
En vertu de nos hypothèses
on a $A=A^1\otimes_{F^{\circ 1}}F^\circ$.
Exactement comme dans la preuve du lemme \ref{lem-normal-kr}
ci-dessus, il suffit de traiter le cas où $B=A^1$. 
Supposons donc que $A^1$ est $F^{\circ 1}$-normale, et soit $C$
une $A$-algèbre entière et $F^\circ$-plate telle que $A_F\to C_F$
soit un isomorphisme. 
Comme $\mathfrak d(A_{\neq 0})$ est divisible on a 
\[\mathfrak d(C_{\neq 0})=\mathfrak d(A_{\neq 0})=\mathfrak d(F\gpm)=\mathfrak d(F^{\circ \gpm})\]
si bien que $C=C^1\otimes_{F^{\circ 1}}F=C^1\otimes_{A^1}A$. 
Comme $C^1$ est $F^{\circ 1}$-normale, $A^1\to C^1$ est un isomorphisme. 
Par conséquent, $A\to C$ est un isomorphisme et $A$ est $F^\circ$-normale. 
\end{proof}

\begin{theo}[d'après Grauert-Remmert et Temkin]\label{theo-apres-temkin}
Soit $F$ un corpoïde valué et soit $X$ un $F^\circ$-schémoïde plat
et
de
type fini. 
\begin{enumerate}[1]
\item Si le $F^\circ$-schémoïde $X$ est à fibres
géométriquement réduites, il est
faiblement normal. 
\item Supposons $F$ stable et de hauteur finie, $\abs{F\gpm}$ divisible
et $X_F$ réduit. 
\begin{enumerate}[b]
\item Si le $F^\circ$-schémoïde $X$ est faiblement normal
il est à fibres réduites. 
\item Le normalisé faible de $X$ est fini sur $X$. 
\end{enumerate}
\end{enumerate}
\end{theo}

\begin{proof}
Les deux assertions sont locales sur $X$,
ce qui permet de supposer que ce 
dernier est affine.

Commençons par l'assertion (1).
Les lemmes \ref{lem-normal-kr}
et \ref{lem-normal-trivgrad}
permettent de se ramener au cas non gradué, 
dans lequel (1) vaut d'après \cite{temkin2010}, Prop.
3.5.2. 

Supposons maintenant $F$ stable et de hauteur finie, $\abs{F\gpm}$ divisible
et $X_F$ réduit. Nous allons montrer que le normalisé
faible $Y$ de 
$X$ est fini sur $X$ et à fibres réduites sur $\spec F^\circ$,
ce qui montrera à la fois (2a) et (2b). Là encore, les lemmes 
\ref{lem-normal-kr} et \ref{lem-normal-trivgrad}
permettent de se ramener au cas non gradué ; l'assertion en vue vaut alors
 d'après \cite{temkin2010}, étapes 2, 3 et 4 
 de la preuve du théorème 3.5.5. Précisons que Temkin travaille avec un corps
 valué 
 $F$ séparablement clos, mais son raisonnement s'applique en fait tout aussi bien lorsque $F$
 est stable à groupe des valeurs divisible, en raison des faits suivants : 
 si $F$ est stable et de hauteur $1$, son complété $\widehat F$ est encore stable ; le 
 théorème de finitude de Grauert et Remmert (ingrédient crucial
 de l'étape 2 de la preuve de Temkin) vaut encore pour les corps stables à groupe des valeurs divisibles (\cite{bosch-g-r1984}, 6.4.1, Corollary 5). 
\end{proof}

%

\begin{coro}[Théorème de la fibre réduite, version graduée et valuée]
\label{coro-fibre-reduite}
Soit $F$ un corpoïde valué et soit $X$ un $F^\circ$-schémoïde
plat de type fini telle que $X_F$
soit géométriquement réduit. Il existe une extension valuée finie $L$ de $F$ telle que le
normalisé faible $Y$ du $L^\circ$-schéma $X_{L^\circ}$ soit fini sur $X_{L^\circ}$ et tel que 
les fibres de $Y\to \spec L^\circ$ soient géométriquement réduites et à composantes irréductibles géométriquement irréductibles. Pour toute
extension valuée $\Lambda$ de $L$, le schémoïde
$Y_{\Lambda^\circ}$ est le normalisé faible
du $\Lambda^\circ$-schémoïde $X_{\Lambda^\circ}$. 
\end{coro}

\begin{proof}
Comme $X$ est plat, il est de présentation finie et un argument
de passage à la limite permet de supposer que la valuation de $F$ est de hauteur finie. Soit $K$ une clôture algébrique de $F$ munie d'un prolongement quelconque de la valuation de $F$. 
L'assertion (2)
du théorème \ref{theo-apres-temkin}
assure
que le normalisé faible $Z$ du $K^\circ$-schémoïde $X_{K^\circ}$ est fini sur $X_{K^\circ}$
et que les fibres
de $Z\to \spec K^\circ$ sont réduites. Il existe un sous-corpoïde $L$ de $K$
fini sur $F$ et un $L^\circ$-schémoïde plat et de type fini $Y$, fini sur $X_{L^\circ}$, tel
que $Y_L=X_L$ et $Z=Y_{K^\circ}$ ; les fibres de $Y\to \spec L^\circ$
sont géométriquement réduites. 
Il résulte dès lors de l'assertion (1) du théorème \ref{theo-apres-temkin}
que pour toute extension valuée $\Lambda$ de $L$, le schémoïde
$Y_{\Lambda^\circ}$ est le normalisé du $\Lambda^\circ$-schémoïde $X_{\Lambda^\circ}$ ; c'est
en particulier le cas lorsque $\Lambda=L$. 
Enfin comme la valuation de $F$ est de hauteur finie le schéma $\spec L^\circ$ est ensemblistement fini ; quitte à agrandir $L$, on peut donc supposer que les composantes irréductibles des fibres de $Y\to \spec L^\circ$ sont géométriquement irréductibles. 
\end{proof}

\begin{prop}\label{prop-schema-decoupage}
Soit $F$ un corpoïde valué et soit $X$ un
$F^\circ$-schémoïde plat, de présentation finie et
à fibres réduites. 

\begin{enumerate}[1]
\item Soient $\eta$ et $\xi$
deux points de $\spec F^\circ$
avec $\xi\in \overline{\{\eta\}}$. 
Soient $U_1,\ldots, U_m$
les composantes connexes de $X_\eta$. Les $(\overline{U_i})_\xi$ sont
alors des ouverts fermés deux à deux disjoints de $X_\xi$ et 
$X_\xi=\coprod_i (\overline{U_i})_\xi$. 
\item Supposons que la valuation de $F$ est de hauteur finie. Il existe alors un recouvrement ouvert fini $(X_i)$ de $X$ tel que pour tout $\xi\in \spec F^\circ$ la fibre $X_{i,\xi}$ est ou bien vide, ou bien une composante connexe de $X_\xi$. 
\end{enumerate}
\end{prop}

\begin{proof}
Montrons
tout d'abord (1). Quitte à remplacer $F^\circ$ par $F^\circ/\mathfrak p$, où $\mathfrak p$ est l'idéal premier
de $F^\circ$ correspondant à $\eta$,
on peut supposer que $\eta$ est le point générique de $\spec F^\circ$.
Les $(\overline{U_i})_\xi$ sont des fermés de $X_\xi$ par définition, et
ils recouvrent $X_\xi$ car $X\to \spec F^\circ$ est plat et \emph{a fortiori}
générisant. Il suffit pour conclure de prouver que les $(\overline{U_i})_\xi$ 
sont deux à deux disjoints. Fixons donc deux indices $i$ et $j$ différents. 
Soit $f$ la fonction sur $X_\eta$ telle que $f|_{U_i}=1$ et $f|_{X_F\setminus U_i}=0$. La fonction $f$ annule le polynôme $T^2-T$
appartenant à $F[T]$
ce qui entraîne, puisque $X$ est faiblement normal sur $F^\circ$ en vertu du théorème \ref{theo-apres-temkin}, 
que $f$ se prolonge en une fonction $\overline f$ sur $X$.
On a alors
$\overline f(x)=1$ pour tout $x\in \overline U_i$ et $\overline f(x)=0$ pour tout $x\in \overline{X\setminus U_i}$, et en particulier pour tout 
$x\in \overline {U_j}$. En conséquence $\overline {U_i}\cap \overline {U_j}=\emptyset$.  

Montrons maintenant (2).
On suppose donc
la valuation de $F$ de hauteur finie. Soit $\xi \in \spec F^\circ$ et soit $x\in X_\xi$. 
Pour toute générisation $\eta$ de $\xi$ sur $\spec F^\circ$,
il existe d'après (1) une unique composante connexe $U_\eta$ de $X_\eta$ dont l'adhérence rencontre la composante connexe de $x$ dans $X_\xi$
(et $\overline U_\eta$ \emph{contient} en fait cette composante). 
Notons $U$ la réunion des $U_\eta$ où $\eta$ parcourt l'ensemble des générisations de $\xi$ dans $\spec F^\circ$. C'est une partie
constructible de $X$. Nous allons montrer que $U$ est ouverte, ce qui permettra de conclure puisque $X$ est quasi-compact, puisque $x\in U$ et puisque toute fibre de $U\to \spec F^\circ$ est par construction une composante connexe d'une fibre de $X\to \spec F^\circ$. 
Pour montrer que $U$ est ouverte, il suffit de s'assurer qu'elle est stable par générisation. Soit donc $\eta$ une générisation de $\xi$ dans $\spec F^\circ$ et soit $y\in U_\eta$ ; soit $z$ une générisation de $y$ sur $X$, dont on note $\zeta$ l'image sur $\spec F^\circ$. Soit $\Omega$ la composante connexe de $z$ dans $X_\zeta$.  Il résulte de (1) que $\overline \Omega_\eta$ est une réunion de composantes connexes de $X_\eta$. Comme $y$ est adhérent à $z$, il s'ensuit que  $\overline \Omega_\eta$ contient $U_\eta$, ce qui entraîne que $\overline \Omega_\xi$ contient $U_\xi$ ; il vient $\Omega=U_\zeta$, puis $z\in U$.
\end{proof}

\section{Extension de la théorie de Grauert
et Remmert}
\label{s-grg}

\subsection{Conventions générales}\label{ss-conv-gen}
Dans cette section les valuations et normes sont sauf mention
expresse du contraire à valeurs réelles. 
Un \emph{polyrayon}
est une famille finie de réels strictement positifs. 

Nous utiliserons librement la géométrie analytique au sens de Berkovich (\cite{berkovich1990}, \cite{berkovich1993}).
En ce qui concerne les notions d'espace réduit ou géométriquement réduit dans ce contexte, nous renvoyons le lecteur au chapitre 2 de \cite{ducros2018}, et plus particulièrement au paragraphe 2.6.9. 

\subsubsection{}
On fixe pour toute la suite
un corps ultramétrique complet $k$.
Un polyrayon est dit \emph{$k$-libre}
si c'est une famille libre du $\Q$-espace
vectoriel $\R_+\gpm/\abs{k\gpm}^\Q$. Si $r$ est un polyrayon $k$-libre nous noterons $k_r$ le corps $k\{T/r, rT\inv\}$. 

\subsubsection{}
On fixe également un sous-groupe $\Gamma$ de $\R_+\gpm$
tel que $\Gamma\cdot \abs{k\gpm}\neq \{1\}$ (autrement dit, $\Gamma$
est non trivial si la valeur absolue de $k$ est triviale). Nous allons faire quelques brefs rappels sur la théorie des espaces $k$-analytiques $\Gamma$-stricts ; le lecteur trouvera
davantage de détails et des justifications de nos assertions à la section 
3.1 de \cite{ducros2018}. 

Une algèbre $k$-affinoïde $A$ est dite \emph{$\Gamma$-stricte} si elle peut s'écrire
comme un quotient de $k\{T_1/r_1,\ldots, T_n/r_n\}$ pour une 
certaine famille $(r_1,\ldots,r_n)$  d'éléments de $\Gamma$, et il suffit pour que ce soit le cas que $A$
puisse s'écrire comme un quotient de de $k\{T_1/r_1,\ldots, T_n/r_n\}$ 
où les $r_i$ appartiennent à $(\abs k\gpm\cdot \Gamma)^\Q$.
Une algèbre $k$-affinoïde est $\Gamma$-stricte si et seulement si sa semi-norme spectrale prend ses valeurs dans $(\abs k\gpm\cdot \Gamma)^\Q\cup\{0\}$.

Un espace $k$-affinoïde est dit $\Gamma$-strict si son algèbre de fonctions analytiques
est $\Gamma$-stricte. Un espace $k$-analytique $X$ est dit $\Gamma$-strict s'il possède un G-recouvrement par des domaines affinoïdes $\Gamma$-stricts et si l'intersection de deux domaines affinoïdes $\Gamma$-stricts de $X$ est
G-recouverte par des domaines affinoïdes $\Gamma$-stricts (cette définition est compatible avec la précédente lorsque $X$ est affinoïde). Tout point d'un bon espace $k$-analytique $\Gamma$-strict possède une base de voisinages affinoïdes $\Gamma$-stricts
(c'est pour assurer cette dernière propriété que nous imposons la condition $\Gamma\cdot \abs{k\gpm}
\neq \{1\}$). 

Notons deux cas particuliers extrêmes : si $\Gamma=\R_+\gpm$ tout espace $k$-analytique est $\Gamma$-strict (le lecteur qui souhaite ignorer la question des paramètres réels de définition des espaces analytiques peut donc supposer $\Gamma=\R_+\gpm$) ; si $\Gamma=\{1\}$ (ce qui n'est possible que si
$k$ n'est pas trivialement valué), les espaces $k$-analytiques $\Gamma$-stricts sont les espaces strictement $k$-analytiques.

\subsubsection{}\label{def-annresiduel-gamma}
Si $A$ est un anneau muni d'une semi-norme sous-multiplicative
$\norm \cdot$ 
nous noterons $\widetilde A$ le $\Gamma$-annéloïde 
résiduel de $A$, c'est-à-dire 
\[\coprod_{\gamma \in \Gamma}\{a\in A, \norm a\leq \gamma\}
/\{a\in A, \norm  a <\gamma\}.\]
Si $a$ est un élément de $A$ et si $s$ est
un élément de $\Gamma$
supérieur ou égal à $\norm a$ nous noterons $\widetilde a^s$ l'image de $a$ dans $\widetilde A^s$ ; si $\norm a\in \Gamma$
nous écrirons $\widetilde a$ au lieu de $\widetilde a^{\norm a}$.

En particulier, nous noterons $\widetilde k$
le $\Gamma$-corpoïde résiduel de $k$, c'est-à-dire 
 \[\coprod_{\gamma \in \Gamma}\{x\in k, \abs x\leq \gamma\}
/\{x\in k, \abs x <\gamma\}.\] Si  $\Gamma=\R_+\gpm$ c'est le corpoïde résiduel
considéré par Temkin ; si $\Gamma=\{1\}$ c'est
son corps résiduel au sens classique. 

Si $A$ est une algèbre
$k$-affinoïde la notation
$\widetilde A$
désignera toujours
le $\Gamma$-annéloïde résiduel
de $A$ \emph{relatif à sa semi-norme spectrale} ; nous n'utiliserons cet objet que lorsque $A$ est $\Gamma$-stricte, car il est essentiellement inexploitable sinon. Notons que comme la semi-norme
spectrale commute aux puissances,
$\widetilde A$ est réduit. Si $X$ désigne l'espace $\mathscr M(A)$, 
le schémoïde $\spec \widetilde A$ sera noté $\widetilde X$. On dispose d'une application de spécialisation de $X$ vers $\widetilde X$
qui est anti-continue et surjective. La dimension de $\widetilde X$ est égale à celle de $X$ -- pour un énoncé un peu plus précis dont nous aurons besoin,
voir le lemme \ref{lem-dim-xtilde} en appendice.

\subsection{}Un $k$-espace de Banach est un $k$-espace vectoriel $M$
muni d'une norme ultramétique $\norm\cdot $ telle que $\norm {\lambda m}
=\abs \lambda \cdot \norm m$ pour tout $(\lambda, m)\in k\times M$, et pour laquelle il
est complet. 
Une $k$-algèbre de Banach est une $k$-algèbre $A$ munie d'une norme
sous-multiplicative qui en fait un $k$-espace de Banach. En pratique les seules $k$-algèbres de Banach auxquelles nous aurons affaire seront
des 
algèbres affinoïdes, et un morphisme entre deux telles algèbres
sera par définition un morphisme \emph{borné}
de $k$-algèbres.

Un module de Banach sur une $k$-algèbre de Banach $A$
est un $A$-module $M$ muni d'une norme $\norm \cdot$
qui en fait un $k$-espace de Banach et est telle que $\norm {am}\leq \norm a \cdot \norm m$ pour 
tout
$(a,m)\in A\times M$.

\subsection{}Soit $A$ une $k$-algèbre de Banach.

\subsubsection{}
Soit $r$ un réel $>0$. On notera $A(r)$ le $A$-module
de Banach dont le $A$-module sous-jacent est $A$
et dont la norme est $a \mapsto r\norm a$. 

Si $M$ est un $A$-module  Banach,
le produit tensoriel $M\hotimes_A A(r)$
a pour $A$-module
sous-jacent le module $M$
et pour norme l'application $m\mapsto r\|m\|$. Nous le noterons $M(r)$. 

\subsubsection{}
Si $(M_i)$ est une famille de $A$-modules de Banach, la \emph{somme directe complétée} $\widehat{\bigoplus M_i}$
est l'ensemble des familles $(m_i)\in \prod M_i$ telles que pour tout $\epsilon>0$, 
l'ensemble des indices $i$ tels que $\norm {m_i}\geq \epsilon$ soit fini
; notons que la famille $I$ peut être de cardinal quelconque, mais qu'un élément donné
de $\widehat{\bigoplus M_i}$ n'a qu'un nombre (au plus) dénombrable de composantes non nulles. On munit cette somme directe complétée
de la norme $(m_i)\mapsto \max \|m_i\|$, qui en fait un $A$-module de  Banach.
Si la famille $(M_i)$ est finie, 
on le notera simplement $\bigoplus M_i$.

\subsubsection{}\label{def-fortgen}
Soit $M$ un $A$-module de Banach. Nous dirons qu'une famille finie $(m_i)$
d'éléments de $M$ est \emph{fortement génératrice} si les $m_i$ sont tous non nuls et si
pour tout $m\in M$ il existe une famille finie $(a_i)$ d'éléments de $A$
telle que $m=\sum a_i m_i$ et $\norm m=\max \;\norm {a_i}\cdot \norm {m_i}$ 
(notons que pour avoir cette dernière égalité, il suffit que  $\norm m\geq \max\; \norm {a_i}\cdot \norm {m_i}$
car l'autre inégalité est automatique). 

Soit $(m_i)$ une famille finie fortement génératrice de $M$. Soit $(n_i)$ une famille d'éléments
de $M$ tels que $\|n_i-m_i\|<\|m_i\|$ pour tout $i$. La famille $(n_i)$ est alors
fortement génératrice. En effet, soit $m\in M$. Puisque $(m_i)$ est fortement génératrice, il existe une famille $(a_i)$ d'éléments de $A$ tels que
$m=\sum a_i m_i$ et $\|m\|=\max \|a_i\|\cdot \|m_i\|$. En posant $\mu_i=m_i-n_i$ pour tout $i$, on a
$m=\sum a_i n_i+\sum a_i\mu_i$. Soit $\epsilon$ le maximum des $\|\mu_i\|\cdot \|m_i\|\inv$ ; on a $\epsilon<1$
et \[\|\sum a_i\mu_i\|\leq \max \|a_i\|\cdot \|\mu_i\|\leq \epsilon \max \| a_i\|\cdot \|m_i\|=\epsilon \| m\|.\]
On a donc écrit
$m=\sum a_i n_i+\mu$ avec $\norm \mu \leq \epsilon\norm m$. En réitérant le processus et en utilisant la complétude de $A$ on obtient
une écriture de $m$ sous la forme $\sum \alpha_i n_i$ avec $m=\max \|\alpha_i\|\cdot \|m_i\|=\max \|\alpha_i\|\cdot \| n_i\|$. 

\subsubsection{}
Soit $M$ un $A$-module de Banach et soit $N$ un sous-module
de $M$. Nous dirons que $N$
est \emph{strictement fermé}
dans $M$ si la distance à $N$ de tout  élément de $M$ est atteinte. La terminologie est consistante : si $N$ est strictement 
fermé il est fermé (un élément de $M$ à distance nulle de $N$ appartiendra en effet alors à $N$) ; c'est donc
un $A$-module de Banach. 

\subsubsection{}\label{def-fortadm}
Soit $\phi \colon N\to M$ une application $A$-linéaire bornée entre $A$-modules de Banach. Rappelons que $\phi$ est dite
\emph{admissible}
si la norme quotient sur $\phi(N)$ est équivalente à la norme induite par celle de $M$. Nous dirons que $\phi$ est \emph{fortement}
admissible si tout élément $m$ de $\phi(N)$ possède un antécédent $n$ par $\phi$ tel que $\norm n=\norm m$. Cela revient à demander 
que la norme quotient sur $\phi(N)$ soit \emph{égale}
à la norme induite par celle de $M$ et qu'elle soit de surcroît atteinte dans chaque fibre ; cette dernière condition
équivaut au caractère strictement fermé de $\mathrm{Ker}(\phi)$ dans $N$. 

Soit $(m_i)$ une famille finie d'éléments non nuls de $M$. Il résulte des définitions que $(m_i)$ est fortement
génératrice si et seulement si l'application 
\[\bigoplus_i A(\norm{m_i})\to M, \;(a_i)\mapsto \sum a_i m_i\]
est une surjection fortement admissible. 

\begin{prop}\label{prop-fort-gen}
Soit $r=(r_1,\ldots, r_n)$ un polyrayon, soit $A$ l'algèbre $k$-affinoïde $k\{T/r\}$,
soit $(\rho_j)$ une famille finie de réels strictement positifs
et soit $M$ le $A$-module $\bigoplus A(\rho_j)$. Soit $N$ un sous-$A$-module
de $M$.  Le sous-module $N$ de $M$ est strictement fermé
et il
possède
une famille finie fortement génératrice.
\end{prop}

\begin{proof}
Soit $s$ un polyrayon $k$-libre tel que $\abs{k_s\gpm}\neq \{1\}$ et tel que les $r_i$ soient tous de torsion 
modulo $\abs{k_s\gpm}$ ; soit $\Theta$ la famille d'indéterminées canonique de $k_s$. 

Pour tout $i$, on note $n_i$ l'ordre de $r_i$
modulo $\abs{k_s\gpm}$, on choisit un élément $\lambda_i$ de $k_s$
tel que $\abs{\lambda_i}=r_i^{n_i}$, et l'on pose $S_i=T_i^{n_i}/\lambda_i$. 
Notons $E$ l'ensemble des $n$-uplets $(e_1,\ldots, e_n)$ tels
que $0\leq e_i<n_i$ pour tout $i$. 
L'anneau $k_s\{T/r\}$ alors un $k_s\{S\}$-module libre de base $(T^e)_{e\in E}$, 
et l'on a de plus 
\[\left\|\sum_{e\in E} \lambda_e T^e\right\|=\max_e \|\lambda_e\|\cdot r^e\]
pour toute famille $(\lambda_e)$ d'éléments de $k_s\{S\}$. Il s'ensuit que $M_s$ est isomorphe
en tant que $k_s\{S\}$-module à $\bigoplus_{j,e}k_s\{S\}(r^e\rho_j)$. 

Il résulte alors de 
\cite{gruson1968}, III, \S 1, cor. de la prop. 1,  que tout sous-$k_s\{S\}$-module de 
$M$ est strictement fermé dans $M$ et possède une 
partie finie fortement génératrice. En particulier, $A_sN$ est strictement fermé dans $M_s$ et possède
une partie finie fortement génératrice. Indiquons quelques conséquences de ces faits : 

\begin{itemize}[label=$\bullet$]
\item $A_sN$ est complet, et partant égal au sous-module $N_s$ de $M_s$ constitué des séries en $\Theta$ dont chaque coefficient appartient à $N$ ; 
\item $N$ est de type fini (si $(\sum n_{I,j}\Theta^I)_j$ est une famille génératrice finie de $N_s$ alors $(n_{0,j})_j$ est
une partie génératrice finie de $N$) ; 
\item $N$ est strictement fermé dans $M$ (si $m\in M$ et si $\sum n_I\Theta^I$ réalise la distance de $m$ à $N_s$
alors $n_0$ réalise la distance de $m$ à $N$). 
\end{itemize}

Soit $(x_1,\ldots, x_\ell)$ une partie finie fortement génératrice du $k_s\{S\}$-module $N_s$. Il
résulte immédiatement des
définitions que cette famille est \emph{a fortiori}
fortement génératrice pour $N_s$ vu comme $A_s$-module.

Pour tout $i$,
posons $\sigma_i=\|x_i\|$. Écrivons
$x_i=\sum_Jx_{i,J}\Theta^J$ (où les $x_{i,J}$ appartiennent à $N$)
et notons $\mathscr J_i$ l'ensemble (fini et non vide) des indices $J$
tels que $\|x_{i,J}\|s^J=\sigma_i$ ; posons $y_i=\sum_{J\notin \mathscr J_i} x_{i,J}\Theta^J$.

Soit $n$ appartenant à $N$.
Il suffit de prouver 
qu'on peut écrire 
\[n=\sum_{1\leq i\leq \ell, J\in \mathscr J_i}a_{i,J}x_{i,J}\]
où les $a_{i,J}$ appartiennent à $A$ et vérifient les inégalités $\|a_{i,J}\|s^{(-J)}\sigma_i\leq \|n\|
$. C'est évident si $n=0$ ; supposons désormais $n\neq 0$ (ceci entraîne que $N\neq \{0\}$ 
et donc que $\ell>0$). 
Comme $(x_1,\ldots, x_\ell)$ engendre fortement $N_s$
on peut écrire $n=\sum b_ix_i$ où les $b_i$ sont des éléments
de $A_s$ tels que $\|b_i\|\sigma_i\leq \|n\|$ pour tout $i$. Écrivons $b_i=\sum_J b_{i,J}\Theta^J$ ; on a $\|b_{i,J}\|s^J\sigma_i\leq
\|n\|$ pour tout $(i,J)$.
Comme $n$ appartient à $N$ (et donc est constante lorsqu'on la voit
comme série en la multivariable $\Theta$) on a 
\[n=\sum_{i,J}b_{i,-J}x_{i,J}=\sum_{i,J\in \mathscr J_i}b_{i,-J}x_{i,J}
+\sum_{i,J\notin \mathscr J_i}b_{i,-J}x_{i,J}.\]
Soit $(i,J)$ un couple d'indices. On a $\|b_{i,-J}\|s^{-J}\sigma_i\leq\|n\|$ et si 
de plus $J\notin \mathscr J_i$ alors
\[\|b_{i,-J}\|s^J \|x_{i,J}\|
\leq \|b_{i,-J}\|\cdot \|y_i\|
\leq s^J\|n\|\sigma_i\inv \|y_i\|,\]
d'où finalement la majoration
\[\|b_{i,-J}\|\cdot
\|x_{i,J}\|
\leq\|n\| \sigma_i\inv \|y_i\|\;.\]
Posons $\epsilon=\max_i \sigma_i\inv \|y_i\|$. On a $0\leq \epsilon<1$ et $n$
s'écrit par ce qui précède \[\sum_{i,J\in \mathscr J_i}b_{i,(-J)}x_{i,J}+\nu\]
avec $\|b_{i,-J}\|s^{-J}\sigma_i\leq\|n\|$ pour tout $(i,J)$ et $\|\nu\| \leq \epsilon \|n\|$. En réitérant le processus et en utilisant
la complétude de $A$ on obtient une écriture de $n$ de la forme \[n=\sum_{1\leq i\leq m, J\in \mathscr J_i}c_{i,J}x_{i,J}\]
où les $c_{i,J}$ appartiennent à $A$ et vérifient les inégalités $\|c_{i,J}\|s^{-J}\sigma_i\leq \|n\|$.
\end{proof}

\begin{prop}\label{prop-fort-adm}
Soient $r$ et $\rho=(\rho_i)$ deux polyrayons, soit $A$ l'algèbre
$k$-affinoïde $k\{T/r\}$ et soit $P$ le $A$-module
de Banach
fini $\bigoplus A(\rho_i)$.  Soit $\phi$ 
un morphisme de $P$ vers un $A$-module
de Banach fini $N$. Soit
$L$ une extension complète de $k$. Les assertions suivantes sont équivalentes : 
\begin{enumerate}[i]
\item $\phi$ est une surjection fortement admissible ;  

\item $\phi_L$ est une surjection fortement admissible.
\end{enumerate}
\end{prop}

\begin{proof}
Plaçons-nous tout d'abord dans le cas où $\abs{k\gpm}\neq \{1\}$. Supposons que
$\phi$ est une surjection fortement admissible.
La surjection $\phi\colon P\to N$ contracte les normes, et
la norme de $N$ coïncide avec la norme quotient associée à $\phi$ ; il
résulte alors de \cite{gruson1966}, \S3, Théorème 1 (1), que $\phi_L$ est une surjection contractant les normes
et
que la
norme de $N_L$ est égale à la norme quotient
associée à $\phi_L$ ; comme $\mathrm{Ker}(\phi_L)$ est strictement fermé dans $P_L$ d'après la
proposition \ref{prop-fort-gen}, $\phi_L$ est fortement admissible. 

Réciproquement, supposons $\phi_L$
fortement admissible. L'application 
$\phi_L$ est alors surjective ; comme $\phi_L$ se déduit de $\phi$ par extension des scalaires
algébrique de $A$ à $A_L$, et comme $A_L$ est fidèlement plat sur $A$, le morphisme $\phi$ est surjectif, et 
il est admissible par la variante ultramétrique du théorème de l'image ouverte (qui ne vaut que lorsque la valuation 
du corps de base n'est pas triviale). 

Comme l'injection $k\hookrightarrow L$ est isométrique, il découle de \cite{gruson1966}, \S3, Th. 1, 
que la flèche $Q\to Q_L$ est isométrique pour tout $k$-espace de Banach $Q$. Il s'ensuit que $\phi$ diminue
les normes, puisque c'est par hypothèse le cas de $\phi_L$. Si l'on note $N'$ le module $N$ muni de la norme quotient, on a donc un diagramme
$P\to N'\hookrightarrow N$ d'applications $k$-linéaires diminuant les normes, la seconde étant
un isomorphisme borné de $k$-espaces de Banach. On en déduit un diagramme
$P_L\to N'_L\hookrightarrow N_L$ dont la seconde flèche est un isomorphisme borné de $L$-espaces de Banach. 

Par définition, la norme de $N'$ est la norme quotient associée à la surjection de $P$ vers $N'$, laquelle contracte les normes ; 
le noyau de cette surjection étant strictement fermé dans $P$ d'après la
proposition \ref{prop-fort-gen},
elle est fortement admissible. 
Il résulte alors
de l'implication déjà établie que $P_L\to N'_L$ est une surjection
fortement admissible. Puisque $\phi_L$ est fortement admissible par hypothèse, $N'_L\to N_L$ est isométrique. Les flèches $N'\to N'_L$ et $N\to N_L$ étant isométriques, $N'\to N$ est isométrique et $\phi$ est fortement admissible.

Il reste à traiter le cas où $k$ est trivialement valué. Soit $t$ un réel strictement positif et soit $L'$ une extension complète de $k$ composée de $L$ et $k_t$ ; notons que si $L$
est trivialement valué on a $L'=L_t$. 

Par un calcul direct, on vérifie que $\phi$ est une surjection fortement admissible si et seulement si c'est le cas de $\phi_t$. Il résulte de ce qui précède
 que $\phi_t$ est une surjection fortement admissible si et seulement si c'est le cas de $\phi_{L'}$. Enfin, $\phi_{L'}$ est une surjection fortement admissible si et seulement si c'est le cas de $\phi_L$, d'après ce qui précède
 si $L$ n'est pas trivialement valué et
par calcul direct sinon car alors $L'=L_t$.  Par conséquent $\phi$ est une  surjection fortement admissible si et seulement si c'est le cas de $\phi_L$. 
\end{proof}

 \begin{defi}
Soit $V$ un $k$-espace de Banach. Une famille $(v_j)_{j\in J}$
d'éléments non nuls de $V$ sera appelée une \emph{base de Schauder
orthogonale} 
si le morphisme
de $k$-espaces
de Banach $\widehat{\bigoplus_i k(\norm{v_j})}\to V$ 
qui envoie $(\lambda_j)$ sur $\sum \lambda_j v_j$ est une isométrie.

Plus concrètement, cela signifie que tout élément $v$ de $V$ a une unique écriture sous la forme 
$\sum \lambda_j v_j$, où $(\lambda_j)$ est une famille de scalaires à support dénombrable dont les valeurs absolues tendent vers zéro, et que $\norm v=\max \abs {\lambda_j}\cdot \norm{v_j}$. 
\end{defi}

\begin{rema}
La définition traditionnelle d'une base de Schauer orthogonale (\emph{cf.} par exemple la définition 6 de \cite{bosch-g-r1984}, \S 2.7.2) requiert que l'ensemble $J$ d'indices soit dénombrable, mais nous nous permettons de la modifier en 
supprimant cette exigence.
\end{rema}

\begin{rema}\label{rem-schauder-incomplet}
Soit $V$ un $k$-espace de Banach
et soit $v$ un élément non  nul de $V$. Si $V$ possède une base de Schauder orthogonale $(v_i)$, il résulte de la proposition 7
de \cite{bosch-g-r1984}, \S 2.7.2 (appliquée  à un sous-espace
complet de $V$ engendré topologiquement  par une famille dénombrable de $v_j$ et contenant $v$)
que $V$ possède une base de Schauder orthogonale contenant $v$. 
\end{rema}

\subsection{}\label{ss-schauder-telesc}
Soit $L$ une extension complète de $k$ et soit $\Lambda$ une extension complète de $L$. Si $(\ell_i)$ est  une base de Schauder orthogonale 
de $L$ sur $k$ et si $(\lambda_j)$ est une base de Schauder orthogonale
de $\Lambda$ sur $L$ alors $(\ell_i\lambda_j)_{i,j}$ est une base de Schauder
orthogonale de $\Lambda$ sur $k$. 

\subsection{}\label{rappel-kr-stable}
Rappelons que si $r$ est un polyrayon, nous désignons par $k(T/r)$ le corps $k(T)$ muni de
la valuation de Gauß de paramètre $r$. Son complété $\widehat{k(T/r)}$, qui n'est autre que
$k_r$ si $r$ est $k$-libre, s'identifie à
$\mathscr H(\eta_r)$ où $\eta_r$ est le point de $\A^{1,\mathrm{an}}_k$ défini précisément par la
norme de Gauß $\sum a_i T^i \mapsto \max \abs{a_i}\cdot r^i$. 

Si $k$ est stable alors $\widehat{k(T/r)}$ est stable : on peut le déduire de la stabilité
de $k(T/r)$ (\cite{kuhlmann2010}) et de la préservation de la stabilité par complétion en hauteur $1$
(\cite{gruson1968}, Remarque \no 3 page 57) ; on pourra aussi se reporter
à la preuve de Michael Temkin (\cite{temkin2010}, Th. 6.3.1). 

\begin{lemm}\label{lem-kr-schauder}
Soit $r$ un polyrayon. Le corps valué complet  $\widehat{k(T/r)}$
possède une base de Schauder orthogonale sur $k$. 
\end{lemm}

\begin{rema}
D'après la
remarque
\ref{rem-schauder-incomplet}, 
le lemme ci-dessus entraînera
l'existence d'une base de Schauder orthogonale de $\widehat {k(T/r)}$ sur $k$ \emph{contenant $1$} ; mais
en fait le lecteur
vérifiera que
la base de Schauder orthogonale construite dans notre preuve contient automatiquement 1. 
\end{rema}

\begin{proof}[Démonstration du lemme \ref{lem-kr-schauder}]
En vertu de \ref{ss-schauder-telesc} on peut raisonner par récurrence
sur la longueur de $r$ et donc supposer que celle-ci vaut $1$ ; on désigne
donc à partir de maintenant par $r$ un réel strictement positif et par $T$
une unique indéterminée. 

Supposons que $r\notin \abs{k\gpm}^\Q$. Dans ce cas $\widehat{k(T/r)}=k_r$ et 
la famille $(T^i)_{i\in \Z}$ est une base de Schauder orthogonale 
de $k_r$ sur $k$. 

Supposons que $r\in \abs{k\gpm}^\Q$ et soit $N$ l'ordre de $r$ modulo $ \abs{k\gpm}$. La famille $\{1, T,\ldots, T^{N-1})$ est une 
base de Schauder orthgonale de $\widehat{k(T/r)}$ sur $\widehat{k(T^N/r^N)}$. Il suffit donc de montrer que  $\widehat{k(T^N/r^N)}$
possède une base de Schauder orthogonale sur $k$ ; ceci permet de se ramener au cas où $r\in \abs{k\gpm}$ puis, par renormalisation de l'indéterminée, au cas où $r=1$. 

Fixons un ensemble $\mathscr P$ de polynômes unitaires à coefficients dans $k^\circ$ telle que $P\mapsto \widetilde P^1$ établisse une bijection entre $\mathscr P$ et l'ensemble des polynômes irréductibles unitaires de $\widetilde k^1[T]$ (notons que chaque polynôme $P$ est alors irréductible, puisque toute extension finie de $k$ contenant une racine de $P$ aura un indice d'inertie au moins égal à $\deg P$) ; si la valeur absolue de $k$ est non triviale, on suppose de plus que $\mathscr P$ est constitué de polynômes séparables. Nous allons montrer que la famille
$\mathscr B:=((T^i)_{i\geq 0}, (\frac{T^d}{P^m})_{P\in \mathscr P, 0\leq d<\deg P, m>0})$
est une base de Schauder orthogonale de $\widehat{k(T)}$ sur $k$. 

Si la valeur absolue de $k$  est triviale, c'est évident : dans ce cas, une base de Schauder orthogonale de $\widehat{k(T)}=k(T)$ sur $k$ en est simplement une base  au sens usuel, et la théorie de la décomposition en éléments simples sur le corps $k$ permet de conclure. 

Supposons maintenant la valeur absolue de $k$ n'est pas triviale. La théorie de la décomposition en éléments simples sur $\widetilde k^1$ assure que $\mathscr B$
se réduit en une base de $\widetilde k^1(T)$ sur $\widetilde k^1$ ; il s'ensuit qu'elle
constitue une base de Schauder orthogonale du sous-espace de Banach $E$ de $\widehat{k(T)}$
qu'elle engendre (topologiquement) ; il suffit maintenant de montrer que $E=\widehat{k(T)}$. Notons $\eta$ le point de Gauß de $\P^{1,\mathrm{an}}_k$ ; on a $\widehat{k(T)}=\hr \eta$. 

Supposons tout d'abord que $k$ est algébriquement clos. Chaque polynôme $P$ est alors de la forme $T-a_P$ avec $a_P\in k^\circ$ ; les $a_P$ constituent un système complet de représentants de $k^\circ$ modulo $k^{\circ \circ}$.
Soit $f\in \kappa(\eta)$. Il existe un voisinage affinoïde $V$ de $\eta$ sur lequel $f$ est définie. Quitte à restreindre $V$ on peut supposer qu'il est décrit par une conjonction d'inégalités de la forme 
\[\abs T\leq R\;\text{et}\;\abs{T-a_{P_1}}\geq r_1\;\text{et}\;\ldots\;\text{et}\;\abs{T-a_{P_s}}\geq r_s\]
où $1<R$ et $0<r_j<1$ pour tout $j$. Montrons que $f$ (vue comme fonction analytique sur $V$) s'écrit comme une somme
\[\sum a_i T^i+\sum_{j=1}^s \sum_i \frac{\alpha_{ij}}{(T-a_{P_j})^i}\]
où $\abs{a_i}R^i\to 0$ et $\abs{\alpha_{ij}}^ir_j^{-i}\to 0$ pour tout $j$ ; cela assurera que $f$ appartient à $E$, et partant que $\kappa(\eta)\subset E$ et donc que $E=\hr \eta$ par densité de $\kappa(\eta)$ dans $\hr \eta$. 
On raisonne par récurrence sur $s$. Si $s=0$ alors $V$ est le disque fermé de rayon $R$ et l'assertion requise est évidente. Supposons $s>0$ et l'assertion vraie pour les entiers $<s$. La restriction de $f$ à la couronne
$r_s\leq \abs{T-a_{P_s}}<1$ est de la forme $\sum_{i \in Z}
 \alpha_i(T-a_{P_s})^i$ où $\abs {\alpha_i}r_s^i$ tend vers $0$ quand $i$ tend vers $-\infty$ (et où $\alpha_ir^i$ tend vers $0$ lorsque $i$ tend vers $+\infty$ pour tout $r<1$). Posons $g=f-\sum_{i<0}  \alpha_i(T-a_{P_s})^i$. La fonction $g$ est alors une fonction analytique sur $V$ qui se prolonge par construction
au voisinage affinoïde de $\eta$ défini par les inégalités
\[\abs T\leq R\;\text{et}\;\abs{T-a_{P_1}}\geq r_1\;\text{et}\;\ldots\;\text{et}\;\abs{T-a_{P_{s-1}}}\geq r_{s-1}.\]On conclut en appliquant l'hypothèse de récurrence à $g$. 

Il reste à traiter le cas où le corps $k$ n'est plus nécessairement algébriquement clos. Soit $\alg k$ le complété
d'une clôture algébrique de $k$ et soit $Q$ le quotient $\hr \eta/E$, muni de la norme quotient. Il s'agit
de montrer que $Q$ est nul ;  mais $Q\otimes_k \alg k$ s'injecte d'après \cite{gruson1966}, \S3, Th. 1 dans $Q\hotimes_k \alg k$ ; il suffit donc de montrer que ce dernier 
est nul ; comme
la suite $0\to E\hotimes_k \alg k\to \hr \eta\hotimes_k \alg k\to Q\hotimes_k \alg k\to 0$ est exacte en vertu
de \emph{loc. cit.}, il suffit finalement de vérifier que $E\hotimes_k \alg k=\hr \eta\hotimes_k \alg k$.

La norme de $\hr \eta$ est universellement multiplicative (Berkovich, \emph{cf.} \cite{ducros2009}, Lemme 1.8
et exemple 1.9) ; en particulier, la norme tensorielle de $\hr \eta\hotimes_k \alg k$ est multiplicative, et coïncide
donc avec sa norme spectrale . Or comme le point de Gauß $\eta_{\alg k}$ de $\P^{1,\mathrm{an}}_{\alg k}$
est invariant sous l'action de Galois, c'est l'unique antécédent de $\eta$ sur $\P^{1,\mathrm{an}}_{\alg k}$, 
et $\mathscr M(\hr \eta\hotimes_k \alg k)$ est donc réduit au singleton $\{\eta_{\alg k}\}$.
Il s'ensuit que le morphisme naturel
$\hr \eta\hotimes_k \alg k\to \hr{\eta_{\alg k}}$ est un isomorphisme isométrique ; on identifie
désormais $\hr \eta\hotimes_k \alg k$ à $\hr{\eta_{\alg k}}$.

Pour tout $i\geq 0$ on a $T^i\in E\hotimes_k \alg k$. 
Soit maintenant $P$ un élément de $\mathscr P$
et soient $\alpha_1,\ldots, \alpha_{\deg P}$ 
ses racines dans $\alg k$. Pour tout entier $m$, 
le sous-$\alg k$-espace vectoriel de $\hr {\eta_{\alg k}}$ 
engendré par les éléments de la forme $T^d/P^\ell$ avec $0\leq d<\deg P$ et $1\leq \ell\leq m$
est contenu (par décomposition en éléments simples) dans celui engendré par les $\frac 1{(T-\alpha_i)^\ell}$
pour $i$ variant entre $1$ et $\deg P$ et $\ell$ entre $1$ et $m$ ; pour des raisons
de dimension, ces deux sous-espaces vectoriels coïncident. Par conséquent, pour tout entier $m>0$, 
tout polynôme $P\in \mathscr P$ et toute racine $\alpha$ de $P$ dans $\alg k$, l'élément $\frac 1{(T-\alpha)^m}$ 
appartient à $E\hotimes_k \alg k$. Comme l'ensemble des racines dans $\alg k$ des polynômes $P$ appartenant à $\mathscr P$ contient un système de représentants de $(\alg k)^\circ$ modulo  $(\alg k)^{\circ\circ}$
on déduit du cas algébriquement clos déjà traité que 
 $E\hotimes_k \alg k=\hr{\eta_{\alg k}}$. 
\end{proof}

\begin{defi}\label{def-distingue}
Soit $A$ une algèbre $k$-affinoïde et soit $r$ un polyrayon.
Nous dirons qu'un morphisme $\phi\colon k\{T/r\}\to A$
est \emph{distingué}
si tout élément $a$ de $A$ possède
un antécédent $b$ par $\phi$
tel que $\|b\|$ soit égal à la semi-norme spectrale
de $a$. 
\end{defi}

\subsection{}\label{distingue-reduit}
Soit $\phi\colon k\{T/r\}\to A$ un morphisme. 
Si $\phi$ est distingué la semi-norme spectrale de $A$ est une norme, 
ce qui équivaut à dire que $A$ est réduite. 
Et si $A$ est réduite alors $\phi$ est distingué si et seulement si
c'est une surjection fortement admissible lorsque $A$
est munie de sa norme spectrale : cela résulte de la définition et du fait que $\phi$
diminue les normes
spectrales.

\begin{defi}
Soit $A$ une algèbre $k$-affinoïde. Nous dirons que $A$ est \emph{distinguée}
s'il existe un polyrayon $r$ et une surjection distinguée
$u\colon k\{T/r\}\to A$.

Nous dirons que $A$ est \emph{$\Gamma$}-distinguée
s'il existe une surjection distinguée $k\{T/r\}\to A$ pour un certain
polyrayon $r$ 
constitué d'éléments de $\Gamma$. 
\end{defi}

\begin{rema}
Lorsque la valuation de $k$ n'est pas triviale, une algèbre
strictement $k$-affinoïde
est $\{1\}$-distinguée
si et seulement si elle est distinguée
au sens de \cite{bosch-g-r1984}, 6.4.3. 
\end{rema}

\subsection{}\label{distingue-bon-rayon}
Soit $A$ une algèbre $k$-affinoïde distinguée dont on note
$\|\cdot\|$
la norme spectrale (rappelons 
que $A$ est réduite d'après \ref{distingue-reduit}).

\subsubsection{}\label{sss-distingue-bon-rayon}
Il existe 
un morphisme distingué
 $\psi\colon k\{\Theta_1/\rho_1,\ldots, \Theta_m/\rho_m\}\to A$ tel que $\|\psi(\Theta_i)\|=\rho_i$ pour tout
 $i$. 
 Partons en effet d'un morphisme distingué quelconque
 $\phi\colon k\{T_1/r_1,\ldots, T_n/r_n\}\to A$ ; pour tout $i$ on
 pose $\rho_i=\|\phi(T_i)\|$. Soit $I$ l'ensemble des indices $i$ tels que $\rho_i>0$ ; on a $\phi(T_i)=0$
 pour tout $i\notin I$. Le
 morphisme $\phi \colon k\{T/r\}\to A$ admet alors une factorisation
 \[\xymatrix{
 k\{T/r\}\ar[rr]^\chi&&k\{\Theta_i/\rho_i\}_{i\in I}\ar[rr]^\psi&&A
 }\]
 où $\chi$ envoie $T_i$ sur $\Theta_i$ si $i\in I$
 et sur $0$ sinon, et où
 $\psi$ envoie $\Theta_i$ sur $\phi(T_i)$ pour tout $i\in I$.
 Soit $a\in A$. Il possède un antécédent $b$ par $\phi$ tel que $\|b\|=\|a\|$. 
 On a alors $a=\psi(\chi(b))$ et $\|a\|\leq \|\chi(b)\|\leq \|b\|=\|a\|$, si
 bien que $\|\chi(b)\|=\|a\|$.

\subsubsection{}
\label{sss-distingue-perturbation}
Donnons-nous
un morphisme $\psi\colon k\{\Theta_1/\rho_1,\ldots, \Theta_m/\rho_m\}\to A$ comme au \ref{sss-distingue-bon-rayon}
ci-dessus. 
 Pour tout $i$, on pose $f_i=\psi(\Theta_i)$ et on suppose
 donné un élément $g_i\in A$ tel que $\norm{f_i-g_i}<\rho_i$. 
 
 Soit $a\in A$. Comme $\psi$ est distingué il existe une série $\sum \alpha_I \Theta^I$
 appartenant à $k\{\Theta/\rho\}$ telle que $a=\sum \alpha_If^I$ et $\max \abs{\alpha_I}\rho^I=\norm a$.
 En écrivant $f_i=g_i+(f_i-g_i)$ et en développant l'égalité $a=\sum \alpha_If^I$ on voit qu'il existe une série
 $\sum \beta_I \Theta^I$
 appartenant à $k\{\Theta/\rho\}$ telle que $a=\sum \beta_Ig^I$ et telle que $\max \abs{\beta_I}\rho^I=\norm a$. 
 Par conséquent, le
 $\psi' \colon  k\{\Theta_1/\rho_1,\ldots, \Theta_m/\rho_m\}\to A$ qui envoie $\Theta_i$ sur $g_i$ pour tout $i$ est lui
 aussi une surjection distinguée.
 
 \begin{rema}\label{rem-distingue-bonrayon}
 Soit $A$ une algèbre $k$-affinoïde distinguée.
 Si $A$ est $\Gamma$-distinguée, il existe une surjection distinguée $k\{T/r\}\to A$ avec $r$ constitué d'éléments de
 $\Gamma$, ce qui implique que $\norm A\subset \abs k\cdot \Gamma$ (et entraîne par ailleurs que $A$ est $\Gamma$-stricte). 

Réciproquement, supposons que $\norm A\subset \abs k\cdot \Gamma$, et
choisissons une surjection distinguée $k\{T/r\}\to A$. Il
résulte de \ref{sss-distingue-bon-rayon}
que $k\{T/r\}\to A$ se factorise par une surjection distinguée $v\colon k\{\Theta_1/\rho_1,\ldots, \Theta_n/\rho_n\}\to A$ telle que
$\norm{v(\Theta_i)}=\rho_i$ pour tout $i$. 
Chacun des $\rho_i$ appartient donc à $\norm A\setminus\{0\}\subset
\abs {k\gpm}\cdot \Gamma$.
En renormalisant convenablement les indéterminées on obtient une surjection distinguée
$w\colon k\{\Theta'_1/\rho'_1,\ldots, \Theta'_n/\rho'_n\}\to A$ où
chacun des $\rho'_i$ appartient à $\Gamma$ (et 
où $\norm{w(\Theta'_i)}=\rho'_i$ pour tout $i$). 
Il s'ensuit que $A$ est 
$\Gamma$-distinguée. 
\end{rema} 
 
\begin{lemm}\label{lem-distingue-changebase}
Soit $A$ une $k$-algèbre affinoïde et soit $L$ une extension complète de $k$
possédant une base de Schauder orthogonale sur $k$.
On
fait les hypothèses suivantes : 
\begin{enumerate}[a]
\item l'algèbre $A$ est réduite ; 
\item lorsqu'on munit $A$ de sa norme spectrale, la norme tensorielle de $A_L$ coïncide
avec sa (semi)-norme spectrale (ce qui implique que $A_L$
est réduite). 
\end{enumerate}
Les assertions suivantes sont alors équivalentes : 
\begin{enumerate}[i]
\item l'algèbre $k$-affinoïde $A$ est distinguée ; 
\item l'algèbre $L$-affinoïde $A_L$ est distinguée. 
\end{enumerate}
\end{lemm}

\begin{proof}
Choisissons une base de Schauder $(\Theta_j)_{j\in J}$ de $L$ sur $k$,
en faisant en sorte qu'il existe $0\in J$ tel que $\Theta_0=1$
(remarque \ref{rem-schauder-incomplet}).
Supposons $A$ distinguée et choisissons un morphisme distingué $\phi \colon k\{T/r\}\to A$.
Le morphisme $\phi$ est une surjection fortement admissible et
la proposition \ref{prop-fort-adm}
(que l'on applique en voyant $A$ comme un $k\{T/r\}$-module de Banach fini \emph{via}
$\phi$)
assure que $\phi_L$
est
une surjection fortement admissible lorsque $A_L$ est munie de sa norme tensorielle, qui coïncide ici par hypothèse
avec sa norme spectrale. Par conséquent $\phi_L$ est distingué et $A_L$
est distinguée. 

Réciproquement, supposons $A_L$ distinguée. On note $\|\cdot\|$
les normes spectrales de $A$ et $A_L$. Choisissons
un morphisme distingué $L\{T_1/r_1,\ldots, T_n/r_n\}\to A_L$,
et notons $f_i$ l'image de $T_i$ pour tout $i$. 
Il résulte de \ref{distingue-bon-rayon} qu'on peut supposer
que $\|f_i\|=r_i$ pour tout $i$ puis, quitte à perturber
si besoin le morphisme $L\{T/r\}\to A_L$, 
que $f_i$
est pour tout $i$ une somme \emph{finie}
$\sum \lambda _{i,j} \Theta_j$
(où les $\lambda_{i,j}$
appartiennent à $A$). 
Pour tout
$(i,j)$ tel que $\lambda_{i,j}\neq 0$ on pose
$s_{i,j}=\|\lambda _{i,j}\|$ ; les réels $s_{i,j}$ sont tous 
strictement positifs ; on a par ailleurs $r_i=\|f_i\|=\max_j s_{i,j}\rho_j$. 
Soit $\phi \colon k\{S_{i,j}/s_{i,j}\}_{i,j}\to A$ le
morphisme qui envoie $S_{i,j}$ sur $\lambda_{i,j}$ pour tout $(i,j)$ ; on note
$\phi_L\colon L\{S_{i,j}/s_{i,j}\}_{i,j}\to A_L$ le morphisme déduit de $\phi$ par extension des scalaires. 

Soit $a$ appartenant à $A$.
Par hypothèse, il existe un élément $\sum a_I T^I$
de $L\{T/r\}$ tel que $a=\sum a_If^I$ et $\max\;\abs{a_I} r^I=\|a\|$. Comme
\[\left\|\sum_j S_{i,j}\Theta_j\right\|=\max_J s_{i,j}\rho_j=r_i\]
pour tout $i$, on peut évaluer la série
$\sum a_I T^I$ en $(\sum_j S_{i,j}\Theta_j)_{1\leq i\leq n}$. On obtient un élément
$b\in L\{S_{i,j}/s_{i,j}\}_{i,j}$, de norme majorée
par $\|\sum a_I f^I\|=\|a\|$, et qui vérifie par construction l'égalité $\phi_L(b)=a$. 
Puisque $a\in A$ on a également $\phi(b_0)=a$ où $b_0$ désigne le coefficient 
de
$\Theta_0=1$
dans la décomposition de $b$ sur la «base» $(\Theta_j)$. 
Par ailleurs la définition de la norme de $L\{S_{i,j}/s_{i,j}\}_{i,j}$ entraîne que $\|b_0\|\leq \|b\|
\leq\|a\|$. Comme
$\phi$ contracte les normes spectrales il vient  $\|b_0\|=\|a\|$. 
\end{proof}

\begin{lemm}\label{lem-norm-ar}
Soit $A$ une algèbre $k$-affinoïde réduite, munie de sa norme spectrale $\norm \cdot$. 
Soit $r=(r_1,\ldots, r_n)$ un polyrayon, soit $\norm \cdot$ la norme tensorielle de $A_{\widehat{k(T/r)}}$
et soit $\norm \cdot_\infty$ sa semi-norme spectrale. 
On a l'égalité $\norm \cdot_\infty=\norm \cdot$, et
de plus $\norm {A_{\widehat{k(T/r)}}}$ est égal à
$\norm A \cdot r_1^\Z\cdot\ldots\cdot r_n^\Z$. 
\end{lemm}

\begin{rema}\label{rem-ar-reduit}
On sait que
$\norm \cdot_\infty$ est une norme car $A$ est réduite et car $\widehat{k(T/r)}$ est analytiquement séparable, \emph{cf.}
\cite{ducros2009}, exemple 1.9 et théorème 3.1 ; mais nous n'allons pas utiiser ce fait 
dans la preuve du lemme, qui le redémontrera donc au passage.
\end{rema}

\begin{proof}[Preuve du lemme \ref{lem-norm-ar}]
Soit $P$ un élément de $A_{\widehat{k(T/r)}}$ de la forme
$\sum a_I T^I$ où $I$ parcourt un ensemble fini. 
On a $\norm P\leq \max_I \norm {a_I}\cdot r^I$.
Soit $\pi\colon \mathscr M(A_{\widehat{k(T/r)}})
\to \mathscr M(A)$ l'application continue naturelle. 
On a les égalités
\begin{eqnarray*}
\norm P_{\infty}&=&\sup_{y\in \mathscr M(A_{\widehat{k(T/r)}})} \abs{P(y)}\\
&=&\sup_{x\in \mathscr M(A)}\sup_{y\in \pi\inv(x)}\abs{P(y)}\\
&=&\sup_{x\in \mathscr M(A)}\max_I \abs{a_I(x)}r^I\\
&=&\max_I \norm{a_I}r^I.
\end{eqnarray*}
Comme on a par ailleurs $\norm P_\infty \leq \norm P$, 
il vient $\norm P_\infty=\norm P=\max_I \norm{a_I}r^I$. 

On conclut en remarquant que $\norm \cdot$ et $\norm \cdot_\infty$ sont des normes
de $\widehat{k(T/r})$-espaces vectoriels et que les éléments $P$ comme ci-dessus engendrent un sous-$\widehat{k(T/r})$ -espace
vectoriel dense de $A_{\widehat{k(T/r)}}$.
\end{proof}

La proposition suivante fait appel de manière cruciale à la notion de $\Gamma$-annéloïde résiduel introduite au
\ref{def-annresiduel-gamma}.

\begin{prop}\label{distingue-residu-reduit}
Soit $r=(r_1,\ldots, r_n)$ un polyrayon
constitué d'éléments de $\Gamma$ et soit $I$ un idéal de $k\{T/r\}$. 
Soit $(a_1,\ldots, a_m)$ une famille fortement génératrice de $I$
telle que $\norm {a_i}\in \Gamma$ pour tout $i$. 
Soit $A$ l'algèbre quotient $k\{T/r\}/I$
et soit $\norm \cdot_{\mathrm q}$ sa norme quotient. 

\begin{enumerate}[1]
\item La norme quotient $\norm \cdot_{\mathrm q}$ est sous-multiplicative et à valeurs dans $\abs k \cdot \Gamma$. On note
$\widetilde A_{\mathrm q}$ le $\Gamma$-annéloïde résiduel de $(A,\norm \cdot _{\mathrm q})$. 

\item La surjection $k\{T/r\}\to A$
induit un isomorphisme
$\widetilde k[r\backslash T]/(\widetilde{a_1},\ldots, \widetilde{a_m})\simeq \widetilde A_{\mathrm q}$. 

\item 
Les assertions suivantes sont équivalentes : 
\begin{enumerate}[j]
\item La surjection  $k\{T/r\}\to A$ est distinguée. 
\item La surjection  $k\{T/r\}\to A$  induit un isomorphisme de $\widetilde k$-algèbres
\[\widetilde k [r\backslash T]/(\widetilde{a_1},\ldots, \widetilde{a_m})\simeq \widetilde A\]
(rappelons que $\widetilde A$ désigne le $\Gamma$-annéloïde résiduel de $A$ relatif à sa semi-norme spectrale). 
\item La $\widetilde k$-algèbre $\widetilde k [r\backslash T]/(\widetilde{a_1},\ldots, \widetilde{a_m})$
est réduite. 
\end{enumerate}
\end{enumerate}
\end{prop}

\begin{proof}
La norme de $k\{T/r\}$ est multiplicative, ce qui entraîne que $\norm \cdot_{\mathrm q}$ est sous-multiplicative. 
Par ailleurs, l'idéal $I$ est strictement fermé
dans $k\{T/r\}$ d'après la proposition 
\ref{prop-fort-gen}, ce qui implique que $k\{T/r\}\to A$ est fortement
admissible lorsqu'on voit $A$ comme munie de $\norm \cdot_{\mathrm q}$.
La norme de $k\{T/r\}$ étant  à valeurs
dans $\abs k\cdot \Gamma$, il en résulte que 
$\norm A_{\mathrm q}\subset \abs k\cdot \Gamma$. 

\subsubsection{}
Montrons que
le morphisme 
$\widetilde k[r\backslash T]\/(\widetilde{a_1},\ldots, \widetilde{a_m})\to \widetilde A_{\mathrm q}$
est bijectif. Pour tout $a\in A$ et tout
élément
$R$ de $\Gamma$ supérieur
ou égal à $\norm a_{\mathrm q}$
on note $\widetilde a_{\mathrm q}^R$
l'image de $a$ dans $\widetilde A_{\mathrm q}^R$,
et l'on omettra le $R$ en exposant si $R=\norm a_{\mathrm q}$.

Commençons par la surjectivité. Soit $\alpha\in \widetilde A_{\mathrm q}$ ; montrons
qu'il possède un antécédent dans $\widetilde k[r\backslash T]$. C'est clair
si $\alpha=0$. Supposons que $\alpha\neq 0$ et choisissons 
un élément $a$
de $A$ de norme
appartenant à $\Gamma$ et
tel que
$\alpha=\widetilde a_{\mathrm q}$.
Puisque 
$k\{T/r\}\to A$ est fortement admissible relativement
à $\norm \cdot_{\mathrm q}$, 
il existe un antécédent $b$ de $a$
dans $k\{T/r\}$ tel que $\|b\|=\|a\|_{\mathrm q}$. L'élément $\alpha=\widetilde a_{\mathrm q}$ de $\widetilde A_{\mathrm q}$ est alors
égal à l'image de $\widetilde b$. 

Soit maintenant $\beta$ un élément de $\widetilde k[r\backslash T]$.
Nous allons montrer que l'image de $\beta$ dans $\widetilde A_{\mathrm q}$ est nulle
si et seulement si $\beta\in(\widetilde{a_1},\ldots, \widetilde{a_m})$. On peut supposer $\beta$ non nul ; il est alors égal à
$\widetilde b$ pour un certain élément $b$
de $k\{T/r\}$ dont la norme appartient à $\Gamma$ ; notons $a$ l'image de $b$ dans $A$. L'image
de $\beta$ dans $\widetilde A_{\mathrm q}$ est égale à $\widetilde a_{\mathrm q}^R$. 

Supposons que $\beta$ appartienne à $(\widetilde{a_1},\ldots, \widetilde{a_n})$. Il existe alors
$\lambda_1,\ldots \lambda_n\in k\{T/r\}$ tels que $\|\lambda_i\|\cdot \|\alpha_i\|\leq R$ pour tout $i$ et tels que
$\|b-\sum \lambda_i a_i\|<R$. Comme $b-\sum \lambda_i \alpha_i$ a également pour image $a$
dans $A$, il vient $\|a\|_{\mathrm q}<R$ puis $\widetilde a_{\mathrm q}^R=0$. 

Inversement, supposons que $\widetilde a_{\mathrm q}^R=0$. Cela signifie que $\|a\|_{\mathrm q}<R$. Par conséquent, $a$ possède un antécédent 
$c$ dans $k\{T/r\}$ tel que $\|c\|<R$. Puisque $b$ et $c$ ont même image dans $A$, la différence $b-c$ appartient à $I$. Comme les $\alpha_i$
engendrent fortement $I$, il existe une famille $(\lambda_i)$ d'éléments de $k\{T/r\}$ tels que $\|\lambda_i\|\cdot \|\alpha_i\|\leq \|b-c\|=R$
pour tout $i$ et telle que $b-c=\sum \lambda_i a_i$. Mais on a alors
$\beta=\widetilde b=\sum_{i\in I} \widetilde{\lambda_i} \widetilde {a_i}$
où $I$ désigne l'ensemble des indices $i$ tels que $\norm \lambda_i\cdot \norm {a_i}=\norm b$.

\subsubsection{}
Montrons maintenant (3) par implications circulaires. Si $k\{T/r\}\to A$ est distinguée, la norme $\norm \cdot_{\mathrm q}$ coïncide avec la norme spectrale de $A$, et (ii) se déduit alors de (2). 
Si (ii) est vraie alors (iii) est vraie car $\widetilde A$ est réduite. Enfin si (iii) est vraie alors on déduit de (2) que $\norm \cdot _{\mathrm q}$
commute aux puissances, ce qui implique qu'elle coïncide avec la semi-norme spectrale de $A$, c'est-à-dire que $k\{T/r\}\to A$ est distingué. 
\end{proof}

\begin{theo}\label{theo-reduit-distingue}
Supposons que le corps $k$ est stable et soit $A$ une algèbre $k$-affinoïde. L'algèbre $A$ est distinguée si et seulement si elle est réduite.
\end{theo} 

\begin{proof}
Comme $\Gamma$ ne joue pas de rôle dans l'énoncé, on peut le supposer égal à $\R_{>0}$
(il va intervenir ci-dessous implicitement à travers les annéloïdes résiduels). 
On sait que si $A$ est distinguée, elle est réduite. Réciproquement, supposons $A$ réduite
et munissons-la de sa norme spectrale $\norm \cdot$. 
La réduction $\widetilde A$ est de type fini sur $\widetilde k$ ; par conséquent,
$\mathfrak d(\widetilde A_{\neq 0})$ 
est un monoïde de type fini sur $\mathfrak d(\widetilde k\gpm)$, ce
qui veut dire que le monoïde des valeurs non nulles de la norme spectrale de $A$
est de type fini sur $\abs{k\gpm}$.
Il existe donc un polyrayon $r=(r_1,\ldots, r_n)$
tel que $\widehat{k(T/r)}$ ne soit pas trivialement valué et tel que
$\norm A\subset \abs k\cdot r_1^\Z\cdot\ldots \cdot r_n^\Z$. 

Les lemmes \ref{lem-kr-schauder}
et 
\ref{lem-norm-ar} assurent qu'on  peut appliquer le lemme \ref{lem-distingue-changebase}
à l'extension $\widehat{k(T/r)}$ de $k$, ce qui signifie qu'on peut s'assurer du caractère distingué de $A$
après passage au corps $\widehat{k(T/r)}$ (notons que l'algèbre $A_{\widehat{k(T/r)}}$ est encore réduite d'après le 
lemme \ref{lem-norm-ar}, voir aussi la remarque \ref{rem-ar-reduit}). 
Par ailleurs, le corps
valué
 $\widehat{k(T/r)}$
est stable (\ref{rappel-kr-stable}), et $\norm {A_{\widehat{k(T/r)}}}=\norm A\cdot  r_1^\Z\cdot\ldots \cdot r_n^\Z$
par le lemme \ref{lem-norm-ar}.
Par conséquent on peut supposer, quitte à étendre les scalaires à $\widehat{k(T/r)}$, que la valeur absolue de $k$ n'est pas triviale et que $\norm A\subset \abs k$ (avec égalité si $A\neq \{0\}$), ce qui implique que $A$ est strictement $k$-affinoïde. 

Récapitulons : $A$ est une algèbre strictement $k$-affinoïde
réduite sur le corps $k$ qui est non trivialement valué et stable, et $\norm A\subset \abs k$. 
En vertu de \cite{bosch-g-r1984}, 6.4.3, Th. 1, ceci entraîne que $A$
est distinguée (et même $\{1\}$-distinguée). 
\end{proof}

Le lemme suivant est un cas particulier facile de résultats très généraux de descente
établis par Conrad et Temkin
dans leur prépublication \cite{conrad-tXXX}. 
Nous en donnons une preuve directe pour la commodité du lecteur. 

\begin{lemm}\label{lem-descente-imferme}
Soit $u\colon A\to B$ un morphisme d'algèbres $k$-affinoïdes, et soit $k^{\mathrm a}$ une
clôture algébrique de $k$. Si $u_{\alg k}\colon A_{\alg k}
\to B_{\alg k}$ est une surjection admissible
alors $u$ est une surjection admissible.
\end{lemm}

\begin{proof}
Comme $\Gamma$ n'apparaît pas dans l'énoncé, 
on peut le supposer égal à $\R_{>0}$. Posons $C=A_{\alg k}$
et $D=B_{\alg k}$. 
L'annéloïde $\widetilde C$ est
fini sur $\widetilde A\otimes_{\widetilde k}\widetilde{k^{\mathrm a}}$
d'après \cite{temkin2004}, Prop. 3.1 (v),
et $\widetilde A\to \widetilde C$ est injective
car
$A\to A_{\alg k}$ préserve les semi-normes spectrales
puisque $\mathscr M(A_{\alg k})\to \mathscr M(A)$ est surjectif. Le morphisme $\widetilde A\to \widetilde C$ est en conséquence une 
injection entière. De même, $\widetilde B\to \widetilde D$ est une injection entière. 
Par ailleurs, notre hypothèse sur $u_{\alg k}$ implique que $\widetilde D$ est finie sur $\widetilde C$, et \emph{a fortiori}
entière sur $\widetilde A$. Par conséquent, la sous-$\widetilde A$-algèbre $\widetilde B$ de $\widetilde D$ est entière sur $\widetilde 
A$ (et même fine sur $\widetilde A$ puisque $\widetilde B$ est de type fini sur $\widetilde k$). 
Cela signifie que $\mathscr M(B)\to \mathscr M(A)$
est sans bord (\cite{temkin2004}, prop. 3.4) ; comme il s'agit d'un morphisme entre espaces affinoïdes, il est fini
(\cite{berkovich1990}, Cor. 2.5.13). 
Dès lors $B$ est un $A$-module de Banach fini, et $D=B\otimes_AC$. 
Or $C$ est une $A$-algèbre fidèlement plate (\cite{berkovich1993}, 
Lemma 2.1.2) ; par conséquent,
la surjectivité de $u_{\alg k}$ entraîne celle de $u$. 
\end{proof}

\begin{lemm}\label{lem-univ-distingue}
Soit
$A$ une algèbre $k$-affinoïde $\Gamma$-stricte
géométriquement réduite, soit $r$ un
polyrayon constitué d'éléments de $\Gamma$ et soit $u\colon k\{T/r\}\to A$ un
morphisme surjectif. Soit $(a_1,\ldots, a_m)$
une famille fortement génératrice de $\mathrm{Ker}(u)$.
Les assertions suivantes sont équivalentes : 

\begin{enumerate}[i]
\item la $\widetilde k$-algèbre $\widetilde k[r\backslash T]/(\widetilde{a_1},\ldots, \widetilde{a_m})$
est géométriquement réduite ; 
\item pour toute extension complète $L$ de $k$, le morphisme $u_L$ est distingué ; 
\item le morphisme $u$ est distingué et il existe une extension complète algébriquement close
$L$ de $k$ telle que $u_L$ soit distingué ;
\item le morphisme $u$ est distingué et $\widetilde A$ est géométriquement réduite. 
\end{enumerate}
\end{lemm}

\begin{rema}\label{rem-univ-distingue}
Comme tout idéal de $k\{T/r\}$ admet une famille
fortement génératrice
(proposition \ref{prop-fort-gen}), les assertions (ii), (iii) et (iv) 
du lemme ci-dessus sont équivalentes pour tout morphisme surjectif
$u\colon k\{T/r\}\to A$.  
\end{rema}

\begin{proof}
Notons pour commencer que
la proposition \ref{prop-fort-adm}
assure que $(a_1,\ldots, a_m)$ est une famille fortement génératrice de $\mathrm{Ker}(u_L)$ pour toute extension
complète $L$ de $k$. 

L'implication (i)$\Rightarrow$(ii) découle de la proposition
\ref{distingue-residu-reduit}, et (ii)$\Rightarrow$(iii) est évidente. Supposons que (iii) soit vraie. Comme $u$ et $u_L$ sont distingués, 
la proposition \ref{distingue-residu-reduit} entraîne que les applications naturelles 
$\widetilde k[r\backslash T]/(\widetilde{a_1},\ldots, \widetilde{a_m})\to \widetilde A$
et $\widetilde L[r\backslash T]/(\widetilde{a_1},\ldots, \widetilde{a_m})\to \widetilde {A_L}$
sont des isomorphismes. Par conséquent le morphisme naturel
$\widetilde A\otimes_{\widetilde k}\widetilde L\to \widetilde{A_L}$ est un isomorphisme ; comme le 
but de cet isomorphisme est réduit, sa source l'est aussi et $\widetilde A$ est donc géométriquement réduite. 

Enfin supposons que (iv) soit vraie. Comme $u$ est distingué, la proposition \ref{distingue-residu-reduit}
assure que $\widetilde k[r\backslash T]/(\widetilde{a_1},\ldots, \widetilde{a_m})\to \widetilde A$ est un isomorphisme. 
Si $L$ est une extension complète de $k$ le quotient
$\widetilde L[r\backslash T]/(\widetilde{a_1},\ldots, \widetilde{a_m})$ est donc isomorphe à $\widetilde A\otimes_{\widetilde k}\widetilde L$
et est par conséquent réduit d'après nos hypothèses, d'où (i).
\end{proof}

\begin{defi}
Soit $A$ une algèbre $k$-affinoïde. Si $r$ est un polyrayon, 
un morphisme surjectif $u\colon k\{T/r\}\to A$ est dit \emph{universellement}
distingué si $u_L$ est distingué pour toute extension complète $L$ de
$k$.

L'algèbre $A$ est dite universellement distinguée (resp. universellement
$\Gamma$-distinguée) s'il existe un polyrayon $r$ (resp. un polyrayon $r$ constitué
d'éléments de $\Gamma$)
et une surjection 
universellement distinguée $u\colon k\{T/r\}\to A$. Si c'est le cas on peut toujours
faire en sorte que $\norm{u(T_i)}=r_i$ pour tout $i$, par
le même procédé que celui décrit au  \ref{distingue-bon-rayon} (resp. 
\ref{rem-distingue-bonrayon}).
\end{defi}

\begin{theo}\label{theo-affinoide-ponctuel}
Soit $A$ une algèbre $k$-affinoïde
$\Gamma$-stricte et géométriquement réduite. Il existe une extension finie $L$ de $k$,
qu'on peut prendre séparable si $\abs{k\gpm}\neq \{1\}$,
telle que $A_L$ soit universellement
$\Gamma^\Q$-distinguée. 
\end{theo}

\begin{proof}
Quitte à remplacer $\Gamma$ par $\Gamma^\Q$
(ce qui ne change pas la notion d'algèbre $\Gamma$-stricte) on
peut supposer que $\Gamma$ est divisible. Pour toute extension complète $F$ de $k$, l'algèbre $A_F$ est réduite
et on la considérera toujours comme munie de sa norme spectrale
qu'on notera $\|\cdot\|$
(cette notation n'est pas ambiguë car comme $\mathscr M(A_F)\to \mathscr M(A)$ est surjective, $A\hookrightarrow A_F$ préserve les normes spectrales).
Soit $k^{\mathrm a}$ une clôture algébrique de $k$. Le corps $\alg k$ est stable, et l'algèbre
$A_{\alg k}$ est réduite. Le théorème \ref{theo-reduit-distingue}
assure alors que que $A_{\alg k}$ est distinguée. Il existe donc un polyrayon $r=(r_1,\ldots, r_n)$ 
et une surjection distinguée $v\colon \alg k\{T/r\}\to A_{\alg k}$. Puisque
$A$
est $\Gamma$-stricte et $\Gamma$ divisible, on a 
$\norm{A_{\alg k}}\subset \abs{\alg k}\cdot \Gamma$ ; 
en vertu de la remarque
\ref{rem-distingue-bonrayon} on peut donc supposer que chacun des $r_i$ appartient à $\Gamma$. 
Il résulte par ailleurs de  \ref{distingue-bon-rayon}
que
quitte à perturber un peu $v$, on peut supposer qu'il existe une extension finie $L$ de $k$, et même
finie séparable si $\abs{k\gpm}\neq \{1\}$,
telle que $v(T_i)\in A_L$ pour tout $i$. 
Soit $J$ le noyau de $v$. La proposition \ref{prop-fort-gen}
assure que $J$ possède une famille finie fortement génératrice $(a_1,\ldots, a_m)$, 
et il découle de \ref{def-fortgen}
qu'on peut supposer, quitte à agrandir $L$ et à perturber
un peu les $a_i$, que
les $a_i$ appartiennent à $L\{T/r\}$. 

Notons $u$ le morphisme $L\{T/r\}\to A_L$ qui envoie $T_i$ sur $v(T_i)$ pour tout $i$, et $I$
l'idéal $(a_1,\ldots, a_m)$ de $L\{T/r\}$. On a par construction $v=u_{\alg k}$ et $J=I_{\alg k}$. 
Comme $v$ est une surjection admissible, le lemme \ref{lem-descente-imferme}
assure que $u$ est une surjection admissible. 
Puisque $v=u_{\alg k}$
on a $\mathrm{Ker}(v)=\mathrm{Ker}(u)_{\alg k}$, et
on a d'autre part $\mathrm{Ker}(v)=J=I_{\alg k}$. 
Par fidèle platitude de $A_{\alg k}$
sur $A_L$
(\cite{berkovich1993}, Lemma 2.1.2), il vient 
$\mathrm{Ker}(u)=I$.
Comme $(a_1,\ldots, a_m)$ engendre fortement $J$, 
il résulte de la proposition \ref{prop-fort-adm}
(et de la fin de \ref{def-fortadm}) que 
$(a_1,\ldots, a_m)$ engendre fortement $I$. Le morphisme
$v$ étant distingué, la proposition \ref{distingue-residu-reduit}
entraîne que $\widetilde {k^a}[T/r]/(\widetilde {a_1},\ldots, \widetilde {a_r})$ est réduite.
Il résulte alors du lemme \ref{lem-univ-distingue}
que $u$ est universellement distingué. En conséquence, $A_L$
est universellement distinguée. 
\end{proof}

\begin{rema}
On ne peut pas espérer que l'énoncé du théorème ci-dessus vaille
en général avec $\Gamma$ à la place de $\Gamma^\Q$. En effet supposons
qu'il
existe un élément $r$ appartenant à $\Gamma^\Q\setminus (\Gamma\cdot \abs{k\gpm}^\Q)$. Pour toute extension finie $L$ de $k$
on a
alors\[\abs{L_r\gpm}=\abs{L\gpm}\cdot r^\Z\not\subset \abs{L\gpm}\cdot\Gamma,\]
si bien que $L_r$ n'est jamais $\Gamma$-distinguée. En revanche, la
présentation canonique $k\{T/r, rS\}/(ST-1)\simeq k_r$ est universellement $\Gamma^{1/N}$-distinguée, où $N$ est l'ordre de $r$ modulo $\Gamma$. 

\end{rema}

\begin{prop}\label{prop-univdistingue-1}
Soit $A$ une algèbre $k$-affinoïde $\Gamma$-stricte géométriquement réduite,
soit $r$ un polyrayon constitué
d'éléments de $\Gamma$ et soit $u\colon k\{T/r\}\to A$
une surjection admissible.  
On note $I$ le noyau de $u$. 
Soit $(a_1,\ldots, a_m)$ une famille génératrice de $I$
constituée d'éléments non nuls
et telle que pour toute extension complète $L$
de $k$, les propriétés suivantes soient satisfaites : 

\begin{enumerate}[a]
\item $\widetilde L[r\backslash T]/(\widetilde{a_1},\ldots
\widetilde{a_m})$ est réduite ; 
\item la flèche naturelle
$\widetilde L[r\backslash T]/(\widetilde{a_1},\ldots
\widetilde{a_m})\to \widetilde{(A_L)}$ est injective. 
\end{enumerate}
La famille $(a_i)$ engendre alors fortement $I$, et $u$ est universellement
distingué. 
\end{prop}

\begin{proof}
Soit $\norm \cdot$ la norme spectrale de $A$.
Il suffit de montrer que la famille $(a_i)$ engendre fortement $I$. Il découlera alors de 
 la proposition \ref{prop-fort-adm}, de l'hypothèse (a) et de la proposition \ref{distingue-residu-reduit}
que $u_L$ est distingué pour toute extension complète $L$ de $k$, et donc
que $u$ est universellement distingué. 

Par ailleurs pour montrer que la famille $(a_i)$ engendre fortement $I$, il suffit de le faire après
une extension quelconque des scalaires (prop.  \ref{prop-fort-adm}). Comme nos hypothèses sont elles-mêmes invariantes
par une telle extension, on peut supposer que $k$ est algébriquement clos, et que $\Gamma\subset \abs{k\gpm}\neq \{1\}$ (l'algèbre $A$ est alors strictement $k$-affinoïde).
Cette dernière hypothèse permet
de se ramener aussitôt par renormalisation des indéterminées au cas où $r=1$. Comme on peut multiplier chacun des $a_i$ par un scalaire
sans changer le caractère fortement générateur éventuel de $(a_i)$ ni nos hypothèses sur l'algèbre $\widetilde k[r\backslash T]/(\widetilde{a_1},\ldots
\widetilde{a_m})$, on peut 
supposer que $\|a_i\|=1$ pour tout $i$.

Soit $J$ le noyau de $k^\circ\{T\}\to A$.
Pour que la famille $(a_1,\ldots, a_m)$ engendre fortement $I$, 
il faut et il suffit qu'elle engendre 
$J$ sur $k^\circ\{T\}$.  La condition énoncée est en effet clairement nécessaire par définition d'une famille fortement
génératrice. Réciproquement, supposons que $(a_1,\ldots, a_m)$ engendre $J$ sur $k^\circ\{T\}$
et soit $a\in k\{T\}$ ; montrons qu'on peut écrire $a$ sous la forme $\sum \alpha_i a_i$ avec $\norm{\alpha_i}\leq \norm a$ pour
tout  $i$. C'est évident si $a=0$ ; sinon $\norm a \in \abs{k\gpm}$, si bien que quitte à diviser $a$ par un scalaire convenable, 
on peut supposer que $\norm a =1$ ; mais l'assertion requise découle
alors directement du fait que  $(a_1,\ldots, a_m)$ engendre $J$ sur $k^\circ\{T\}$.

Il suffit donc pour conclure de démontrer
que $(a_1,\ldots, a_m)$ engendre $J$ sur $k^\circ\{T\}$. 
Comme 
$A^\circ$ est plat
sur $k^\circ$, l'idéal $J$ de $k^\circ\{T\}$
est de type fini, \emph{cf.}
\cite{frg1}, Prop. 1.1 (c).
Soit $b$ un élément non nul de $J$ et soit $\mu$ un élément
de $k^\circ$
tel que $\|b\|=\abs \mu$. L'image de $b/\mu$ dans $A$ est nulle, et
l'image de $b/\mu$ dans $\widetilde k^1[T]/(\widetilde{a_1},\ldots
\widetilde{a_m})$ est donc également nulle
par injectivité du morphisme
$\widetilde k[T]/(\widetilde{a_1},\ldots
\widetilde{a_m})\to \widetilde A$.
Il existe donc $\lambda_1,\ldots, \lambda_m,\nu$ dans $k^\circ\{T\}$ et $\epsilon$
dans $k^{\circ \circ}$ tels
que $b/\mu=\sum_i \lambda_ia_i+\epsilon\nu$. Comme $u(a_i)=0$ pour tout $i$, il vient $\epsilon u(\nu)=0$ et donc $\epsilon=0$ ou $u(\nu)=0$, c'est-à-dire $\nu\in J$ ; dans les deux cas on a
$b\in I+k^{\circ \circ}J$.  
Le $k^\circ\{T\}$-module de type fini $J/I$ satisfait donc l'égalité $J/I=k^{\circ \circ}J/I$. Le lemme de Nakayama assure alors
l'existence d'un élément $\omega \in k^{\circ \circ}\{T\}$ tel que $(1+\omega)(J/I)=0$. Puisque $1+\omega$ est inversible dans $k^\circ\{T\}$
il vient $J=I$. 
\end{proof}

\subsection{Un critère utile}\label{ss-critere-utile}
Reprenons les notations de la proposition \ref{prop-univdistingue-1}.
Nous allons énoncer un critère pour que les hypothèses
de celle-ci
soient satisfaites, auquel nous ferons appel lors de la preuve du théorème principal. 

Supposons que
la $\widetilde k$-algèbre
$\widetilde k[r\backslash T]/(\widetilde{a_1},\ldots
\widetilde{a_m})$
est géométriquement réduite. Soit $L$ une extension
complète de $k$. Elle satisfait alors l'hypothèse (a). Par ailleurs,
le morphisme
naturel de $\widetilde L[r\backslash T]/(\widetilde{a_1},\ldots
\widetilde{a_m})$ vers $\widetilde{(A_L)}$
est entier d'après
\cite{temkin2004}, Prop. 3.1 (v), et il induit donc un morphisme fermé
de $\spec \widetilde{(A_L)}$
vers $\spec \widetilde L[r\backslash T]/(\widetilde{a_1},\ldots
\widetilde{a_m})$. Il en
résulte, l'annéloïde
$\widetilde L[r\backslash T]/(\widetilde{a_1},\ldots
\widetilde{a_m})$ étant réduit,  que
$\widetilde L[r\backslash T]/(\widetilde{a_1},\ldots
\widetilde{a_m})$ s'injecte dans $ \widetilde{(A_L)}$
si et seulement si
$\spec \widetilde{(A_L)}$
se surjecte sur
$ \spec\widetilde L [r\backslash T]/(\widetilde{a_1},\ldots
\widetilde{a_m})$. 

Pour que cette dernière condition soit satisfaite, nous allons montrer qu'il suffit que
le schémoïde
$\mathsf X:=
\spec \widetilde k[r\backslash T]/(\widetilde{a_1},\ldots,
\widetilde{a_m})$
possède les deux propriétés
suivantes, en notant $\mathsf E$ l'ensemble
de ses points maximaux : 
\begin{itemize}
\item [$(\alpha)$] pour tout $\xi \in \mathsf E$, la composante irréductible $\overline{\{\xi\}}$
de $\mathsf X$ est géométriquement irréductible ; 
\item [$(\beta)$] pour tout $\xi \in \mathsf E$,
il existe un élément $f$ non nul de $k\{T/r\}$ tel que
$\widetilde f(\eta)=0$ pour tout $\eta\in \mathsf E\setminus \{\xi\}$, un domaine 
affinoïde
$V$ de $\mathscr M(A)$ purement de dimension $d:=\dim \overline{\{\xi\}}$, 
et un point $v$ de $V$ tel que
$\abs{u(f)(v)}=\norm f$. 
\end{itemize}
Supposons en effet que ce soit le cas, et montrons que chacun des points maximaux de
$\mathsf X_{\widetilde L}$ appartient à l'image de $\spec \widetilde {A_L}$ ; cela permettra de conclure, puisque 
la flèche $\spec \widetilde{(A_L)}
\to \mathsf X_{\widetilde L}$
est fermée. 

Soit donc $\zeta$ un point maximal de $\mathsf X_{\widetilde L}$ et soit $\xi$ son image sur $\mathsf X$ ; l'hypothèse $(\alpha)$ assure
que $\zeta$ est l'unique antécédent de $\xi$ sur $\mathsf X_{\widetilde L}$. Soit $W$ le domaine affinoïde de $\mathscr M(A_L)$ défini par l'égalité $\abs{u(f)}=\norm f$. L'hypothèse $(\beta)$ assure que $W$ rencontre $V_L$, qui est purement de dimension $d$ ; en particulier, $W$ rencontre au moins une composante irréductible de dimension $d$ de $\mathscr M(A_L)$, si bien que l'ouvert $\widetilde W$ de 
$\spec \widetilde{A_L}$ rencontre au moins une composante irréductible de dimension $d$ de ce dernier
(lemme \ref{lem-dim-xtilde}) ; soit $\omega$ son point générique.
Par définition de $W$, la fonction $\widetilde{u(f)}$ ne s'annule pas en $\omega$ ; par conséquent, l'image de $\omega$ sur $\mathsf X$ n'appartient pas au lieu des zéros de $\widetilde f$, lequel contient
$\overline {\{\eta\}}$ pour tout point maximal $\eta$ différent de $\xi$
(en vertu de l'hypothèse $(\beta)$). 
Il s'ensuit que l'image de $\omega$ sur $\mathsf X_L$ est située sur $\overline {\{\zeta\}}$. Comme $\spec \widetilde {A_L}\to \mathsf X_L$
est fini et comme $\omega$ est le point générique d'une composante irréductible de dimension $d$, on en déduit que l'image de $\omega$ sur $\mathsf X_L$
est égale à $\zeta$.

\section{Le théorème principal}

\subsection{}
Nous suivons toujours les conventions générales
décrite aux paragraphes \ref{ss-conv-gen} \emph{et sq.}

\subsubsection{}
Nous ferons un grand usage dans ce qui va suivre de la théorie
de la réduction des germes d'espaces analytiques
introduite par Temkin dans \cite{temkin2004}, et plus précisément
de sa déclinaison $\Gamma$-stricte développée au chapitre 3 de \cite{ducros2018} (dans lequel nous décrivons en détail la 
théorie).

Rappelons simplement ici quelques notations utiles.
Si $F\hookrightarrow L$ est une extension de corpoïdes $\Gamma$-gradués, on note $\mathbf P_{L/F}$ l'ensemble
des (classes d'équivalence de) valuations sur $L$ qui sont triviales sur $F$. Pour toute partie $E$ de $L$, on note $\P_{L/F}\{E\}$ le sous-ensemble de $\P_{L/F}$ formé des valuations dont l'annéloïde contient $E$. Les $\P_{L/F}\{E\}$ 
pour $E$ fini engendrent une topologie dont on munit $\P_{L/F}$; ce dernier est alors quasi-compact.

Si $(X,x)$ est un germe d'espace $k$-analytique $\Gamma$-strict, nous noterons $\widetilde{(X,x)}$ la «partie $\Gamma$-stricte» 
de sa réduction à la Temkin, qui est notée $\widetilde{(X,x)}^\Gamma$ dans \cite{ducros2018} (voir les paragraphes 3.5.3 et 3.5.9 de \emph{op. cit.}
pour sa définition et ses principales propriétés). C'est un ouvert quasi-compact et non vide de $\P_{\hrt y/\widetilde k}$. 

\subsubsection{}
Nous utiliserons également la notion de platitude en géométrie analytique telle qu'elle est introduite et étudiée dans \cite{ducros2018} (voir notamment le chapitre 4 de \emph{op. cit.}). 

Nous nous servirons à plusieurs reprises du résultat suivant : si $p\colon Y\to X$ est un morphisme plat entre espaces $k$-analytiques compacts et $\Gamma$-stricts alors $p(Y)$ est un domaine analytique (compact) et $\Gamma$-strict de $X$. Lorsque $\Gamma=\{1\}$ c'est un résultat de Raynaud (\emph{cf.}
\cite{frg2}, Cor. 5.11). On trouvera dans 
 \cite{ducros2018} une démonstration du cas général par réduction au cas où $\Gamma=\{1\}$, ainsi qu'une nouvelle preuve de ce dernier sans modèles formels (ils sont remplacés par la réduction
 des germes d'espaces analytiques) ; \emph{cf. op. cit.}, théorèmes 9.1.3 et 9.2.1.
 
 Nous aurons également
 besoin d'un énoncé technique sur le lieu d'équidimensionalité d'un morphisme plat, que nous démontrons en appendice
 (théorème \ref{theo-dimrel-s1}).

\begin{defi}\label{def-gamma-symp}
Soit $A\to B$ un morphisme entre algèbres $k$-affinoïdes $\Gamma$-strictes ; posons $X=\mathscr M(A)$ et $Y=\mathscr M(B)$. 
Une présentation $B\simeq A\{T/r\}/(a_1,\ldots, a_m)$
de $B$ sur $A$ (ou de $Y$ sur $X$) est dite
\emph{$\Gamma$-sympathique}
si elle possède les propriétés suivantes, en notant
$D$ le polydisque relatif $\mathscr M(A\{T/r\})$ :

\begin{enumerate}[1]
\item le polyrayon $r$ est constitué d'éléments de $\Gamma$ ; 
\item pour tout $i$ compris entre $1$ et $m$ la
semi-norme spectrale $\rho_i$
de $a_i$ appartient à 
 $\Gamma$, et pour tout $x\in X$ la norme 
spectrale de $a_i|_{D_x}$  est encore égale à $\rho_i$ ; 
\item pour tout point $x$ de $\mathscr M(A)$,
la famille $(a_1|_{D_x},\ldots, a_m|_{D_x})$ est une famille fortement
génératrice de l'idéal qu'elle engendre dans $\hr x\{T/r\}$ ; 
\item le morphisme
naturel $p\colon \spec \widetilde A[r\backslash T]/(\widetilde{a_1}, \ldots, \widetilde {a_m})\to \widetilde X$ est plat et à fibres 
géométriquement réduites ; 
\item les composantes irréductibles des fibres de $p$ sont géométriquement irréductibles ; 
\item il existe un recouvrement ouvert fini
$(\Omega_i)$ de  $\spec \widetilde A[r\backslash T]/(\widetilde{a_1}, \ldots, \widetilde {a_m})$ tel que pour tout $i$ et tout $\xi \in \widetilde X$ l'intersection $p\inv(\xi)\cap \Omega_i$ soit ou bien vide, ou bien une composante connexe de $p\inv(\xi)$. 
\end{enumerate}
\end{defi}

\subsection{}
Soient $A$ et $B$ comme ci-dessus, soit $B\simeq A\{T/r\}/(a_1,\ldots, a_m)$ une présentation
$\Gamma$-sympathique de $B$ et soit $D$ le polydisque relatif
$\mathscr M(A\{T/r\})$. 

\subsubsection{}\label{sss-avantages-sympathique}
Soit $L$ une extension complète de $k$ et soit $x$ un $L$-point
de $\mathscr M(A)$. Il découle découle alors des conditions (1) à (4)
de la définition d'une présentation sympathique, de la proposition \ref{prop-fort-adm} couplée à la 
fin de \ref{def-fortadm}, ,
de la proposition \ref{distingue-residu-reduit}
et du lemme \ref{lem-univ-distingue}
que $(a_1|_{D_x},\ldots, a_m|_{D_x})$ est une famille
fortement
génératrice de l'idéal de $L\{T/r\}$ qu'elle engendre,
que la surjection $L\{T/r\}\to L\{T/r\}/(a_1|_{D_x}\ldots, a_m|_{D_x})$ est universellement
distinguée, 
et que $\widetilde {Y_x}\to \spec \widetilde L[r\backslash T]/(\widetilde{a_1|_{D_x}}, \ldots, 
\widetilde{a_m|_{D_x}})$
est un isomorphisme. 
Une présentation $\Gamma$-sympathique 
d'un morphisme entre espaces $k$-affinoïdes
$\Gamma$-stricts
fournit ainsi entre 
autres une description uniforme et universelle des
réductions (graduées) de ses fibres.

\subsubsection{}
Soit $C$ une algèbre affinoïde $\Gamma$-stricte sur une extension complète de $k$ et soit $A\to C$ un morphisme ; rappelons 
qu'on désigne par $B_C$ l'algèbre $B\hotimes_AC$. 
L'isomorphisme $B_C\simeq C\{T/r\}/(a_1,\ldots, a_m)$ est 
alors une présentation $\Gamma$-sympathique de $B_C$ sur $C$. 
C'est en effet une conséquence de la stabilité des propriétés (1) à (6) de la définition par extension des scalaires, stabilité qui se voit comme suit :
elle est évidente en ce qui concerne (1), (2), (4) et (5) ;  en ce qui concerne (3), cela a été expliqué
en \ref{sss-avantages-sympathique} ;
et en ce qui concerne (6),
elle est immédiate une fois remarqué qu'en vertu de (5), les composantes connexes des fibres de $p$ sont géométriquement connexes.

\subsubsection{}
Si $\Gamma=\{1\}$
et si $A$ est distinguée
(ce qui est par exemple le cas dès que $k$ est algébriquement clos
et que $A$ est réduite), on peut déduire de la présentation $\{1\}$-sympathique de $B$ sur $A$ que nous nous sommes donnée
un modèle formel plat et
à fibres réduites de $\mathscr M(B)\to \mathscr M(A)$, comme en atteste la proposition suivante. 

\begin{prop}\label{prop-retrouve-fibred}
Supposons que $\Gamma=\{1\}$ (ce qui entraîne que $k$ n'est pas trivialement
valué). Soit $A\to B$ un morphisme
entre algèbres strictement $k$-affinoïdes, 
et soit 
$B\simeq A\{T\}/(a_1,\ldots, a_m)$
une présentation $\{1\}$-sympathique
de $B$ sur $A$. 
Si $A$ est $\{1\}$-distinguée il existe $b_1,\ldots, b_\ell$
appartenant à $A^\circ\{T\}$
tels que le morphisme 
\[q\colon \mathrm{Spf}\;(A^\circ\{T\}/(a_1,\ldots, a_m,b_1,\ldots, b_\ell))\to \mathrm{Spf}\;A^\circ\]
possède les propriétés suivantes : 
\begin{enumerate}[a]
\item la fibre générique de $q$
s'identifie
à $\mathscr M(B)\to \mathscr M(A)$ ;
\item $q$ est plat ; 
\item le morphisme induit par $q$ entre les fibres spéciales
de ses source et but 
s'identifie
à
$\spec \widetilde A[T]/(\widetilde{a_1},\ldots, \widetilde{a_m})\to \spec \widetilde A$ ;
\item les fibres de $q$ sont géométriquement réduites ; 
\item les  composantes irréductibles des fibres de
$q$ sont géométriquement irréductibles ; 
\item il existe un recouvrement ouvert fini $(\Omega_i)$ de
$\mathrm{Spf}\;(A^\circ\{T\}/(a_1,\ldots, a_m,b_1,\ldots, b_\ell))$
tel que pour tout indice $i$, chaque fibre de $q|_{\Omega_i}$
soit ou bien vide, ou bien
une composante connexe de la fibre correspondante de $q$. 
\end{enumerate}
\end{prop}

\begin{proof}
Mentionnons tout d'abord que 
comme $\Gamma=\{1\}$, 
les réductions graduées sont ici les réductions usuelles ; 
les applications ordinaires de réduction sont celles que nous notons
$a\mapsto \widetilde a^1$ ; remarquons que l'on a $\widetilde{a_i}^1=\widetilde{a_i}$ pour tout $i$ puisque $a_i$ est de norme spectrale $1$. On note $D$ le polydisque unité relatif $\mathscr M(A\{T\})$. 

La proposition 1.1 de \cite{frg1}
assure qu'il existe
une famille \emph{finie}
$(b_1,\ldots, b_\ell)$ d'éléments de $A^\circ\{T\}$
telle que le quotient de $A^\circ\{T\}/(a_1,\ldots, a_m)$ par son idéal de
$k^\circ$-torsion
soit isomorphe au quotient
$A^\circ\{T\}/(a_1,\ldots, a_m,b_1,\ldots, b_\ell)$. Nous allons montrer que cette famille $(b_1,\ldots, b_\ell)$ convient. 
Il suffit de s'assurer que (a), (b) et (c) sont vérifiée, car les assertions (d), (e), et (f) découleront alors
de (c) et de la définition d'une présentation $\{1\}$-sympathique. 

\subsubsection{}
Par construction, l'image de chacun des
$b_i$ dans $A\{T\}/(a_1,\ldots, a_m)$ est nulle (car de $k^\circ$-torsion). 
Par conséquent, $A\{T\}/(a_1,\ldots, a_m,b_1,\ldots, b_\ell)$ s'identifie à $B$, d'où (a). 

\subsubsection{}
Posons $\mathfrak B=A^\circ\{T\}/(a_1,\ldots, a_m,b_1,\ldots, b_\ell)$. 
On a
\[\mathfrak B\otimes_{A^\circ}\widetilde A
=\widetilde A[T]/(\widetilde{a_1},\ldots, \widetilde{a_m}, \widetilde{b_1}^1,\ldots, \widetilde{b_\ell}^1).\]
Soit $\xi\in \spec \widetilde A$ et soit $x$ un point de $\mathscr M(A)$
se spécialisant en $\xi$. On a 
\[\mathfrak B\widehat\otimes_{A^\circ}\mathrm{Spf}\;\hr x^\circ
=
\hr x^\circ\{T\}/(a_1|_{D_x},\ldots, a_m|_{D_x},
b_1|_{D_x}\ldots, b_\ell|_{D_x}).\]
Comme les $(a_i|_{D_x})$ forment une famille fortement
génératrice de l'idéal qu'ils engendrent dans $\hr x\{T\}$, c'est une famille
génératrice de l'intersection de cet idéal avec $\hr x^\circ\{T\}$. Par conséquent, 
les $b_i|_{D_x}$ appartiennent à l'idéal de $\hr x^\circ\{T\}$ engendré par les $a_i|_{D_x}$, 
si bien que l'image de $\widetilde{b_i|_{D_x}}^1$ dans $\hrt x[T]/(\widetilde{a_1|_{D_x}}, \ldots, \widetilde{a_m|_{D_x}})$
est nulle pour tout $i$. On en déduit que pour tout $i$, l'image de $\widetilde{b_i}^1$
dans $\widetilde A[T]/(\widetilde{a_1},\ldots,\widetilde{a_m})$ s'annule 
en tout point de $(\spec \widetilde A[T]/(\widetilde{a_1},\ldots,\widetilde{a_m}))_\xi$. 
Ceci valant pour tout $\xi\in \spec \widetilde A$, l'image de $\widetilde{b_i}^1$
dans $\widetilde  A[T]/(\widetilde{a_1},\ldots,\widetilde{a_m})$ est nilpotente. Or comme 
$\spec \widetilde  A[T]/(\widetilde{a_1},\ldots,\widetilde{a_m})\to \spec \widetilde A$ est plat
et à fibres réduites, et comme $\widetilde A$ est réduite, $\widetilde  A[T]/(\widetilde{a_1},\ldots,\widetilde{a_m})$
est réduite
(combiner les proposition 6.4.1 et 6.5.3
de \cite{ega42}). Par conséquent, l'image de $\widetilde {b_i}^1$ dans $\widetilde  A[T]/(\widetilde{a_1},\ldots,\widetilde{a_m})$ est nulle pour tout $i$. Il vient
\[\mathfrak B\otimes_{A^\circ}\widetilde A
=\widetilde A[T]/(\widetilde{a_1},\ldots, \widetilde{a_m}).\]

La
fibre spéciale du morphisme de schémas formels
$\mathrm{Spf}\;\mathfrak B\to \mathrm{Spf}\;A^\circ$ 
s'identifie donc à
$\spec \widetilde A[T]/(\widetilde{a_1},\ldots, \widetilde{a_m})\to \spec \widetilde A$, d'où (c).

\subsubsection{}
Nous allons montrer que 
$\mathfrak B$ est plat sur $A^\circ$, ce qui achèvera la démonstration de la proposition. 
En vertu du lemme 1.6 de \cite{frg1}, il suffit
de vérifier que pour tout
$\lambda$ non nul dans $k^{\circ \circ}$, la $(A^\circ /\lambda A^\circ)$-algèbre
$\mathfrak B/\lambda \mathfrak B$ est plate. Fixons donc un tel $\lambda$. Les 
$k^\circ$-algèbres $A^\circ$ et $\mathfrak B$ sont plates par construction ; par conséquent, 
les $k^\circ/(\lambda)$-algèbres de présentation finie
$A^\circ /\lambda A^\circ$ et $\mathfrak B/\lambda \mathfrak B$
sont plates.
Or $k^\circ/(\lambda)$ est local et
son idéal maximal $k^{\circ\circ}/\lambda k^\circ$ est constitué
d'éléments nilpotents. On déduit alors du théorème 11.3.10 de \cite{ega43}
que pour que $\mathfrak B/\lambda\mathfrak B$ soit plat
sur $A^\circ /\lambda A^\circ$, il suffit 
que $\mathfrak B/k^{\circ \circ}\mathfrak B=\widetilde A[T]/(\widetilde{a_1},\ldots, 
\widetilde{a_m})$ soit plat sur $A^\circ/k^{\circ \circ}=\widetilde A$, 
ce qui est le cas par hypothèse (l'égalité $A^\circ/k^{\circ \circ}=\widetilde A$ est une conséquence immédiate du caractère $\{1\}$-distingué de $A$). 
\end{proof}

\begin{theo}\label{theo-main}
Supposons que $\Gamma$ est divisible. Soit
$Y\to X$ un morphisme entre espaces $k$-affinoïdes $\Gamma$-stricts, 
supposé plat et à fibres géométriquement réduites. Il existe une famille finie $(X_i\to X)$
de morphismes dont les images recouvrent $X$ et telle que pour tout $i$, 
les conditions suivantes soient satisfaites : 
\begin{itemize}[label=$\diamond$]
\item l'espace $X_i$ est affinoïde et $\Gamma$-strict ; 
\item le morphisme $X_i\to X$ possède une factorisation $X_i\to X'_i\to X$
où  $X'_i$ est affinoïde et $\Gamma$-strict, où $X_i\to X'_i$ est fini, radiciel et plat et où $X'_i\to X$ est quasi-étale ;
\item si on a de plus $\abs{\hr x\gpm}\neq \{1\}$ pour tout $x\in X$ (c'est notamment le cas dès que la valuation de $k$
est non triviale) alors $X_i\to X$ est quasi-étale  ; 
\item le morphisme $Y\times_X X_i\to X_i$ possède une présentation $\Gamma$-sympathique. 
\end{itemize}

\end{theo}

\begin{rema}
L'hypothèse que $\Gamma$ est divisible n'est pas très contraignante. En effet, supposons donné un morphisme $Y\to X$
entre espaces affinoïdes $\Gamma$-stricts, plat et à fibres réduites, mais sans hypothèse sur $\Gamma$. Comme les espaces $X$ et $Y$ sont $\Gamma$-stricts, ils sont $\Gamma^\Q$-stricts, et le théorème ci-dessus peut alors être appliqué en remplaçant $\Gamma$ par $\Gamma^\Q$. 
\end{rema}

\begin{proof}[Démonstration du théorème \ref{theo-main}]
L'assertion à prouver est locale sur $X$ ; il suffit donc de montrer qu'elle vaut au-dessus d'un voisinage affinoïde $\Gamma$-strict
convenable d'un point $x$ donné sur $X$. On note $A$ et $B$ les algèbres de fonctions analytiques respectives de $X$ et $Y$, et l'on pose $C=B\hotimes_A \hr x$.

\subsubsection{}
Puisque $C$ est géométriquement
réduite par hypothèse, le théorème
\ref{theo-affinoide-ponctuel} assure qu'il existe une extension finie $E$ de $\hr x$, 
séparable si $\hr x$ n'est pas trivialement valué, telle que $C_E$ soit universellement $\Gamma$-distinguée. Choisissons une surjection $v\colon E\{T_1/r_1,\ldots, T_n/r_n\}\to C_E$ universellement distinguée
avec les $r_i$ appartenant à $\Gamma$. Pour tout $i$, on pose $f_i=v(T_i)$ ; on peut
supposer que $\norm {f_i}=r_i$ quel que soit $i$
(remarque \ref{rem-distingue-bonrayon}).

Soit $\abs \cdot$ une valuation appartenant à la réduction 
$\widetilde{(X,x)}$ du germe $(X,x)$. Choisissons un prolongement
de $\abs \cdot$ à $\widetilde E$, encore noté $\abs \cdot$ ; suivant nos conventions générales en matière de corpoïdes valués, 
l'annéloïde 
de la valuaiton $\abs \cdot $ de $\widetilde E$ sera noté
$\widetilde E^\circ$. 
Notons que $\widetilde E$ est de degré de transcendance fini sur $\widetilde k$ ; par conséquent, la hauteur de $\abs \cdot$ est finie. 

Soit $\mathsf B$ l'image de $\widetilde B$ dans $\widetilde {C_E}$.
La $\widetilde E$-algèbre $\widetilde {C_E}$ est entière sur $\widetilde E\cdot \mathsf B$ (cela se prouve de la même manière que les assertions (iv) et (v) de \cite{temkin2004}, Prop. 3.1). Il existe donc une famille finie $\beta_1,\ldots, \beta_\ell$ d'éléments 
de $\widetilde{C_E}$, entiers sur $\widetilde E^\circ \cdot \mathsf B$, et qui engendrent $\widetilde{C_E}$ comme $\widetilde E$-algèbre. 
Quitte à agrandir la famille $(\beta_i)$, on peut supposer qu'elle contient une famille
génératrice de $\mathsf B$ comme $\widetilde k$-algèbre. 
Puisque $C_E$ est universellement
distinguée, la $\widetilde E$-algèbre $\widetilde{C_E}$ est géométriquement réduite
(et $\widetilde{C_\Lambda}=\widetilde{(C_E)}_{\widetilde \Lambda}$
pour toute extension complète $\Lambda$ de $E$).
D'après la version graduée et valuée du théorème de la fibre réduite (corollaire
\ref{coro-fibre-reduite}), il existe une extension valuée
finie $L$
de $\widetilde E$
telle que le normalisé faible $R$ de la $L^\circ$-algèbre $L^\circ[\beta_1,\ldots, \beta_\ell]$
 soit fini sur $L^\circ[\beta_1,\ldots, \beta_\ell]$, telle que les fibres de $\spec R\to \spec L^\circ$ soient géométriquement réduites et à composantes
irréductibles géométriquement irréductibles, et telle que $R_{\Lambda^\circ}$ soit le normalisé faible de $\Lambda^\circ[\beta_1,\ldots, \beta_\ell]$ pour toute extension valuée $\Lambda$ de $\widetilde L$. Quitte à remplacer $E$
par une extension finie convenable (et même finie séparable si la valuation de $\hr x$ n'est pas triviale) on peut supposer que $L=\widetilde E$.
La $\widetilde E^\circ$-algèbre $R$ possède
dès lors les propriétés suivantes : 

\begin{itemize}[label=$\diamond$] 
\item elle est de type fini et $\widetilde E\cdot R=\widetilde{C_E}$ ; 
\item elle contient $\mathsf B$ et est finie
sur $\widetilde E^\circ \cdot \mathsf B$ ; 
\item le morphisme $\spec R\to
\spec \widetilde E^\circ$ est plat à fibres géométriquement réduites,
et les composantes irréductibles de ses fibres sont géométriquement
irréductibles. 
\end{itemize}

\subsubsection{}
Pour tout $i$ compris entre $1$ et $n$, la réduction $\widetilde {f_i}$ est un élément non nul de $\widetilde{C_E}$. Il existe un élément
$\lambda_i$ non nul de $\widetilde E^\circ$ tel que $\lambda_i\widetilde{f_i}$ appartienne à $R$. Relevons $\lambda_i$ en un élément $\ell_i$ de $E$ de valeur absolue 
$\mathfrak d(\lambda_i)$. 
Le morphisme
\[v\colon E\{T_1/r_1,\ldots, T_n/r_n\}\to C_E, \;T_i\mapsto f_i\]
étant une surjection universellement distinguée il en va de même de 
\[ E\{T_1/\abs {\ell_1}\cdot r_1,\ldots, T_n/\abs {\ell_n} \cdot r_n\}\to C_E, \;T_i\mapsto \ell_i f_i.\] 
On peut donc se ramener au cas où chacune des $\widetilde {f_i}$ appartient à $R$. 

Soit $\alpha_1,\ldots, \alpha_p$ une famille d'éléments non nuls de l'annéloïde $R$ telle que les $\widetilde {f_i}$ et les $\alpha_i$ 
engendrent $R$ comme $\widetilde E^\circ$-algèbre. Fixons $i$ entre $1$ et $p$. 

Pour tout $i$, relevons $\alpha_i$ en un élément $a_i$ de $C_E$ de norme spectrale  égale à $\mathfrak d(\alpha_i)$. 
Le morphisme
\[E\{T_1/r_1,\ldots, T_n/r_n,T_{n+1}/\mathfrak d(\alpha_1),\ldots, T_{n+p}/\mathfrak d(\alpha_p)\}\to C_E, \left\{\begin{array}{ccc}
T_i&\mapsto&f_i\;\;\text{si}\;i\leq n\\
T_i&\mapsto&a_{i-n}\;\;\text{sinon}\end{array}\right.\]
est encore une surjection universellement distinguée.
En concaténant la famille des $f_i$ et celle des $a_i$ et en renumérotant, on se ramène finalement
au cas où $(\widetilde {f_1},\ldots, \widetilde{f_n})$ est une famille
génératrice de $R$ comme
$\widetilde E^\circ$-algèbre.

%
%

\subsubsection{}Nous allons maintenant montrer l'existence d'un espace affinoïde $\Gamma$-strict
$X'$ muni d'un morphisme
$X'\to X$ et qui possède les propriétés suivantes : 

\begin{itemize}[label=$\diamond$] 
\item le point $x$ possède un unique antécédent $x'$ sur $X'$, et $\hr {x'}\simeq_{\hr x}E$ ; 

\item le morphisme $X'\to X$ admet une factorisation $X'\to X''\to X_0$
où:\begin{itemize}[label=$\bullet$]
\item $X_0$ est un voisinage affinoïde $\Gamma$-strict de $x$ dans $X$ ; 
\item $X''\to X_0$ est fini étale ; 
\item $X'=X''$ si $\abs{\hr x\gpm}\neq \{1\}$, et $X'\to X''$ est fini, radiciel et plat sinon. 
\end{itemize}
\end{itemize}

\paragraph{Le cas où $\abs{\hr x\gpm}\neq \{1\}$}
Dans ce cas $E$ est séparable sur $\hr x$ et il résulte du théorème 4.4.1
de \cite{berkovich1993}
que l'extension $\hr x\hookrightarrow E$ est induite par un morphisme étale $(X',x')\to (X,x)$ de germes d'espaces $k$-analytiques, 
d'où notre assertion. 

\paragraph{Le cas où $\abs{\hr x\gpm}=1$}
Le corps $\hr x$ n'est alors autre que le corps résiduel $\kappa(x)$ de $\mathscr O_{X,x}$. 
En dévissant l'extension $E$ en ses parties séparable et radicielle on obtient l'existence d'une $\mathscr O_{X,x}$-algèbre 
locale, finie et étale $E_1$ et d'une d'une $E_1$-algèbre locale, finie, radicielle et plate $E_2$ telles que $E_2\otimes_{\mathscr O_{X,x}}\kappa(x)=E$. 
Il existe, là encore par le théorème
4.4.1 de \cite{berkovich1993}, un voisinage affinoïde $\Gamma$-strict $X_0$ de $x$ dans
$X$ tel que $E_1$ soit induite par une $\mathscr O_X(X_0)$-algèbre finie étale ; si $X''$ désigne l'espace affinoïde correspondant le point $x$ a un unique
antécédent $x''$ sur $X''$ et $E_1=\mathscr O_{X'',x''}$. Il existe un voisinage affinoïde $\Gamma$-strict $X''_0$ de $x''$ dans $X''$ tel que $E_2$
soit induite par une $\mathscr O_{X''}(X''_0)$-algèbre finie, radicielle et plate ; 
si $''$ désigne l'espace affinoïde correspondant le point $x''$ a un unique
antécédent $x'$ sur $X'$ et $E_2=\mathscr O_{X',x'}$. Quitte à restreindre $X_0$ et $X''$ on peut supposer que $X''_0=X''$, 
et le morphisme $X'\to X$ possède alors les propriétés requises. 

\paragraph{Choix de fonctions prolongeables à $X'$}\label{par-def-phii}
Quitte à restreindre $X_0, X''$ et $X'$ on peut supposer par approximation
que la fonction $f_i$ 
provient pour tout $i$
d'une
fonction
analytique sur $Y':=Y\times_X X'$, notée $\phi_i$.
Il suffit en effet de s'assurer que les propriétés qu'on impose
aux $f_i$ sont stables par petites perturbations. Or c'est évident en ce qui concerne les propriétés résiduelles, et en ce qui concerne le caractère universellement distingué
du morphisme $v\colon E\{T_1/r_1,\ldots, T_n/r_n\}\to C_E, \;T_i\mapsto f_i$ c'est une conséquence de \ref{distingue-bon-rayon}.

\subsubsection{Simplification du problème}\label{sss-description-simplification}
Il suffit maintenant de montrer l'existence d'un domaine affinoïde $\Gamma$-strict $V$ de $X'$ contenant $x'$, tel que $\widetilde {(V,x')}$ 
contienne la valuation $\abs \cdot$ de $\hrt{x'}=\widetilde E$, et tel que $Y'\times_{X'}V=Y\times_X V$ possède une présentation $\Gamma$-sympathique sur $V$. 

Supposons en effet avoir exhibé un tel domaine.
La preuve du lemme 2.4 de  \cite{ducros2003} (théorème de Gerritzen-Grauert dans le cas non nécessairement strict) assure que $V$ est une union finie de domaines rationnels dont la définition ne fait intervenir que des scalaires appartenant à $\Gamma$ ; on peut donc supposer que $V$ est un tel domaine.
Dans ce cas c'est l'image
réciproque d'un domaine affinoïde $\Gamma$-strict $V_0$
de $X''$ (si $\abs{\hr x \gpm}\neq \{1\}$ alors $X'\simeq X''$ et il n'y a rien à faire ; sinon il suffit d'élever les inégalités définissant $V$ dans $X'$ à une puissance convenable de l'exposant caractéristique de $k$). 
Le morphisme $V\to X$ admet donc une factorisation $V\to V_0\to X$ où $V_0$ est affinoïde et $\Gamma$-strict, où $V_0\to X$ est quasi-étale et où $V\to V_0$ est fini, radiciel et plat,
et est un isomorphisme si $\abs{\hr x \gpm}\neq \{1\}$. 

Comme $V$ est plat sur $X$ et comme $V$ est $\Gamma$-strict l'image de ce morphisme est un domaine analytique compact $\Gamma$-strict $U$
de $X$
contenant $x$. La réduction 
$\widetilde{(U,x)}$ contient par fonctorialité la valuation $\abs \cdot$ de $\hrt x$.

Le point $x$ et la valuation $\abs \cdot \in \widetilde{(X,x)}$ ayant été choisis
arbitrairement, il s'ensuit en vertu de la quasi-compacité
des réductions de germes et de la compacité de $X$
qu'il existe une famille couvrante de $X$ possédant les propriétés requises.

\subsubsection{}
Nous allons
donc désormais chercher à exhiber un domaine $V$ comme au \ref{sss-description-simplification}. On remplace $X$ par $X'$, etc. ce qui permet
de supposer à partir de maintenant que $X'=X, x'=x, Y'=Y$ et $E=\hr x$.
Notons $\mathscr V$ l'ensemble des domaines affinoïdes $\Gamma$-stricts
de $X$ contenant $x$ et tels que $\widetilde{(V,x)}$ contienne 
$\abs \cdot$. Notre but est de construire un domaine appartenant à $\mathscr V$ au-dessus
duquel $Y$
possède une présentation $\Gamma$-sympathique. Nous rédigerons le raisonnement comme suit : nous «restreindrons» $X$ de façon répétée, ce qui voudra dire par convention que nous le remplaçons par un élément bien choisi de $\mathscr V$ ; lorsqu'il arrivera qu'il suffise de remplacer $X$ par un \emph{voisinage}
$\Gamma$-strict de $x$ dans $X$, nous dirons que nous restreignons $X$ \emph{autour de $x$} ; et nous montrerons à la fin que $Y$ possède une présentation sympathique sur $X$.

\paragraph{}\label{p-notations-hauteur}
Soit $\mathsf h$ la hauteur de $\abs \cdot$. On note $\tau_0,\ldots, \tau_{\mathsf h}$ les points de $\spec \hrt x^\circ$, numérotés
de sorte que $\tau_{i+1}$ soit une spécialisation de $\tau_i$ pour tout $i$ entre $0$ et $h-1$ ; le point $\tau_0$ est donc le point générique de
$\spec \hrt x^\circ$, et $\tau_{\mathsf h}$ est son point fermé. 
Pour tout $i$ compris entre $0$ et $\mathsf h-1$ on choisit un élément $\nu_i$ de $\hrt x^\circ$ s'annulant 
en $\tau_{i+1}$ mais pas en $\tau_i$, et on pose $\nu_{\mathsf h}=1$. Le lieu d'inversibilité
$\mathrm D(\nu_i)$ de $\nu_i$ sur $\spec \hrt x^\circ$
est exactement $\{\tau_j\}_{i\geq j\geq 0}$ ; on fixe une fonction $\theta_i\in \mathscr O_{X,x}$ telle que
$\abs {\theta_i(x)}\in \Gamma$ et $\widetilde{\theta_i(x)}=\nu_i$ ; on restreint $X$ autour de $x$ de sorte que $\theta_i\in A$ pour tout $i$ ; on suppose par ailleurs que $\theta_{\mathrm h}=1$. 

\paragraph{}\label{p-union-filtrante1}
Pour tout $V\in \mathscr V$ nous noterons $A_V$ 
et $B_V$ les algèbres de fonctions analytiques respectives de $V$ et $Y\times_X V$ ; nous désignerons par $\mathsf A_V$ l'image de $\widetilde{A_V}$
dans $\hrt x$, et par $\mathsf B_V$ celle de $\widetilde{B_V}$ dans $\widetilde C$ ; remarquons
que $\mathsf B_X$ est ce que nous notons simplement $\mathsf B$, et nous écrirons de même $\mathsf A$ au lieu de $\mathsf A_X$. Pour tout espace affinoïde $Z$ nous désignerons 
par $\norm \cdot_Z$ la semi-norme spectrale de l'anneau
des fonctions analytiques sur $Z$ 
(et le cas échéant nous écrirons $\norm a_Z$ plutôt que que $\norm {a|_Z}_Z$). 

Soit $V\in \mathscr V$. Par construction, $\widetilde{(V,x)}
=\P_{\hrt x/\widetilde k}\{\mathsf A_V\}$ ; 
en particulier, $\hrt x^\circ\supset\mathsf A_V$. 

Réciproquement soit $\theta$ un élément
de $\mathscr O_{X,x}\gpm$ tel que $\abs {\theta(x)}\in \Gamma$
et $\widetilde {\theta(x)} \in \hrt x^\circ$, et soit
$W$ un voisinage affinoïde $\Gamma$-strict de $x$ dans $X$ sur lequel $\theta$ est définie. 

Soit $V$ le domaine affinoïde $\Gamma$-strict de $W$
défini par l'inégalité $\abs \theta\leq \abs{\theta(x)}$ ; le domaine $V$
contient $x$, et $\norm \theta_V=\abs{\theta(x)}$ par construction.
On a
\[\widetilde{(V,x)}=\widetilde{(X,x)}\cap \P_{\hrt x/\widetilde k}\{\widetilde{\theta(x)}\}\] et $\widetilde{(V,x)}$ contient par conséquent $\abs \cdot$, si bien que $V\in \mathscr V$.
L'égalité $\norm \theta_V=\abs{\theta(x)}$ assure que l'image de l'élément $\widetilde \theta$
de $\widetilde{A_V}$ dans $\hrt x$ est alors égale à $\widetilde {\theta(x)}$, et ce dernier appartient par conséquent à $\mathsf A_V$. 

On déduit de ce qui précède que
$\hrt x^\circ$ est la réunion filtrante des $\mathsf A_V$ pour $V$ parcourant 
$\mathscr V$.

\paragraph{}\label{p-union-filtrante2}
Il résulte de \ref{p-union-filtrante1}
qu'on peut restreindre $X$ de sorte que $\norm{\theta_i}_X=\abs{\theta_i(x)}$ pour tout $i$. 
Pour tout $V\in \mathscr V$, notons $V_i$ le domaine affinoïde
$\Gamma$-strict de $V$ défini par l'égalité $\abs {\theta_i}=\abs{\theta_i(x)}$ (remarquons que  comme $\theta_{\mathsf h}=1$
tout domaine
$V$ appartenant à $\mathscr V$ vérifie $\norm{\theta_{\mathsf h}}_V=\abs{\theta_{\mathsf h}(x)}$ , et on a alors $V_{\mathsf h}=V$). 
L'image de l'élément $\widetilde{\theta_i}$ de $\widetilde{A_V}$ dans $\hrt x$ est alors égale à $\nu_i$
(qui appartient donc à $\mathsf A_V$)
et l'annéloïde $\widetilde{A_{V_i}}$ s'identifie à $\widetilde {A_V}[\frac 1{\widetilde{\theta_i}}]$. 
Il vient $\mathsf A_{V_i}=\mathsf A_V[\frac 1{\nu_i}]$. 

L'annéloïde de valuation $\mathscr O_{\spec \hrt x^\circ, \tau_i}$ est égal 
à $\hrt x^\circ [\frac 1{\nu_i}]$. Il peut en conséquence s'écrire comme
la réunion filtrante des $\mathsf A_V[\frac 1{\nu_i}]=\mathsf A_{V_i}$ pour $V$ parcourant $\mathscr V$.

\subsubsection{}
Par construction, les $\widetilde{f_i}$ engendrent $R$, et ils sont
en particulier entiers sur $\hrt x^\circ \cdot \mathsf B$ ; et par ailleurs $R$
contient $\mathsf B$, qui est de type fini sur $\widetilde k$.
En vertu de \ref{p-union-filtrante1}, il existe 
par conséquent $V\in \mathscr V$ tel que les
$\widetilde{f_i}$ soient entiers sur $\mathsf A_V\cdot \mathsf B\subset \mathsf B_V$ et tels que $\mathsf B\subset \mathsf A_V[\widetilde{f_1},\ldots, \widetilde{f_n}]$. On peut donc
restreindre $X$ de sorte que les $\widetilde{f_i}$ soient entiers sur $\mathsf B$ et que $\mathsf B\subset \mathsf A[\widetilde{f_1},\ldots, \widetilde{f_n}]$.

\subsubsection{}
Nous allons maintenant établir un résultat général que nous appliquerons à plusieurs reprises
dans la suite. 
Soit $W\in \mathscr V$ et soit $\chi$ une fonction analytique sur
$Y_ W$. On pose $h=\chi|_{Y_x}$ et on suppose
que $\rho:=\norm h_{Y_x}$ est un élément de $\Gamma$.

\paragraph{}\label{p-controle-norme1}
Supposons que $\widetilde h$
appartient à $R$. Nous allons montrer qu'il existe un élément $V$ de $\mathscr V$ contenu dans $W$
et tel que 
$\norm\chi_{Y_V}=\rho$ (il suffit pour que ce soit le cas
que $\norm \chi_{Y_V}$ soit majorée par $\rho$, puisque
$\norm \chi_{Y_x}=\norm h_{Y_x}=\rho$).

Comme $\widetilde h$ appartient à $R$, elle est entière
sur $\mathsf B\cdot \hrt x^\circ$. Il existe donc $V\in \mathscr V$
tel que $V\subset W$ et tel que 
$\widetilde h$
soit entière sur $\mathsf B\cdot \mathsf A_V\subset \mathsf B_V$.
ll existe
par conséquent un entier $N>0$, un sous-ensemble
$I$ de $\{0,\ldots, N-1\}$ 
et une famille $(b_i)_{i\in I}$ d'éléments de $B_V$ possédant les propriétés suivantes : 

\begin{itemize}[label=$\diamond$] 
\item pour tout $i\in I$ on a $\norm {b_i}_{Y_V}=\norm{b_i}_{Y_x}
=\rho^{N-i}.$ 

\item $\left|h^N+\sum_{i\in I} (b_i|_{Y_x})h^i\right|<\rho^N$ sur tout $Y_x$. 
\end{itemize}
On a alors $\left|h^N+\sum_{i\in I} b_i \chi^i\right|<\rho^N$ au-dessus d'un voisinage de
$x$ dans $V$ (par propreté topologique de $Y\to X$). Quitte à restreindre $V$ autour de $x$, on peut donc
supposer que $\left|\chi^N+\sum_{i\in I} b_i \chi^i\right|<\rho^N$  sur tout $Y_V$. Ceci entraîne que
$\abs \chi\leq \rho$ en tout point de  $Y_V$ : en effet s'il existait $z\in Y_V$ tel que
$\abs{\chi(z)}>\rho$ on aurait pour tout $i\in I$ l'inégalité $\rho^N>\abs{b_i(z)}\rho^I$ car
$\abs{b_i}\leq \rho^{N-i}$ sur tout $Y_V$ ; mais ceci impliquerait que 
$\left|\chi^N(z)+\sum_{i\in I} b_i \chi^i(z)\right|=\rho^N$, ce qui est absurde.

\paragraph{}\label{p-controle-norme2}
On suppose toujours que $\widetilde h$
appartient à $R$. Soit $i$ un entier
compris entre $0$ et $h$ tel que l'ouvert $\mathrm D(\widetilde h)$
de
$\spec R$ rencontre $(\spec R)_{\tau_i}$. Soit $t$ un point maximal 
de $(\spec R)_{\tau_i}$ situé sur $\mathrm D(\widetilde h)$, et soit $d$ la dimension de $\overline{\{t\}}_{\tau_i}$. 
Le paragraphe
\ref{p-controle-norme1}
assure qu'il existe un domaine $V\in \mathscr V$
contenu dans $W$
tel que $\norm{\chi}_{Y_V}=\rho$. Nous 
allons montrer qu'on peut choisir un tel $V$ en imposant de surcroît
que pour tout $z\in V_i$ il existe un domaine
affinoïde $\Gamma$-strict de $Y_z$, non vide et purement de dimension $d$, sur lequel $\abs \chi$ prend la valeur $\rho$. 

Commençons par restreindre $X$ de sorte que $\chi$ soit
défini sur $Y$ tout entier (autrement dit, $W=X$) 
et que l'on ait $\norm \chi_Y=\rho$, ce
qui est possible d'après \ref{p-controle-norme1}.
Il résulte du corollaire \ref{coro-controle-dimfibres}
qu'il existe un voisinage ouvert de $t$ dans $\spec R$
purement de dimension relative $d$
sur $\spec \hrt x^\circ$ ; on peut toujours supposer
ce voisinage de la forme $\mathrm D(\alpha)$
pour un certain $\alpha$ non nul de $R$,
dont on note $\rho'$ le degré ; quitte à restreindre $X$
autour de $x$, on peut supposer que $\alpha=\widetilde{a|_{Y_x}}$
pour une certaine fonction analytique $a$ sur $Y$ telle que $\norm a_{Y_x}=\rho'$.

En vertu de \ref{p-controle-norme1}, on peut 
restreindre $X$ de sorte que $\norm a_Y=\rho'$. Soit $Y'$ le domaine
affinoïde de $Y$ défini par l'égalité $\abs a=\rho'$. La fibre $Y'_x$ a pour réduction $\mathrm D(\alpha)_{\tau_0}$, qui est purement de dimension $d$ par hypothèse. Il s'ensuit que $Y'_x$ est purement de dimension $d$ (lemme
\ref{lem-dim-xtilde}). Le morphisme $Y'\to X$ est plat, et ses fibres sont géométriquement réduites ; elles sont en particulier réduites, et satisfont \emph{a fortiori}
la propriété $(S_1)$. Il découle
alors du théorème
\ref{theo-dimrel-s1}
et de la propreté topologique de $Y'\to X$
que quitte à restreindre $X$ autour de $x$, on peut supposer que 
les fibres de $Y'\to X$ sont purement de dimension $d$. 

Nous allons montrer qu'il existe un élément $V$ de $\mathscr V$
tel que pour tout $z\in V_i$, la fonction $\abs {a\chi}$ prenne la valeur $\rho\rho'$ sur $Y_z$ ; compte-tenu du fait que $\norm \chi_Y=\rho$ et $\norm a_Y=\rho'$, cela signifie qu'il existe un point de $Y_z$ en lequel $\abs \chi=\rho$ et $\abs a=\rho'$, c'est-à-dire encore qu'il existe un point de $Y_z$ en lequel $\abs \chi=\rho$ et qui est situé sur le domaine affinoïde $\Gamma$-strict $Y'_z$ de $Y_z$, lequel est purement de dimension $d$. 

Pour alléger les notations, on peut remplacer $\chi$ par $a\chi$ et $\rho$ par $\rho' \rho$ (et bien sûr $h$ par $(a|_{Y_x})h$) ; il s'agit alors
maintenant simplement de montrer qu'il existe un élément $V$ de $\mathscr V$ tel que pour tout $z\in V_i$, la fonction $\abs \chi$ prenne la valeur $\rho$ sur $Y_z$, c'est-à-dire encore, compte-tenu du fait que $\norm \chi_Y=\rho$, tel que $\norm \chi_{Y_z}=\rho$  pour tout $z\in V_i$. Après ces substitutions, la fonction $\widetilde h$ est encore inversible en le point (maximal) $t$ de la fibre  $(\spec R)_{\tau_i}$. 

Notons $\abs \cdot^\star$
la valuation de $\hrt x$ d'annéloïde $\mathscr O_{\spec \hrt x^\circ, \tau_i}$.
Le morphisme
$\spec R\to \spec \hrt x^\circ$ est plat et donc générisant ; il existe par conséquent un point $t'$
de $(\spec R)_{\tau_0}$ qui se spécialise  en $t$. 
Soit $y$ un point de $Y_x$ dont l'image sur $\widetilde{Y_x}$ est égale à $t'$. 
Choisissons une valuation de $\kappa(t')$
dont l'annéloïde domine $\mathscr O_{\overline{\{t'\}},t}$, et prolongeons-la
arbitrairement en une valuation de $\hrt y$ ; celle-ci
 induit la valuation 
$\abs \cdot ^\star$ sur $\hrt x$
et on la note
encore
$\abs \cdot^\star$. 
Par construction, l'annéloïde de $\abs \cdot^\star$ dans $\hrt y$
contient l'image de $R$, 
et en particulier l'image de $\hrt x^\circ \cdot \mathsf B$, et
\emph{a fortiori}
celle
de $ \mathsf B$ ; il s'ensuit que
$\abs \cdot ^\star\in \widetilde{(Y,y)}$. Notons également que
comme $\widetilde h$ est inversible en $t$ on a $\abs{\widetilde h(t')}^\star=1$. 
Ceci entraîne que $\abs{h(y)}=\rho$ et que $\abs{\widetilde{h(y)}}^\star=1$. 

Soit $\Omega$ le domaine affinoïde $\Gamma$-strict
de $Y$ défini par l'égalité $\abs \chi=\rho$. Par
définition de $\Omega$
on a $\widetilde{(\Omega,y)}=
\widetilde{(Y,y)}\cap\P_{\hrt y/\widetilde k}\{\widetilde{h(y)}, \widetilde{h(y)}\inv\}$ ; 
en particulier, $\widetilde{(\Omega,y)}$
contient $\abs \cdot^\star$. 

Par platitude, l'image de $\Omega$ sur $X$ est un domaine analytique $\Gamma$-strict
et compact $U$ de $X$, qui contient $x$. Par fonctorialité, 
$\widetilde{(U,x)}$ contient la valuation 
$\abs \cdot^\star$ de $\hrt x$. L'annéloïde $\mathsf A_U$
est donc contenu dans $\mathscr O_{\spec \hrt x^\circ,\tau_i}$, et il existe
de
ce fait
un élément $V$ de $\mathscr V$
tel que 
$\mathsf A_U\subset \mathsf A_{V_i}$. Cette dernière inclusion 
assure que
$\widetilde{(V_i,x)}$ est contenu dans $\widetilde{(U,x)}$. 
Il existe donc un voisinage affinoïde $\Gamma$-strict $X'$ de $x$ dans $X$ tel que 
$X'\cap V_i\subset U$. Quitte à remplacer $V$ par $V\cap X'$ on peut
supposer que $V_i\subset U$ ; 
par définition de $U$ 
il existe
alors pour tout point $z$ de $V_i$ un point de $Y_z$ en lequel $\abs \chi=\rho$. Il s'ensuit
que $V$
satisfait aux conditions requises.

\subsubsection{}
En vertu de \ref{p-controle-norme1}
on peut restreindre $X$ de sorte que 
$\norm{\phi_i}_Y=r_i$ pour tout $i$
(les $\phi_i$ ont été définies
en \ref{par-def-phii}).

Notons $D=\mathscr M(A\{T_1/r_1,\ldots, T_n/r_n\})$
le $k$-polydisque fermé relatif
de polyrayon
$(r_1,\ldots, r_n)$ sur $X$. 
Soit $\pi \colon Y\to D$
le morphisme défini par le $n$-uplet $(\phi_1,\ldots, \phi_n)$. 
Par construction, la restriction $\pi|_{Y_x}\colon Y_x\to D_x$
est une immersion fermée. 

Soit $y\in Y_x$ ; posons $z=\pi(y)$. Soit $J$ l'ensemble
des indices $i$ tels que $\abs{T_i(z)}=r_i$ ; c'est aussi
l'ensemble des indices $i$ tels que $\abs{f_i(y)}=r_i$, c'est-à-dire
aussi celui des indices $i$ tel que $\widetilde{f_i}(\widetilde y)\neq
0$, en désignant par $\widetilde y$ 
l'image de $y$ sur $\widetilde{Y_x}$. 
Le corpoïde $\hrt z$ est une extension de $\hrt x$
et 
\[\widetilde{(D,z)}=
\P_{\hrt z/\widetilde k}\{\mathsf A\cup\{\widetilde {T_i(z)}\}_{i\in J}\}.\]
L'image réciproque
$\mathscr U$ 
de $\widetilde{(D_,z)}$ sur $\mathbf P_{\hrt y/\widetilde k}$ 
est
donc
égale à 
\[ \P_{\hrt y/\widetilde k}\{\mathsf A\cup\{\widetilde {f_i(y)}\}_{i\in J}\}.\]
Le sous -annéloïde de $\hrt y$ engendré par $\mathsf A$ et les $\widetilde {f_i(y)}$ pour $i\in J$ est celui engendré par $\mathsf A$ et les $\widetilde{f_i}(\widetilde y)$ pour $i$ variant entre $1$ et $n$, c'est-à-dire par l'image de  l'annéloïde $\mathsf A[\widetilde{f_1},\ldots, \widetilde{f_n}]$ de $\widetilde C$. 
Puisque $\mathsf A[\widetilde{f_1},\ldots, \widetilde{f_n}]$
contient $\mathsf B$ 
et est entier sur celui-ci,
l'ouvert $\mathscr U$ est égal à
$\P_{\hrt y/\widetilde k}\{\mathsf B\}$,
c'est-à-dire à
$\widetilde{(Y,y)}$ ; en conséquence, $\pi$
est sans bord en $y$. Ceci valant pour tout $y\in Y_x$ il en résulte par
propreté topologique qu'il existe un ouvert $X'$ de $X$ contenant $x$
tel que $Y_{X'}\to D_{X'}$ soit sans bord.
Quitte à remplacer $X$ par un voisinage
affinoïde $\Gamma$-strict de $x$ dans $X'$, 
on peut donc supposer
que $\pi$ est sans bord. Puisque $Y$ et $D$ sont affinoïdes, 
cela entraîne que $\pi$ est fini.

Comme $\pi|_{Y_x}$ est une immersion fermée, 
la restriction du morphisme canonique
$\mathscr O_D\to \pi_*\mathscr O_Y$ à
la fibre $D_x$ est surjective. On en déduit
que $\mathscr O_D\to \pi_*\mathscr O_Y$ est surjectif au voisinage 
de $D_x$ dans $D$ (c'est une conséquence du lemme de Nakayama, 
voir \cite{ducros2018}, 2.5.4). 
Par propreté
topologique on peut encore restreindre $X$ autour de $x$ de sorte que
$\mathscr O_D\to \pi_*\mathscr O_Y$ soit surjective, ce qui 
signifie que $\pi$ est une immersion fermée. 

\subsubsection{}
Soit $I$ le noyau du morphisme surjectif $A\{T/r\}\to B$. 
Remarquons que $I\hr x\{T/r\}$ est
le noyau de la surjection distinguée
$v\colon \hr x\{T/r\}\to C, T_i\mapsto f_i$ ; nous noterons
$\widetilde v \colon \hrt x [r\backslash T]\to \widetilde C$ la
surjection induite. Comme les $\widetilde {f_i}$ engendrent
$R$
comme $\hrt x^\circ$-algèbre, 
le morphisme $\widetilde v$ induit une flèche
surjective $\hrt x^\circ[r\backslash T]\to R$. 

Nous allons montrer qu'on peut restreindre $X$
de sorte qu'il existe une famille finie $(a_1,\ldots , a_m)$
satisfaisant les propriétés suivantes, en posant $h_i=a_i|_{Y_x}$
quel que soit $i$ : 
\begin{itemize}[label=$\diamond$] 
\item les $h_i$ forment une famille fortement génératrice de 
$\mathrm{Ker}(v)$ et pour tout $i$ la
norme spectrale $\rho_i:=\norm{h_i}_{Y_x}$ appartient à $\Gamma$ ; 
\item les $\widetilde {h_i}$ appartiennent à $\hrt x^\circ[r\backslash T]$, sont de norme $1$
pour la norme de Gauß $\sum \alpha_I T^I\mapsto \max \abs{\alpha_I}$,
et engendrent en tant qu'idéal de $\hrt x^\circ[r\backslash T]$
le noyau de la surjection
$\hrt x^\circ[r\backslash T]\to R$ induite
par $\widetilde v$ ;
\item pour tout $z\in X$ et tout $i$ on a $\norm{h_i}_{Y_z}=\rho_i$.
\end{itemize}

\paragraph{}
On commence par choisir
une famille $(h_1,\ldots, h_m)$ fortement
génératrice de $\mathrm{Ker}(v)$ (proposition \ref{prop-fort-gen}).
Comme cette propriété est stable par petites perturbations d'après \ref{def-fortgen})
on peut supposer quitte à restreindre $X$ que
$h_i$ est pour tout $i$ la restriction d'une fonction analytique $a_i$ sur $D$ appartenant à $I$. 

\paragraph{}
Complétons $(a_1,\ldots, a_m)$ en une famille
$(a_1,\ldots, a_m, b_1,\ldots, b_p)$ qui engendre $I$. 

Soit $j$ un indice tel que $b_j|_{Y_x}=0$ et soit $Y'$ le sous-espace
analytique fermé de $D$ défini par l'idéal
$((a_i)_i, (b_i)_{i\neq j})$. Considérons par abus $\mathscr O_Y$ et $\mathscr O_{Y'}$ comme
des faisceaux cohérents sur $D$. On a une surjection naturelle $\mathscr O_{Y'}\to \mathscr O_Y$, dont la restriction à $D_x$ est
bijective par choix de $j$. Or $Y$ est plat sur $X$, ce qui signifie que $\mathscr O_Y$ est plat sur $X$. 
Il en résulte que $\mathscr O_{Y'}\to \mathscr O_Y$ est un isomorphisme au voisinage de $D_x$ (\cite{ducros2018}, 
lemme 4.5.10). Par propreté topologique on peut donc restreindre $X$
autour de $x$ de sorte que $I$ soit engendré par 
les $a_i$ pour $i$ compris entre $1$ et $m$ et par les $b_i$ pour $i$ compris entre $1$ et $p$ et différent de $j$.

En réitérant l'opération autant de fois
qu'il est nécessaire et en renumérotant les $b_i$ on voit qu'on peut finalement restreindre $X$
autour de $x$ de sorte $I$ soit engendré
par une famille $(a_1,\ldots, a_m,b_1,\ldots,b_p)$ avec $b_i|_{Y_x}\neq 0$ pour tout $i$.
La famille $(h_1,\ldots, h_m, b_1|_{Y_x}, \ldots, b_p|_{Y_x})$ est
alors encore une famille
fortement génératrice de $\mathrm{Ker}(v)$. En incorporant $(b_1,\ldots, b_p)$ à la famille des $a_i$ on se ramène au cas où 
cette dernière engendre $I$. 

\paragraph{}
La norme
spectrale de $\hr x\{T/r\}$ étant à valeurs dans 
$\abs{\hr x\gpm}\cdot \Gamma$, il existe  des
élément $\lambda_i\in \mathscr O_{X,x}\gpm$
tels que $\abs{\lambda_i}\cdot \norm {h_i}_{Y_x}\in \Gamma$
pour tout $i$ et tel que l'élément $\widetilde{\lambda_i(x)h_i}$ de $\hrt x[r\backslash T]$
soit de norme (de Gauß) égale
à $1$. En
restreignant $X$
autour de $x$ pour que
chaque $\lambda_i$ soit définie
et inversible sur $X$, et en multipliant $a_i$
par $\lambda_i$ pour tout $i$, on se ramène au cas où
$\rho_i:=\norm{h_i}_{Y_x}$ appartient à $\Gamma$ pour tout $i$
et où chacun des $\widetilde{h_i}$ est de norme $1$. 

\paragraph{}Pour tout $i$
l'élément $\widetilde{h_i}$ de $\hrt x [r\backslash T]$ appartient
à $\hrt x^\circ[r\backslash T]$ (car il est de norme $1$) et au noyau de $\widetilde v$. 
Complétons $(\widetilde{h_1},\ldots, \widetilde{h_m})$ en une famille génératrice
$(\widetilde{h_1},\ldots, \widetilde{h_m},\alpha_1,\ldots, \alpha_p)$ de 
$\mathrm {Ker}(\widetilde v|_{\hrt x^\circ[r\backslash T]})$ sur $\hrt x^\circ[r\backslash T]$,
avec les $\alpha_j$ tous non nuls. Pour tout $j$, choisissons $\lambda_j\in \hrt x\gpm$ tel que
$\lambda_j\alpha_j$ soit de norme $1$. L'élément $\lambda_j\alpha_j$ appartient encore à 
$\mathrm{Ker}(\widetilde v)$, et $\abs {\lambda_j}\geq 1$. Puisque 
$\alpha_j=\lambda_j\inv(\lambda_j \alpha_j)$, la famille
$(\widetilde{h_1},\ldots, \widetilde{h_m},\lambda_1\alpha_1,\ldots, \lambda_p\alpha_p)$
engendre encore  $\mathrm {Ker}(\widetilde v)\cap \hrt x^\circ[r\backslash T]$ ; par conséquent
on peut, quitte à remplacer $\alpha_j$ par $\lambda_j\alpha_j$ pour tout $j$, supposer
que les $\alpha_j$ sont tous de norme $1$. 

Chacun des $\alpha_j$ appartient à $\mathrm{Ker}(\widetilde v)$. Or comme $v$ est distingué et
comme $(h_1,\ldots,h_m)$ est une famille fortement génératrice de $\mathrm{Ker}(v)$, 
le noyau de $\widetilde v$ est engendré par les $\widetilde{h_i}$ d'après la proposition \ref{distingue-reduit}. 
Chaque $\alpha_j$ appartient donc à $(\widetilde{h_1},\ldots,\widetilde {h_m})$ et peut dès lors
se relever en un élément $\alpha'_j$ appartenant à l'idéal $(h_1,\ldots,h_m)$ de l'anneau $\kappa(x)[T]\subset \hr x \{T/r\}$.
Quitte à restreindre $X$
autour de $x$ on peut supposer que chacun des $\alpha'_j$ est la restriction d'un élément $\alpha''_j$ appartenant 
à l'idéal $(a_1,\ldots, a_m)$ de l'anneau $A[T]\subset A\{T/r\}$, et en particulier à $I$. 
En incorporant les $\alpha''_j$ à la famille des $a_i$, on se ramène au cas où
$(\widetilde{h_1},\ldots, \widetilde{h_m})$ est une famille génératrice
de $\mathrm {Ker}(\widetilde v|_{\hrt x^\circ[r\backslash T]})$ sur $\hrt x^\circ[r\backslash T]$.

\paragraph{}\label{p-norme-equations}
Pour tout $i$, on écrit  $a_i=\sum a_{i,J}T^J$, et on pose $h_{i,J}=a_{i,J}|_{Y_x}$ pour tout $(i,J)$. 

À chaque indice $i$ on associe les sous-ensembles suivants de $\N^n$ : 
\begin{itemize}[label=$\diamond$] 
\item $\mathscr J_0(i)=\{J, \abs{h_{i,J}}\cdot r^J<\rho_i\}.$
\item $\mathscr J_1(i)=\{J, \abs{h_{i,J}}\cdot r^J=\rho_i\;\;\text{et}\;\;\abs{\widetilde{h_{i,J}}}<1\}.$
\item $\mathscr J_2(i)=\{J, \abs{h_{i,J}}\cdot r^J=\rho_i\;\;\text{et}\;\;\abs{\widetilde{h_{i,J}}}=1\}.$
\end{itemize}
(Nos hypothèses entraînent que $\N^n=\mathscr J_0(i)\cup \mathscr J_1(i)\cup \mathscr J_2(i)$, 
que $\mathscr J_1(i)$ et $\mathscr J_2(i)$ sont finis, et que
$\mathscr J_2(i)\neq \emptyset$.) On pose $b_i=\sum_{J\in \mathscr J_0(i)} a_{i,J}T^J$. 

Par construction, $\norm{b_i}_{Y_x}<\rho_i$ pour tout $i$. Par propreté topologique on peut donc supposer quitte
à restreindre $X$ autour
de $x$ que chaque $b_i$ est de norme spectrale 
strictement inférieure à $\rho_i$.

Soit $X'$ le domaine affinoïde $\Gamma$-strict de $X$ défini par les conditions 
\[\{\abs{a_{i,J}}\cdot r^J\leq \rho_i\}_{1\leq i\leq m, J\in \mathscr J_1(i)}\;\;\text{et}\;\;
\{\abs{a_{i,J}}\cdot r^J= \rho_i\}_{1\leq i\leq m, J\in \mathscr J_2(i)}.\]

Le domaine $X'$ contient $x$, et 
\[\widetilde{(X',x)}=
\widetilde{(X,x)}\cap \P_{\hrt x/\widetilde k}\{\{\widetilde{h_{i,J}}\}_{J\in \mathscr J_1(i)}\cup\{\widetilde{h_{i,J}},\widetilde{h_{i,J}}\inv
\}_{J\in \mathscr J_2(i)}\}.\]
En particulier, 
$\widetilde{(X',x)}$ contient $\abs \cdot$. On peut donc
restreindre $X$ de sorte que $X'=X$. 
Il résulte alors de nos constructions que pour tout $z\in V$ et tout $i$ compris entre $1$
et $m$ la norme spectrale $\max \abs{a_{i,J}(z)}\cdot r^J$ de $a_i|_{Y_z}$ est
égale à $\rho_i$.

(On aurait pu également invoquer \ref{p-controle-norme1} et \ref{p-controle-norme2}
dans le cas où $Y=D$, où $f_i=T_i$, où $R=\hrt x^\circ[r\backslash T]$ et où $t$ est le point générique de la fibre spéciale -- une fonction appartenant à $\hrt x^\circ[r\backslash T]$ est inversible en ce point précisément si sa norme de Gauß vaut 1 ; mais la preuve ci-dessus est plus élémentaire.)

\subsubsection{}
Pour tout $i$ compris entre $0$ et $\mathsf h$ et tout point maximal 
$t$ de $(\spec R)_{\tau_i}$, il existe $G_t\in \hr x\{T/r\}$, dont on note $g_t$ l'image dans $C$, telle que les propriétés suivantes soient satisfaites : 

\begin{itemize}[label=$\diamond$]
\item les réels $\norm{G_t}_{D_x}$ et $\norm{g_t}_{Y_x}$ sont égaux
à un même élément $r'_t$ de $\Gamma$ ; 
\item la réduction $\widetilde{G_t}$ appartient à 
$\hrt x^\circ[r\backslash T]$ (ce qui entraîne que $\widetilde {g_t}$ appartient à $R$) ;
\item l'élément $\widetilde{g_t}$ de $R$ est inversible en $t$ et s'annule
en les autres points maximaux de $(\spec R)_{\tau_i}$.
\end{itemize}
(Pour construire $g_t$, on commence par choisir un élément de $R$ satisfaisant la dernière propriété, ce qui est possible car
la réunion des adhérences dans $\spec R$ des points maximaux de $(\spec R)_\tau$
autres que $t$ est un fermé de Zariski de $\spec R$ qui ne contient pas $t$ ; on elève ensuite
cet élément dans $\hrt x^\circ[r\backslash T]$ puis dans $\hr x\{T/r\}$). 

Les propriétés de $G_t$ étant stables par petites perturbations, 
on peut supposer quitte à restreindre $X$
autour de $x$ que $G_t$ est la restriction d'une fonction $\Psi_t$
appartenant à $A\{T/r\}$ ; on note $\psi_t$ l'image de $\Psi_t$ dans $B$. 

Pour tout $(i,t)$, les valuations des coefficients
de $\widetilde{G_t}$
sont toutes inférieures ou égales à $1$. Un raisonnement analogue
à celui tenu au \ref{p-norme-equations}
montre alors qu'on peut supposer quitte à restreindre $X$
que pour tout $(i,t)$ on a $\abs{\Psi_t}\leq r'_t$ sur $D$ tout entier
(ce qui entraîne que $\abs{\psi_t}\leq r'_t$ sur $Y$ tout entier). 

Par ailleurs il résulte de \ref{p-controle-norme2}
qu'on peut, quitte à encore restreindre $X$, supposer que 
pour tout $(i,t)$ et tout $z\in X_i$, la fonction 
$\abs{\psi_t}$ atteint la valeur $r'_t$ sur un domaine affinoïde $\Gamma$-strict
de $Y_z$ purement de dimension $\dim\overline{\{t\}}_{\tau_i}$.

\subsubsection{}
Puisque $(\widetilde {h_1},\ldots, \widetilde{h_m})$ engendre 
$\mathrm{Ker}(\widetilde v|_{\hrt x^\circ[r\backslash T]})$,
on a un isomorphisme de $\hrt x^\circ$-algèbres
\[\hrt x^\circ[r\backslash T]/(\widetilde {h_1},\ldots, \widetilde {h_m})
\simeq R.\]
Comme chaque $h_i$ se prolonge en la fonction 
$a_i$ sur $D$ et comme $\norm{a_i}_D=\rho_i$, les $\widetilde{h_i}$
appartiennent tous à $\mathsf A[r\backslash T]$ ; pour une raison, 
analogue, les $\widetilde{G_t}$ appartiennent
tous à $\mathsf A[r\backslash T]$. 
Pour tout $V\in \mathscr V$, posons $\mathsf M_V
=\mathsf A_V[r\backslash T]/(\widetilde{h_1},\ldots, \widetilde{h_m})$ ; on écrit $\mathsf M$ au lieu de $\mathsf M_X$.

Nous allons maintenant montrer qu'on peut restreindre $X$ de sorte que  le morphisme
$\spec \mathsf  M
\to \spec \mathsf A$ 
satisfasse les propriétés suivantes : 

\begin{enumerate}[a]
\item il est plat et à fibres géométriquement réduites ; 
\item les composantes irréductibles de ses fibres sont géométriquement irréductibles ; 
\item pour tout $i$ compris entre $0$
et $\mathsf h$, pour tout point $\sigma$ de $\spec \mathsf A$
tel que $\nu_i(\sigma)\neq 0$ et $\nu_{i+1}(\sigma)=0$, 
et pour tout point maximal $s$ de $(\spec \mathsf M)_\sigma$, il existe un point maximal $t$ de $(\spec R)_{\tau_i}$ tel que l'image $\check G_t$ de $\widetilde{G_t}$ dans
$\mathsf M$ soit inversible en $s$
et s'annule en les autres points maximaux de 
$(\spec \mathsf M)_\sigma$, et tel 
que $\dim\overline{\{s\}}_\sigma$
soit égal à $\dim \overline{\{t\}}_{\tau_i}$ ; 
\item il existe un recouvrement fini $(\Omega_j)$ de $\spec \mathsf M$ tel que pour tout $j$, chaque fibre
non vide de
$\Omega_j\to \spec \mathsf A$ soit une composante connexe de la fibre correspondante de
$\spec \mathsf M\to \spec \mathsf A$. 
\end{enumerate}

\paragraph{}
Comme $\spec R$ est plat sur $\spec \hrt x^\circ$, 
$\spec \mathsf M_V\to \spec \mathsf A_V$ est plat pour tout élément $V$ de $\mathscr V$ suffisamment petit : la
méthode \ref{ss-methode-schemas} permet en effet
de se ramener à l'assertion analogue dans le cas non gradué,
qui est le théorème 11.2.6 (ii)
de \cite{ega43}. 

\paragraph{}
Pour tout élément $V$  de $\mathscr V$, notons $E_V$ l'ensemble
des points de $\spec \mathsf A_V$ tels que la fibre correspondante de $\spec \mathsf M_V\to \spec \mathsf A_V$ soit géométriquement réduite.

Pour tout $V$, l'ensemble $E_V$ est une partie constructible de $\spec \mathsf A_V$ (on utilise  \ref{ss-methode-schemas} pour se ramener au cas non gradué, et c'est alors le théorème 9.7.7 (iii) de \cite{ega43}) ; 
elle
dépend de manière cartésienne de $V$. 

Comme les fibres de $\spec R\to \spec \hrt x^\circ$ sont géométriquement réduites, la limite projective des $\spec \mathsf M_V\setminus E_V$ est vide ; par compacité constructible, il existe $V$ tel que $E_V=\spec \mathsf M_V$. 

\paragraph{}
Soit $i$ compris entre $0$ et $\mathsf h$ et soit $t$ un point maximal 
de $(\spec R)_{\tau_i}$. Pour tout élément $V$  de $\mathscr V$  notons $E^1_V$ l'ensemble des points $\sigma$ de $\spec \mathsf A_V$ tels
que $\nu_i(\sigma)\neq 0$, tels que $\nu_{i+1}(\sigma)=0$
et tels que la fibre en $\sigma$ de l'ouvert $\mathrm D(\check G_t)$ de $\spec \mathsf M_V$ soit géométriquement irréductible
et de dimension $\dim \overline{\{t\}}_{\tau_i}$. C'est une partie constructible de 
$\spec \mathsf A_V$ (on utilise  \ref{ss-methode-schemas} pour se ramener au cas non gradué, et on combine
alors la proposition 9.2.6 et
le théorème 9.7.7 (i) de \cite{ega43}) ; elle dépend de manière cartésienne de $V$. 
Comme la partie correspondante de $\spec \hrt x^\circ$ est égale à $\{\tau_i\}$, c'est-à-dire précisément
à l'ensemble des points de $\spec \hrt  x^\circ$ en lesquels $\nu_i$ est inversible et $\nu_{i+1}$ nul, un argument de compacité constructible
permet là encore de conclure que $E^1_V$ est pour $V$  assez petit 
l'ensemble des points $\sigma$ de $\spec \mathsf A_V$ tels que 
$\nu_i(\sigma)\neq 0$ et $\nu_{i+1}(\sigma)=0$. 

\paragraph{}
Soit $i$ compris entre $0$ et $\mathsf h$. Pour tout élément $V$  de $\mathscr V$
notons $E^2_V$ 
l'ensemble des points $\sigma$ de $\spec \mathsf A_V$
tels que $\nu_i(\sigma)\neq 0$, tels que $\nu_{i+1}(\sigma)=0$
et tels que la réunion des ouverts $\mathrm D(\check G_t)$ de $\spec \mathsf M_V$ pour $t$ parcourant l'ensemble des points
maximaux de $(\spec R)_{\tau_i}$ contienne 
une partie dense de $(\spec \mathsf M_V)_\sigma$.
C'est une partie constructible de 
$\spec \mathsf M_V$ (on utilise  \ref{ss-methode-schemas} pour se ramener au cas non gradué, et c'est alors la proposition 9.5.3 de \cite{ega43}) ; elle dépend de manière cartésienne de $V$. 
Comme la partie correspondante de $\spec \hrt x^\circ$ est égale à $\{\tau_i\}$, c'est-à-dire précisément
à l'ensemble des points de $\spec \hrt  x^\circ$ en lesquels $\nu_i$ est inversible et $\nu_{i+1}$ nul, un argument de compacité constructible
permet là encore de conclure que $E^2_V$ est pour $V$  assez petit 
l'ensemble des points $\sigma$ de $\spec \mathsf A_V$ tels que 
$\nu_i(\sigma)\neq 0$ et $\nu_{i+1}(\sigma)=0$. 

\paragraph{}
En vertu de la proposition \ref{prop-schema-decoupage} il existe un recouvrement ouvert fini 
$(\Upsilon_j)_j$ de $\spec R$ tel que pour tout $j$, chaque fibre
non vide de $\Upsilon_j\to \spec \hrt x^\circ$
soit une composante connexe de la fibre correspondante de $\spec R\to \spec \hrt x^\circ$. 
Chaque $\Upsilon_j$ est quasi-compact ; si on prend $V$ suffisamment petit, on peut donc
faire en sorte que $\Upsilon_j$ provienne pour tout $j$ d'un ouvert $\Omega_{j,V}$ de $\spec \mathsf M_V$,
dépendant de manière cartésienne de $V$. 

Les $\Upsilon_j$ recouvrant $\spec R$, un argument de compacité constructible
montre que si $V$ est assez petit les $\Omega_{j,V}$ recouvrent $\spec \mathsf M_V$. 

Pour tout $j$ et tout $V$ assez petit pour que cela ait un sens, notons $E^3_{j,V}$ l'image de $\Omega_{j,V}$ sur $\spec \mathsf A_V$, et $E^4_{j,V}$ le sous-ensemble
de $E^3_{j,V}$ formé des points en lesquels la fibre de $\Omega_{j,V}\to \spec \mathsf A_V$ est géométriquement connexe. 
Les ensembles $E^3_{j,V}$ et $E^4_{j,V}$ sont constructibles (pour le second, cela découle du corollaire 
9.7.9 de \cite{ega43}) ; ils dépendent de manière cartésienne
de $V$, et les deux sous-ensembles correspondants de $\spec \hrt x^\circ$ coïncident par choix
des $\Upsilon_j$, et parce que les composantes irréductibles des fibres de $\spec R\to \spec \hrt x^\circ$
sont géométriquement irréductibles. Il s'ensuit par un argument de compacité constructible qu'on 
a pour $V$
suffisamment petit l'égalité $E^3_{j,V}=E^4_{j,V}$ pour tout $j$, ce qui signifie que les fibres de
$\Omega_{j,V}\to 
\spec \mathsf A_V$ sont géométriquement connexes. 

Enfin,  soient $j$ et $\ell$ deux indices distincts.
Pour tout $V$ suffisamment petit pour que cela ait un sens, notons $E^5_{j,\ell, V}$
et $E^6_{j,\ell,V}$ les images respectives de
$\Omega_{j,V}\cap \Omega_{\ell,V}$ et
$\Omega_{j,V}\triangle \Omega_{\ell,V}$
sur $\spec \mathsf A_V$. Ce sont des parties constructibles de 
$\spec \mathsf A_V$, dépendant de lanière cartésienne de $V$ ; par construction, les parties correspondantes de $\spec \hrt x^\circ$ 
sont disjointes. En utilisant une dernière fois un argument de compacité constructible, 
on en déduit que si $V$ est assez petit alors
$E^5_{j,\ell, V}\cap E^6_{j,\ell,V}=\emptyset$, ce qui signifie que pour tout $\sigma\in \spec \mathsf A_V$, les fibres de $\Omega_{j,V}$ et $\Omega_{\ell,V}$
en $\sigma$ sont ou bien disjointes, ou bien égales.

\paragraph{}
Il résulte des paragraphes précédents qu'on peut restreindre $X$ de sorte que $\spec \mathsf M\to \spec \mathsf A$ satisfasse
les propriétés (a), (b), (c) et (d) requises. 

\subsubsection{Fin de la démonstration}
Soit $\widetilde x$ l'image de $x$ sur $\widetilde X$ et soit $\Omega$
l'image réciproque de $\overline{\{\widetilde x\}}$ sur $X$. C'est un voisinage ouvert
de $x$ dans $X$. Soit $V$ un voisinage affinoïde $\Gamma$-strict de $x$ dans $\Omega$.
Le morphisme $\widetilde A\to \widetilde{A_V}$ induit par l'inclusion $V\hookrightarrow X$
se factorise alors par l'image de $\widetilde A$ dans $\kappa(\widetilde x)$, qui s'identifie elle-même
à $\mathsf A$. Modulo cette identification, l'image de $a_i|_{D_V}$ dans $\widetilde{A_V}[r\backslash T]$ est égale à $\widetilde{h_i}$ pour tout $i$, celle de $\Psi_t|_{D_V}$ 
est égale à $\widetilde{G_t}$ pour tout $t$, et celle de $\theta_i$ est égale à $\nu_i$ pour tout $i$. 

En remplaçant $X$ par $V$ on se ramène alors
au cas où
le
morphisme
\[q\colon \spec \widetilde A[r\backslash T]/(\widetilde{a_1},\ldots,\widetilde{a_m})\to \widetilde X\]
satisfait les propriétés suivantes : 
\begin{enumerate}[a]
\item il est plat et à fibres géométriquement réduites ; 
\item les composantes irréductibles de ses fibres sont géométriquement irréductibles ; 
\item pour tout $i$ compris entre $0$
et $h$, pour tout point $\sigma$ de $\spec \widetilde A$
appartenant à $\mathrm D(\widetilde{\theta_i})\setminus 
\mathrm D(\widetilde{\theta_{i+1}})$ 
et pour tout point maximal $s$ de $q\inv(\sigma)$, il existe un point maximal $t$ de $(\spec R)_{\tau_i}$ tel que l'image $\check \Psi_t$ de $\widetilde{\Psi_t}$ dans
$ \widetilde A[r\backslash T]/(\widetilde{a_1},\ldots,\widetilde{a_m})$ soit inversible en $s$
et s'annule en les autres points maximaux de 
$q\inv(\sigma)$, et telle que $\dim \overline{\{s\}}_\sigma=\dim\overline {\{t\}}_{\tau_i}$ ; 
\item il existe un recouvrement ouvert
fini $(\Omega_j)$ de 
$\spec \widetilde A[r\backslash T]/(\widetilde{a_1},\ldots,\widetilde{a_m})$
tel que pour tout $j$ et tout $\xi \in \widetilde X$, l'intersection $\Omega_j\cap q\inv(\xi)$ soit
ou bien vide, ou bien une composante connexe de $q\inv(\xi)$. 
\end{enumerate}
Soit $z$ un point de $X$.
Soit $i$ un entier compris entre $0$ et $\mathsf h$ tel que $z\in X_i$, et minimal
pour cette propriété (rappelons que comme $\theta_{\mathsf h}=1$, le point $z$ appartient automatiquement à $X_{\mathsf h}$).
L'image de $z$ sur $\spec \widetilde A$
appartient alors à $\mathrm D(\widetilde{\theta_i})\setminus \mathrm D(\widetilde{\theta_{i+1}})$. 
Par ce qui précède,
$\spec \hrt z[r\backslash T]/(\widetilde {a_1|_{Y_z}}, \ldots, \widetilde{a_m|_{Y_z}})$
est géométriquement réduit, et ses composantes irréductibles sont géométriquement intègres. 

Par ailleurs, soit $s$ un point maximal de  $\spec \hrt z[r\backslash t]/(\widetilde {a_1|_{Y_z}}, \ldots, \widetilde{a_m|_{Y_z}})$. Il existe alors un point maximal $t$ de $(\spec R)_{\tau_i}$ tel que
$\check \Psi_t$ soit inversible en $s$ et s'annule en tous les autres points maximaux de $\spec \hrt z[r\backslash t]/(\widetilde {a_1|_{Y_z}}, \ldots, \widetilde{a_m|_{Y_z}})$,
et 
tel que $\dim \overline{\{s\}}=\dim \overline{\{t\}}_{\tau_i}$.
L'image de $\Psi_t$ dans $B$ est égale à $\psi_t$. 
On a $\norm{\Psi_t}_{D_z}=r'_t$ ; et comme $z$ appartient à $X_i$, il existe un domaine affinoïde 
$\Gamma$-strict non vide de $Y_z$, purement de dimension $\dim \overline{\{t\}}_{\tau_i}$,
sur lequel $\abs{\psi_t}$ prend la valeur $r'_t$. 
On déduit alors de la proposition \ref{prop-univdistingue-1} et du critère \ref{ss-critere-utile},
qui garantit que les hypothèses de cette dernière sont satisfaites ici,
que $B\simeq A\{T/r\}/(a_1,\ldots, a_m)$ est une présentation $\Gamma$-sympathique de $B$ sur $A$.

\end{proof}

\appendix
\section{Autour de la dimension des espaces analytiques}

Le but de cet appendice est de démontrer différents résultats auxiliaires en théorie de la dimension, qui sont utilisés dans l'article et peuvent avoir par ailleurs leur intérêt propre. 

Le premier de ces énoncés est certainement bien connu ; nous en donnons une démonstration pour la commodité du lecteur. 

\begin{lemm}\label{lem-dim-xtilde}
Soit $X$ un espace $k$-affinoïde, soit $(X_i)$ la famille
de ses composantes irréductibles et pour tout $i$, soit $d_i$
la dimension de $X_i$. Soit $\rho \colon X\to \widetilde X$
l'application de réduction. 

\begin{enumerate}[1]
\item Pour tout $i$, l'image $\rho(X_i)$ est un fermé de $\widetilde X$. 

\item Désignons pour tout $i$ par $(\mathsf X_{ij})_j$ la famille
des composantes irréductibles de $\rho(X_i)$. On a alors
$\dim \mathsf X_{ij}=d_i$ quel que soit $i$, et $(\mathsf X_{ij})_{i,j}$ est
la famille des composantes irréductibles de $\widetilde X$.
\end{enumerate}
\end{lemm}

\begin{proof}
Fixons $i$ et munissons $X_i$ (par exemple) de sa structure réduite. L'immersion fermée $X_i\hookrightarrow X$
étant en particulier un morphisme fini, 
le morphisme induit $\widetilde {X_i}\to \widetilde X$ est fini
(\cite{temkin2004}, Prop. 3.1. (iii)), et par conséquent fermé ; or par surjectivité de $X_i\to \widetilde {X_i}$ l'image de $\widetilde {X_i}\to \widetilde X$ coincide avec $\rho(X_i)$, ce qui achève de démonter (1).

Prouvons maintenant (2). Fixons $i$. Pour montrer que $\rho(X_i)$ est purement de dimension $d_i$, il suffit de démontrer que c'est le cas de $\widetilde X_i$. Soit $\alpha$ un élément non nul de $\widetilde{\mathscr O_{X_i}(X_i)}$ et soit $r$ son degré ; soit $a$ une fonction analytique sur $X_i$ de rayon spectral $r$ et telle que $\widetilde a=\alpha$. L'image réciproque $V$ de $D(\alpha)$ sur $X_i$
est le domaine affinoïde non vide de $X_i$ défini par la condition
$\abs a=r$ et l'on a $\widetilde V=D(a)$ (\cite{temkin2004}, 
Prop. 3.1 (ii)). Puisque $V$ est un domaine affinoïde non vide de $X_i$, il
est de dimension $d_i$. On sait alors que
$\widetilde V=D(\alpha)$ est de dimension $d_i$
(utiliser la proposition 3.1 (i) de \cite{temkin2004} pour se réduire au cas strict, et appliquer la remarque de \cite{bosch-g-r1984}
suivant le corollaire 3 de \S 6.3.4). Ceci valant pour tout élément $\alpha$ non nul de $\widetilde{\mathscr O_{X_i}(X_i)}$, le schémoïde $X_i$ est
purement de dimension $d_i$. 

Vérifions enfin que les $\mathsf X_{ij}$ sont bien les composantes irréductibles de $\widetilde X$. 
Les points génériques des composantes irréductibles de $\widetilde X$ sont exactement les images des points du bord de Shilov de $X$ ; et pour tout $i$, les points génériques des composantes irréductibles de  $\widetilde {X_i}$ sont exactement les images des points du bord de Shilov de $X_i$. Puisque $\rho(X_i)$ est pour tout $i$ l'image du morphisme fini $\widetilde X_i\to \widetilde X$, il suffit pour conclure de montrer que le bord de Shilov de $X$ est la réunion des bords des Shilov des $X_i$. On peut évidemment pour ce faire supposer $X$ non vide. 

Soit $S$ le bord de Shilov de
$X$. Pour tout $i$, notons $S_i$ le bord de Shilov de $X_i$. Soit $f$ une fonction analytique sur $X$. Puisque $f|_{S_i}$ atteint pour tout $i$ son maximum en valeur absolue
sur $S_i$, la fonction $f$ atteint son maximum en valeur absolue sur $\bigcup S_i$ (on se sert ici du fait que $X$ est non vide) ; par conséquent, $S\subset \bigcup S_i$. 
Réciproquement,
soit $i$ un indice et soit $x$ un point de $S_i$.
Par définition il existe une fonction $g$ sur $X_i$
telle que $\abs {g(x)}=r>0$ et $\abs{g(y)}<r$ pour tout $y\in S_i\setminus\{g\}$.
On peut relever $g$ en une fonction sur $X$, que l'on note encore $g$. Soit $h$ une fonction sur $X$ dont la restriction à $X_i$
est génériquement inversible, et qui s'annule en tous les points de toutes les autres composantes de $X$.
Comme $\widetilde {X_i}$ est purement de dimension $d_i$
(on l'a vu plus haut), l'image de $x$ sur $\widetilde {X_i}$
est le point générique d'un fermé de dimension $d_i$, et le corpoïde $\widetilde {\hr x}$
est donc de degré de transcendance $d_i$ sur $\widetilde k$,
ce qui interdit à $x$ de vivre sur un fermé de Zariski strict de l'espace irréductible $X_i$.
Il s'ensuit que $h(x)\neq 0$, puis qu'il existe un entier $N\geq 1$
tel que $\abs{g(x)^Nh(x)}>\abs{g(y)^Nh(y)}$ pour tout $y\in S_i\setminus \{x\}$. Mais alors 
$x$ est l'unique point de $\bigcup_j S_j$ en lequel $\abs{g^Nh}$ atteint son maximum (en effet, $g^Nh$ s'annule
par construction
en tout point de $S_j$ si $j\neq i$) ; par
conséquent, $x\in S$.
\end{proof}

Le second de ces énoncés est la déclinaison analytique d'une version un peu affaiblie
de \cite{ega43}, th. 12.1.1 (iv). Nous utiliserons librement les propriétés $(S_n)$ de Serre (pour $n=1$ ou $2$) et leurs variantes analytiques, 
leur invariance par extension des scalaires, et les prinicipes GAGA à leur sujet 
(\emph{cf.} \cite{ducros2018}, Lemma-Definition 2.4.3, Lemma 2.4.6 et Prop. 2.6.7) ; nous utiliserons également les propriétés de base des composantes irréductibles
en géométrie analytique, rappelées à la section 1.5 de \emph{op. cit.}

Si $\phi \colon Y\to X$ est un morphisme d'espaces analytiques nous dirons que $\phi$ est \emph{de dimension $d$} (resp. \emph{purement de dimension $d$}) en un point $y$ de $Y$ si la fibre $Y_{\phi(y)}$ est de dimension $d$ (resp. purement de dimension $d$) en $y$. 

\begin{theo}\label{theo-dimrel-s1}
Soit $\phi \colon Y\to X$ un morphisme plat entre espaces $k$-analytiques dont les fibres satisfont la propriété $(S_1)$ de Serre. 
Soit $d$ un entier. L'ensemble des points $y$ de $Y$
en lesquels $\phi$ est purement de dimension $d$ est un ouvert de Zariski de $Y$. 
 \end{theo}

 \begin{proof}
Soit $\Omega$ l'ensemble des points de $Y$ en lesquels $\phi$ est purement de dimension $d$. Pour montrer que $\Omega$ est un ouvert de Zariski on peut raisonner G-localement sur $X$ et $Y$, et partant les supposer tous deux affinoïdes.

\subsubsection{Réduction au cas strict}
Soit $r$ un polyrayon $k$-libre tel que $\abs{k_r\gpm}\neq \{1\}$ et tels que $X_r$ et $Y_r$ soient strictement $r$-affinoïdes (où le $r$ en indice désigne l'extension des scalaires de $k$ à $k_r$) ; l'image réciproque $\Omega_r$ de $\Omega$ sur $Y_r$ est exactement l'ensemble des points de $Y_r$ en lesquels
$\phi_r$ est purement de dimension $d$. Soit $\mathfrak s\colon Y\to Y_r$ la section de Shilov de $Y_r\to Y$
(\emph{cf.} \cite{ducros2018}, 1.2.16) ; on a $\Omega=\mathfrak s^{-1}(\Omega_r)$ et il suffit par conséquent
de montrer que $\Omega_r$ est un ouvert de Zariski de $Y_r$ (il résulte en effet de la
description de $\mathfrak s$ par
des formules explicites
qu'elle est continue pour les topologies de Zariski de ses source et but). On peut donc supposer que la valeur absolue de $k$ n'est pas triviale et que $X$ et $Y$ sont strictement $k$-affinoïdes.

\subsubsection{}
On fixe $y$ appartenant à $\Omega$. 
Nous allons montrer qu'il existe 
un ouvert de Zariski de $Y$ contenu dans $\Omega$
et contenant $y$, ce qui permettra de conclure.

\subsubsection{}
Soit  $\delta \colon Y\to Y\times_X Y$
l'immersion diagonale, et soit $\Omega'$ l'ensemble des points de $Y\times_X Y$ en lesquels la seconde projection est
purement de dimension $d$. On a alors $\Omega=\delta^{-1}(\Omega')$, 
et il suffit donc pour conclure de montrer qu'il existe un ouvert de Zariski de $Y\times_X Y$ contenu dans $\Omega'$ 
et contenant $\delta(y)$. Quitte à remplacer $Y\to X$ par $Y\times_X Y\to Y$ et $y$ par $\delta(y)$, on peut donc supposer que 
$Y\to X$ possède une section $\sigma$ et que
$y=\sigma(x)$ pour un certain $x\in X$.

\subsubsection{}\label{sss-dims1-compo}
Soit $(X_i)_{i\in I}$ la famille des composantes irréductibles de $X$, et soit $J$ l'ensemble des $i$ tel que $x\in X_i$ ; munissons chacune des $X_i$ de sa structure réduite et posons $Y_i=Y\times_X X_i$. Supposons avoir montré que pour tout $i\in J$ il existe un fermé de Zariski $F_i$ de $Y_i$
ne contenant pas $y$ tel que $Y_i\to X_i$ soit purement de dimension $d$ en tout point de $Y_i\setminus F_i$, et posons $F=\bigcup_{i\in J}F_i\cup\bigcup_{i\in I\setminus J}Y_i$. Alors $F$ est un fermé de Zariski de $Y$ et le morphisme $\phi$ est 
purement de dimension relative $d$ en tout point de l'ouvert $Y\setminus F$, lequel contient $y$.

Il suffit donc de démontrer l'assertion pour chaque composante irréductible réduite de $X$. 
Cette remarque permet de se ramener au cas où $X$ est intègre, c'est-à-dire irréductible et réduit ; on note $e$ sa dimension. 

\subsubsection{}\label{sss-dims1-dimcomp}
Puisque le morphisme
$\phi$ est plat, on a pour tout
point $z$ de $Y$ l'égalité
$\dim_z Y=\dim_{\phi(z)}X+\dim_z Y_{\phi(z)}$
(\cite{ducros2018}, Lemma 4.5.11).
Il en résulte immédiatement que les assertions suivantes sont équivalentes
(ici $X$ est irréductible de dimension $e$, mais on peut remarquer que ces équivalences
seraient encore avérées si l'on supposait simplement
que $X$
est purement de dimension $e$) : 

\begin{enumerate}[i]

\item il existe un ouvert de Zariski de $Y$ contenu dans $\Omega$ et contenant $y$ ; 
\item toute composante irréductible de $Y$ contenant $y$ est de dimension $d+e$ ; 
\item toutes les composantes irréductibles de $Y$ contenant $y$ ont même dimension. 
\end{enumerate}
Nous avons pour objectif de démontrer (i) ; il revient
donc au même d'établir (ii) ou (iii). Nous allons raisonner par récurrence sur $e$. 

\subsubsection{Le cas où $e=0$}
On a alors $X=\{x\}$ et $\phi^{-1}(x)=Y$ ; dire que
$y\in \Omega$ revient à dire que toute composante irréductible de $Y$ passant par $y$ est 
de dimension $d$, ce qui entraîne immédiatement le résultat voulu puisque la condition (iii) de 
\ref{sss-dims1-dimcomp} est satisfaite ou, plus directement, puisque la dimension de $\phi$ en un point coïncide ici avec
celle de $Y$. 

\subsubsection{On suppose que $e>0$ et que le résultat vaut en dimension $<e$}
On note $\mathscr X$ et $\mathscr Y$ les spectres respectifs de $\mathscr O_X(X)$ et $\mathscr O_Y(Y)$ ; on désigne par $\eta$ l'image de $y$ sur $\mathscr Y$ et par $\xi$ celle de $x$ sur $\mathscr X$. 

\paragraph{Supposons que $\xi$ est le point générique de $\mathscr X$}
Si $y$ est situé sur une seule composante irréductible de $Y$, 
la condition (ii) de \ref{sss-dims1-dimcomp} est satisfaite et le théorème est démontré. Supposons maintenant que
$y$ soit situé sur au moins deux composantes irréductibles de $Y$.

Soit $Z$ une
composante irréductible de $Y$ contenant $y$,
soit $U$ l'ouvert de $Z$ constitué des points qui n'appartiennent pas à d'autres composantes irréductibles de 
$Y$, et soit $Z'$ le complémentaire de $U$ dans $Z$ ; remarquons que $U$ est un ouvert dense de $Z$,
et que $y\in Z'$ par hypothèse. 

Soit $T$ (resp. $T'$) l'intersection 
$\sigma(X)\cap Z$ (resp. $\sigma(X)\cap Z'$) ; on munit $Z, Z', T$ et $T'$ de leurs
structures réduites, et l'on note $\mathscr Z$ le spectre de $\mathscr O_Z(Z)$. Puisque $\xi$ est le point générique du schéma
intègre $\mathscr X$, la flèche $\mathscr Z\to  \mathscr X$ est plate en $\eta$. Comme le point $y$ est situé sur le sous-espace analytique fermé $T$ de $Z$ qui est fini sur $X$, la $\mathscr X$-platitude
de $\mathscr Z$ en
$\eta$ entraîne la $X$-platitude de $Z$ en $y$ en vertu du théorème 8.3.7 de \cite{ducros2018}. Le même argument appliqué à $Z'$ et $T'$ au lieu de $Z$ et $T$ montre que 
$Z'$ est $X$-plat en $y$. 

Comme $y=\sigma(x)$, le point $y$ appartient à l'intérieur relatif de $Y$ sur $X$, et \emph{a fortiori} à celui de $Z$ sur $X$. La $X$-platitude de $Z$ et $Z'$ en $y$ implique alors que $y$ appartient à l'adhérence de la fibre $U_x$ (\cite{ducros2018}, prop. 10.4.2). Or $U$ est un ouvert de $Y$ purement de dimension $\dim Z$ ; par platitude, les fibres de $U\to X$ sont donc
purement de dimension $\dim Z-e$. Puisque $y$ est adhérent à $U_x$, l'une au moins des composantes irréductibles de $\phi^{-1}(x)$ passant par $y$ est de dimension $\dim Z-e$ ; comme $\phi$ est par hypothèse purement de dimension $d$ en $y$, il vient $\dim Z=d+e$ ; comme ceci vaut pour toute composante irréductible $Z$ de $Y$ contenant $y$, l'assertion (ii) de \ref{sss-dims1-dimcomp} est établie, et le théorème est démontré. 

\paragraph{On suppose que $\xi$ n'est pas le point générique de $\mathscr X$}
Soit $\widetilde X$ le normalisé de $X$ et soit $U$ le lieu de normalité de $X$. Posons $\widetilde Y=Y\times_X{\widetilde X}$, et notons
$\widetilde \sigma$
la section de $\widetilde Y\to \widetilde X$ déduite de $\sigma$. 

Soit $V$ (resp. $\widetilde V$) l'image réciproque sur $Y$ (resp. $\widetilde Y$) de $U$. 
Comme $X$ est réduit, $U$ est un ouvert dense de $X$ ; il en résulte par platitude de $Y\to X$ que $V$ est un ouvert dense de $Y$
(\cite{ducros2009}, lemme 5.7). 

Soit $Z$ une composante irréductible de $Y$ contenant $y$. L'intersection $Z\cap V$ est un ouvert dense de $Z$, qui est
irréductible (\emph{cf}. \cite{ducros2018}, Prop. 2.7.16) ; 
c'est une composante irréductible de $V\simeq \widetilde V$, et son adhérence $\widetilde Z$ dans $\widetilde Y$ est donc une composante irréductible de $\widetilde Y$, de même dimension que $Z$ et se surjectant sur celle-ci ; en particulier, $\widetilde Z$ contient un antécédent $y'$ de $y$.

Par construction, $\widetilde Y\to \widetilde X$ est purement de dimension $d$ en $y'$, le point $y'$
appartient à $\widetilde \sigma(\widetilde X)$, et son image sur le schéma intègre
$\spec \mathscr O_{\widetilde X}(\widetilde X)$ n'est pas le point générique.
Il suffit pour conclure de montrer que toutes les composantes irréductibles de $\widetilde Y$ contenant $y'$ ont même dimension ; 
en effet si c'est le cas celle-ci est nécessairement égale à $d+e$
(c'est l'équivalence (ii)$\iff$(iii) de \ref{sss-dims1-dimcomp}), ce qui entraîne que $Z$ est de dimension $d+e$ ; puisque cela vaut pour tout $Z$, l'assertion (ii) de  \ref{sss-dims1-dimcomp}
est  établie,  ce qui termine la preuve du théorème. 

On s'est donc ramené au cas  où $X$ est normal, et c'est
désormais l'assertion (iii) de \ref{sss-dims1-dimcomp}
que l'on cherche à montrer.

Puisque $\xi$ n'est pas le point générique de $\mathscr X$, il existe une fonction non nulle
$f$ sur $\mathscr X$ s'annulant en
$\xi$ ; soit $\mathscr T$ le sous-schéma fermé de $\mathscr X$ 
défini par l'équation $f=0$ ; on déduit du \emph{Hauptidealsatz}
et du lemme 2.7.13 (1)
de \cite{ducros2018}
que $T:=\mathscr T\an$ est
purement de dimension $e-1$.
Par notre hypothèse de récurrence, 
$y$ possède pour toute composante irréductible
$S$ de $T$ un voisinage ouvert de Zariski $V_S$ dans $Y\times_X S$ tel que $V_S\to S$ soit purement de dimension  $d$ en chacun de ses points. 
Il en résulte d'après \ref{sss-dims1-compo} qu'il existe un voisinage ouvert de Zariski $V$ de $y$ dans $Y\times_X
T$ tel que $V\to T$ soit purement de dimension $d$ en chacun de ses points ; en vertu de \ref{sss-dims1-dimcomp},
les composantes irréductibles de $Y\times_X T$ contenant $y$ sont
alors toutes de même dimension (à savoir $d+e-1$). 

Comme la valeur absolue de $k$ n'est pas triviale
et comme les espaces en jeu sont strictement
affinoïdes, toute composante irréductible du schéma $\mathscr Y\times_{\mathscr X}\mathscr T$
est de dimension $d+e-1$
(\cite{ducros2018}, lemma 2.7.6).
Soit $\epsilon$ la dimension de $\overline{\{\eta\}}\an$ ; le schéma $\spec \mathscr O_{\mathscr Y\times_{\mathscr X}\mathscr T,
\eta}$
est alors purement de dimension $d+e-1-\epsilon$
(\emph{op. cit}., Lemma 2.7.13 (1) ; notons qu'on n'a pas besoin ici
que la valeur absolue de $k$ soit non triviale ni que les espaces affinoïdes en jeu soient
stricts). 

En utilisant
à nouveau les lemmes 2.7.3 (1) et 2.7.6 de \cite{ducros2018}, 
on voit qu'il suffit pour prouver l'assertion (ii)
de \ref{sss-dims1-compo}
de démontrer que $\spec \mathscr O_{\mathscr Y,
\eta}$ est équidimensionnel. Or : 

\begin{itemize}[label=$\diamond$]

\item $\spec \mathscr O_{\mathscr Y\times_{\mathscr X}\mathscr T,
\eta}$ est équidimensionnel comme on vient de le voir ; 

\item $\mathscr O_{\mathscr Y\times_{\mathscr X}\mathscr T,\eta}$
est le quotient de $\mathscr O_{\mathscr Y,\eta}$ par $f$, qui n'est pas diviseur de zéro (puisqu'il ne l'est pas dans l'anneau intègre
$\mathscr O_{\mathscr X,x}$ sur lequel $\mathscr O_{\mathscr Y,\eta}$ est plat) ; 

\item l'anneau $\mathscr O_{\mathscr Y,\eta}$ est caténaire puisque $\mathscr Y$ est de type fini sur une algèbre de la forme $k\{T_1/r_1,\ldots, T_n/r_n\}$ qui est régulière (\emph{cf.}
\cite{ducros2009} pour une preuve élémentaire) et partant universellement caténaire. 

\end{itemize}
Il suffit alors désormais pour conclure, grâce à la proposition 5.12.2 de \cite{ega42},
de montrer que  $\mathscr O_{\mathscr Y\times_{\mathscr X}\mathscr T,\eta}$ satisfait $(S_1)$. 

Le schéma $\mathscr X$ est normal, et partant possède la propriété $(S_2)$
de Serre. Son sous-schéma fermé
$\mathscr T$ étant défini localement par une équation régulière, il satisfait $(S_1)$ ; par conséquent, l'espace analytique $T$
satisfait $(S_1)$. 
On en déduit, compte-tenu du fait que $Y\times_X T\to T$ est plat et que ses fibres satisfont
la propriété $(S_1)$ par hypothèse, que $Y\times_X T$ satisfait $(S_1)$
(\cite{ducros2018}, Thm. 11.3.3) ; en conséquence,
l'anneau
local $\mathscr O_{\mathscr Y\times_{\mathscr X}\mathscr T,\eta}$
satisfait $(S_1)$. 
 \end{proof}

\bibliographystyle{smfalpha}
\bibliography{aducros}

\end{document}